\documentclass[12pt,reqno]{amsart}
\usepackage{amssymb}
\usepackage{amscd}
\usepackage{amsfonts}
\usepackage{setspace}
\usepackage{graphicx}
\usepackage{tikz}
\usetikzlibrary{calc}
\usepackage{epic}
\usepackage{version}

\numberwithin{equation}{section} \theoremstyle{plain}
\newtheorem{theorem}[subsection]{Theorem}
\newtheorem{proposition}[subsection]{Proposition}
\newtheorem{lemma}[subsection]{Lemma}
\newtheorem{corollary}[subsection]{Corollary}
\newtheorem{definition}[subsection]{Definition}

\newtheorem*{sylvester-gallai-rpt}{Theorem \ref{syl-gal}}
\newtheorem*{main-struct-repeat}{Theorem \ref{main-struct}}
\newtheorem*{main-structure-theorem-rpt}{Theorem \ref{main-structure-theorem}}
\newtheorem*{almost-group-rpt}{Proposition \ref{almost-group}}
\newtheorem*{precise-dm-rpt}{Theorem \ref{precise-dm}}

\renewcommand{\leq}{\leqslant}
\renewcommand{\geq}{\geqslant}
\newsavebox{\proofbox}
\savebox{\proofbox}{\begin{picture}(7,7)  \put(0,0){\framebox(7,7){}}\end{picture}}

\newcommand{\md}[1]{\ensuremath{(\operatorname{mod}\, #1)}}

\newcommand\Z{\mathbb{Z}}
\newcommand\R{\mathbb{R}}

\newcommand\C{\mathbb{C}}
\newcommand\N{\mathbb{N}}

\renewcommand\P{\mathbb{P}}

\newcommand\eps{\varepsilon}

\newcommand\Sym{\operatorname{Sym}}

\newcommand\length{\operatorname{length}}

\def\endproof{\hfill{\usebox{\proofbox}}\vspace{11pt}}

\begin{document}

\title{On sets defining few ordinary lines}

\author{Ben Green}
\address{Centre for Mathematical Sciences\\
Wilberforce Road\\
Cambridge CB3 0WA\\
England }
\email{b.j.green@dpmms.cam.ac.uk}

\author{Terence Tao}
\address{Department of Mathematics, UCLA\\
405 Hilgard Ave\\
Los Angeles CA 90095\\
USA}
\email{tao@math.ucla.edu}


\begin{abstract}  
Let $P$ be a set of $n$ points in the plane, not all on a line. We show that if $n$ is large then there are at least $n/2$ \emph{ordinary} lines, that is to say lines passing through exactly two points of $P$. This confirms, for large $n$,  a conjecture of Dirac and Motzkin. In fact we describe the exact extremisers for this problem, as well as all sets having fewer than $n - C$ ordinary lines for some absolute constant $C$. We also solve, for large $n$, the ``orchard-planting problem'', which asks for the maximum number of lines through exactly 3 points of $P$. Underlying these results is a structure theorem which states that if $P$ has at most $Kn$ ordinary lines then all but $O(K)$ points of $P$ lie on a cubic curve, if $n$ is sufficiently large depending on $K$.
\end{abstract}

\maketitle
\tableofcontents

\section{Introduction}

The \emph{Sylvester--Gallai theorem} is a well-known theorem in combinatorial geometry. It was proven by Gallai \cite{gallai} in response to a question of Sylvester \cite{sylvester-q} from forty years earlier.

\begin{figure}
\includegraphics[scale=1]{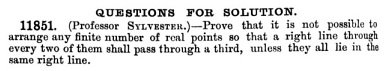} 
\caption{Sylvester's question \cite{sylvester-q}.}\label{sylvester-qq}
\end{figure}

\begin{theorem}[Sylvester-Gallai theorem]\label{syl-gal}
Suppose that $P$ is a finite set of points in the plane, not all on one line. Then there is an \emph{ordinary line} spanned by $P$, that is to say a line in $P$ containing exactly two points.
\end{theorem}

Several different proofs of this now appear in the literature. We will be particularly interested in a proof due to Melchior \cite{melchior} based on projective duality and Euler's formula, which we will recall in Section \ref{melchior-sec}. It is natural to wonder how many ordinary lines there are in a set of $P$ points, not all on a line, when the cardinality $|P|$ of $P$ is equal to $n$. Melchior's argument in fact shows that there are at least three ordinary lines, but considerably more is known. Motzkin \cite{motzkin} was the first to obtain a lower bound (of order $n^{1/2}$) tending to infinity with $n$. Kelly and Moser \cite{kelly-moser} proved that there are at least $3n/7$ ordinary lines, and Csima and Sawyer \cite{csima-sawyer} improved this to $6n/13$ when $n > 7$. Their work used ideas from the thesis of Hansen \cite{hansen}, which purported to prove the $n/2$ lower bound but was apparently flawed. An illuminating discussion of this point may be found in the MathSciNet review of \cite{csima-sawyer}. There are several nice surveys on this and related problems; see \cite{borwein-moser}, \cite[Chapter 17]{erdos-handbook}, \cite{nilakantan} or \cite{pach-sharir}. 

One of our main objectives in this paper is to clarify this issue for large $n$. The following theorem resolves, for large $n$, a long-standing conjecture which has been known as the \emph{Dirac-Motzkin conjecture}. Apparently neither author formally conjectures this in print, though Dirac \cite{dirac} twice states that its truth is ``likely''. Motzkin \cite{motzkin} does not seem to mention it at all.

\begin{theorem}[Dirac-Motzkin conjecture]\label{mainthm}
Suppose that $P$ is a finite set of $n$ points in the plane, not all on one line. Suppose that $n \geq n_0$ for a sufficiently large absolute constant $n_0$. Then $P$ spans at least $n/2$ ordinary lines.  
\end{theorem}

We will in fact establish a more precise result obtaining the exact minimum for all $n \geq n_0$ as well as a classification of the extremal examples. One rather curious feature of this more precise result is that if $n$ is odd there must be at least $3\lfloor n/4\rfloor$ ordinary lines. See Theorems \ref{dirac-motzkin-1} and \ref{dirac-motzkin-2} below for more details.  When $n$ is even, one can attain $n/2$ ordinary lines by combining $n/2$ equally spaced points on a circle with $n/2$ points at infinity; see Proposition \ref{boroczky-examples} below.

For small values of $n$, there are exceptional configurations with fewer than $n/2$ ordinary lines. Kelly and Moser \cite{kelly-moser} observe that a triangle together with the midpoints of its sides and its centroid has $n = 7$ and just $3$ ordinary lines. Crowe and McKee \cite{crowe-mckee} provide a more complicated configuration with $n = 13$ and $6$ ordinary lines. It is possible that Theorem \ref{mainthm} remains true for \emph{all} $n$ with the exception of these two examples (or equivalently, one could take $n_0$ as low as $14$). Unfortunately our method does not give a good bound for $n_0$; we could in principle compute such a bound, but it would be of double exponential type and, we think, not worth the effort.

Our methods also apply (in fact with considerably less effort) to resolve sufficiently large instances of a slightly less well-known (but considerably older) problem referred to in the literature as the \emph{orchard problem}. This was first formally posed by Sylvester \cite{sylvester-orchard} in 1868 (Figure \ref{syl-q}), though the 1821 book of Jackson \cite{jackson} has a whole section containing puzzles of a similar flavour, posed more eloquently than any result in this paper (Figure \ref{jackson-qq}).

\begin{figure}
\includegraphics[scale=0.5]{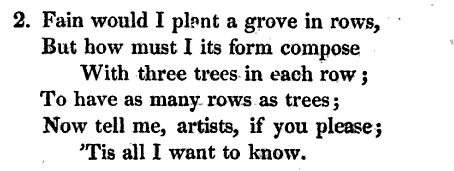} 
\caption{Jackson's question \cite{jackson}.}\label{jackson-qq}
\end{figure}

\begin{figure}
\includegraphics[scale=1]{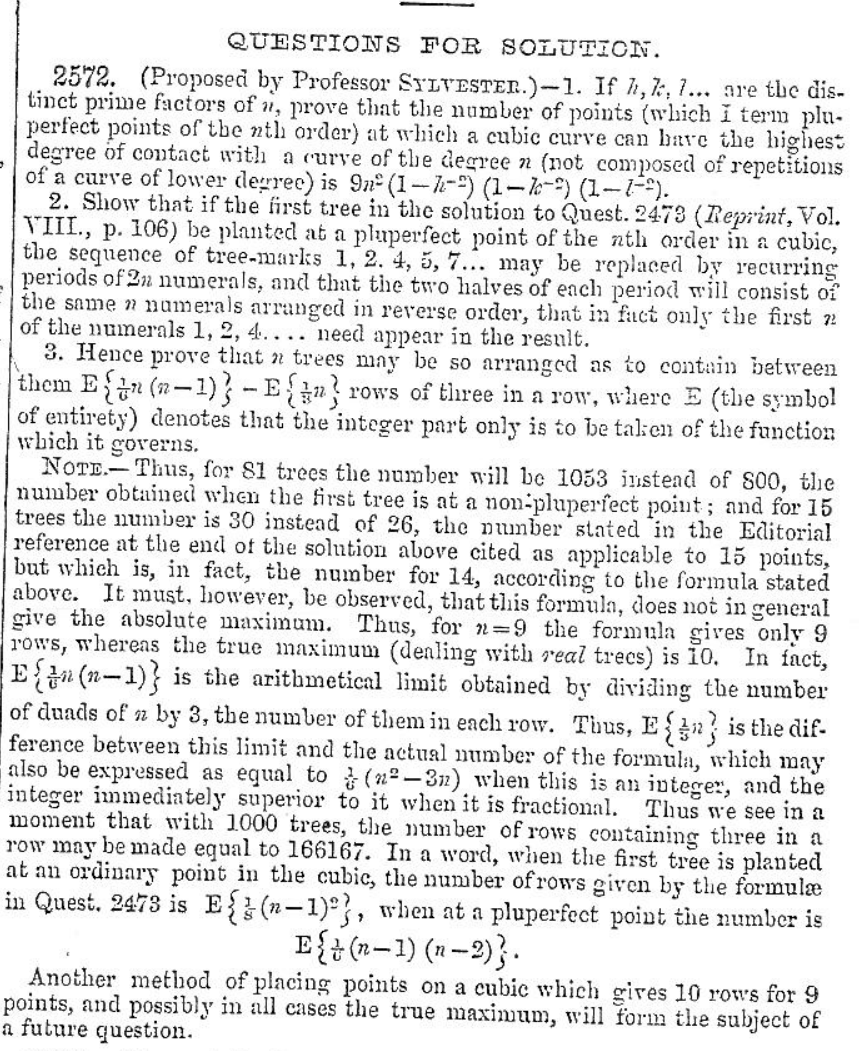} 
\caption{Sylvester's question \cite{sylvester-orchard}.}\label{syl-q}
\end{figure}

\begin{theorem}[Orchard problem]\label{mainthm-orchard}
Suppose that $P$ is a finite set of $n$ points in the plane. Suppose that $n \geq n_0$ for some sufficiently large absolute constant $n_0$. Then there are no more than $\lfloor n (n-3)/6\rfloor + 1$ lines that are \emph{$3$-rich}, that is they contain precisely $3$ points of $P$.  \textup{(}Here and in the sequel, $\lfloor x \rfloor$ denotes the integer part of $x$.\textup{)}
\end{theorem}

This theorem is tight for large $n$, as noted by Sylvester \cite{sylvester-orchard}, and also subsequently by Burr, Gr\"unbaum and Sloane \cite{bgs}, who discuss this problem extensively. We will give these examples, which are based on irreducible cubic curves, in Proposition \ref{bgs-prop} below. In fact these are the \emph{only} examples where equality occurs for large $n$: see the remarks at the very end of \S \ref{orchard}. Again, there are counterexamples for small $n$. In particular, the example of a triangle, the midpoints of its sides and its centroid has $n = 7$ but $6$ lines containing precisely three points of $P$; by contrast, the bound of Theorem \ref{mainthm-orchard} is $5$ in this case. 

As observed in \cite{bgs}, lower bounds for the number $N_2$ of ordinary lines can be converted into upper bounds for the number $N_3$ of $3$-rich lines thanks to the obvious double-counting identity $\sum_{k=2}^n \binom{k}{2} N_k = \binom{n}{2}$ (with $N_k$ denoting the number of $k$-rich lines).  In particular, previously known lower bounds on the Dirac-Motzkin conjecture can be used to deduce upper bounds on the orchard problem.  However, one cannot deduce Theorem \ref{mainthm-orchard} in this fashion from Theorem \ref{mainthm}; this is related to the fact that the extremal examples showing the sharpness of the two theorems are quite different, as we shall see in Section \ref{examples-sec} below.

Underlying the proof of both of these results are structure theorems for sets with few ordinary lines, which are perhaps of independent interest. The most basic such result is the following.  We use the asymptotic notation $X=O(Y)$ or $X \ll Y$ to denote the bound $|X| \leq C Y$ for some absolute constant $C$.

\begin{theorem}[Weak structure theorem]\label{weak-struct} Suppose that $P$ is a finite set of $n$ points in the plane. Suppose that $P$ spans at most $Kn$ ordinary lines for some $K \geq 1$.  Suppose also that $n \geq \exp\exp(C K^C)$ for some sufficiently large absolute constant $C$.  Then all but at most $O(K^{O(1)})$ points of $P$ lie on an algebraic curve $\gamma$ of degree at most $3$. 
\end{theorem}

In fact we establish a slightly more precise statement, see Proposition \ref{new-struct} below. Note that we do not require the algebraic curve $\gamma$ to be irreducible; thus $\gamma$ could be an irreducible cubic, the union of a conic and a line, or the union of three lines.  As we shall see in later sections, cubic curves arise naturally in the study of point-line configurations with few ordinary lines, in large part due to the well-known abelian group structure (or pseudo-group structure) defined by the collinearity relation on such curves (or equivalently, by Chasles's version of the Cayley-Bacharach theorem, see Proposition \ref{chas}). The lower bound $n \geq \exp\exp(CK^C)$ is present for rather artificial reasons, and can likely be improved substantially.
\vspace{11pt}

\emph{Projective geometry.} Much of the paper is best phrased in the language of projective geometry. We recall for the convenience of the reader the notion of the projective plane $\R\P^2$ as $(\R^3 \setminus \{0\})/\sim$, where $(x,y,z) \sim (x', y', z')$ if and only if there is some $\lambda \neq 0$ such that $x' = \lambda x$, $y' = \lambda y$ and $z' = \lambda z$. We denote points of $\R\P^2$ with square brackets, thus $[x,y,z]$ is the equivalence class of $(x,y,z)$ under $\sim$. We have the embedding $\R^2 \hookrightarrow \R\P^2$ given by $(x,y) \mapsto [x, y, 1]$; in fact $\R\P^2$ may be thought of as $\R^2$ together with the \emph{line at infinity} consisting of points $[x,y,0]$ (modulo the equivalence relation $\sim$). For the point-line incidence problems considered in this paper, the projective and affine formulations are equivalent. Indeed given a finite set of points $P$ in $\R\P^2$, we may apply a generic projective transformation so as to move all points of $P$ to the affine plane $\R^2$ if desired, without affecting the number of ordinary lines or $3$-rich lines. This is illustrated in the figures in Section \ref{examples-sec}. \vspace{11pt}

For our purposes, there are two main advantages of working in projective space instead of affine space.  The first is to allow the use of projective transformations to normalise one's geometric configurations, for instance by moving a line to the line at infinity, or transforming a non-singular irreducible cubic curve into an elliptic curve in Weierstrass normal form.  The other main advantage is the ability to utilise \emph{point-line duality}. Given a point $p = [a,b,c]$, one may associate the line $p^* := \{ [x,y,z] : ax + by + cz = 0\}$, and conversely given a projective line $\ell = \{[x,y,z] : ax + by + cz = 0\}$ one may associate the point $\ell^* = [a,b,c]$. It is clear that $p \in \ell$ if and only if $\ell^* \in p^*$. 
Working in the dual can provide us with information that is difficult to access otherwise. We shall see this twice: once in Sections \ref{melchior-sec} and \ref{dual-cubic}, when we apply Euler's formula in the dual setting following an argument of Melchior \cite{melchior}, and then again in Section \ref{somewhat-collinear} where we will employ a convexity argument, due to Luke Alexander Betts,  in the dual setting.

Next, we give a structure theorem which is more precise than Theorem \ref{weak-struct}.

\begin{theorem}[Full structure theorem]\label{main-structure-theorem}
Suppose that $P$ is a finite set of $n$ points in the projective plane $\R\P^2$. Let $K > 0$ be a real parameter. Suppose that $P$ spans at most $Kn$ ordinary lines. Suppose also that $n \geq \exp\exp(CK^C)$ for some sufficiently large absolute constant $C$.  Then, after applying a projective transformation if necessary,  $P$ differs by at most $O(K)$ points \textup{(}which can be added or deleted\textup{)} from an example of one of the following three types:
\begin{enumerate}
\item $n-O(K)$ points on a line;
\item The set 
\begin{align}\nonumber
X_{2m} :=  \{ [ \cos & \frac{2\pi j}{m}, \sin \frac{2\pi j}{m}, 1]: 0 \leq j < m \} \\ & \cup \{ [-\sin \frac{\pi j}{m}, \cos \frac{\pi j}{m},0], 0 \leq j < m \}\label{x2m}
\end{align}
consisting of $m$ points on the unit circle and $m$ points on the line at infinity, for some $m = \frac{n}{2}+O(K)$;
\item A coset $H \oplus g$, $3g \in H$, of a finite subgroup $H$ of the non-singular real points on an irreducible cubic curve, with $H$ having cardinality $n+O(K)$ \textup{(}the group law $\oplus$ on such curves is reviewed in Section \ref{examples-sec} below\textup{)}.
\end{enumerate}
Conversely, every set of this type has at most $O(K n)$ ordinary lines.
\end{theorem}

We have the following consequence, which can handle slowly growing values of $K$.

\begin{corollary}\label{k-to-infinity}
Suppose that $P$ is a finite set of $n$ points in the projective plane $\R\P^2$. Suppose that $P$ spans at most $n(\log \log n)^c$ ordinary lines for some sufficiently small constant $c>0$. Then, after applying a projective transformation, $P$ differs by at most $o(n)$ points from one of the examples \textup{(i)}, \textup{(ii)}, \textup{(iii)} detailed in Theorem \ref{main-structure-theorem} above. In particular we may add/remove $o(n)$ points to/from $P$ to get a set with at most $n + O(1)$ ordinary lines.  
\end{corollary}

Here, of course, $o(n)$ denotes a quantity which, after dividing by $n$, tends to zero as $n$ goes to infinity.\vspace{11pt}

\emph{Remark.} This corollary may, for all we know, be true with a much weaker assumption, perhaps even that $P$ spans $o(n^2)$ ordinary lines.  Very likely, if one could weaken the hypothesis $n \geq \exp\exp(CK^C)$ in Theorem \ref{weak-struct} then one could do so also in Theorem \ref{main-structure-theorem}.
\vspace{11pt}

\emph{Proof methods.} As mentioned previously, the starting point\footnote{One defect of this approach is that it breaks down totally in the complex case, and so we have nothing new to say here about ordinary lines or $3$-rich lines for configurations of complex lines in $\C\P^2$.} of our arguments will be Melchior's proof \cite{melchior} of the Sylvester-Gallai theorem using duality and the Euler formula $V-E+F=1$ for polygonal decompositions of the projective plane $\R\P^2$.  Melchior's argument uses at one point the obvious fact that all polygons have at least three sides to obtain an inequality implying the existence of ordinary lines.  The same argument also shows that if a point set $P$ spans very few ordinary lines, then almost all of the polygons in the dual configuration $\Gamma_P$ (cut out by the dual lines $p^*$ for $p \in P$) must in fact have \emph{exactly} three sides.  Because of this, it is possible to show in this case that the dual configuration contains large regions which have the combinatorial structure of a regular triangular grid.

The next key observation is that inside any triangular grid of non-trivial size, one can find ``hexagonal'' configurations of lines and points (see Figure \ref{hex}) which are dual to the configuration of lines and points arising in Chasles's version of the Cayley-Bacharach theorem (Proposition \ref{chas} below).  From this observation and some elementary combinatorial arguments, one can start placing large subsets of $P$ on a single cubic curve.  For instance, in Proposition \ref{basic-cubic-covering} we will be able to establish a ``cheap structure theorem'' asserting that a set of $n$ points with fewer than $Kn$ ordinary lines can be covered by no more than $500K$ cubic curves.  This observation turns out not to be new -- a closely related technique is used in a paper of Carnicer and God\'es \cite{carnicer-godes} concerning \emph{generalised principal lattices}, which arise in interpolation theory.  

In principle, this cheap structure theorem already reduces the underlying geometry from a two-dimensional one (the projective plane $\R\P^2$) to a one-dimensional one (the union of a number of cubic curves).  Unfortunately, the collinearity relation between distinct cubic curves is too complicated to handle directly.  Because of this, we must refine the previous combinatorial analysis to strengthen the structural control on a point set $P$ with few ordinary lines.  By studying the lines connecting a typical point $p$ in $P$ with all the other points in $P$ using the triangular grid analysis, one can obtain a more complicated partition of $P$ into cubic curves passing through $p$. A detailed statement may be found in Lemma \ref{cubic-covering-tech}. Comparing such a partition with the reference partition coming from the cheap structure theorem, one can obtain Proposition \ref{intermediate}, a structure theorem of intermediate strength. Roughly speaking this result asserts that most of the points in $P$ lie on a \emph{single} irreducible cubic curve, on a union of an irreducible conic and a bounded number of lines (with the points shared almost evenly between the conic and the lines), or on the union of a bounded number of lines only.

The next stage is to cut the number of lines involved down to one.  The key proposition\footnote{Unfortunately, our reduction to this key proposition is somewhat expensive with regards to the quantitative bounds, and is responsible for the double exponential lower bound required on $n$.} here is Proposition \ref{many-rich}. It asserts, roughly speaking, that a set of $n$ points on two or more lines, which contains $\gg n$ points on each line, must generate $\gg n^2$ ordinary lines.  This is a statement that fails in finite field geometries, and must use at some point the torsion-free nature of the real line $\R$ (see Section \ref{examples-sec} for more discussion of this point, particularly with regards to the ``near-counterexamples'' \eqref{p1} and \eqref{p2}).  There are two key cases of the proposition which need to be established.  The first is when the lines involved are all concurrent or, after a projective transformation, all parallel. In this case we use an argument of Betts, Proposition \ref{betts-case}, involving projective duality and convexity. The use of convexity here is where the torsion-free nature of $\R$ is implicitly used.  In the case when the lines are not concurrent, we instead rely on Menelaus's theorem to introduce various ratios of lengths, and then exploit a sum-product estimate of Elekes, Nathanson, and Ruzsa \cite{elekes-nathanson-ruzsa}. This latter result, stated in Proposition \ref{enr-prop-piecewise}, also implicitly exploits the torsion-free nature of $\R$.

The result of the above analysis is a yet stronger structure theorem for sets $P$ with few ordinary lines: $P$ is mostly placed in either an irreducible cubic curve, the union of an irreducible conic and a line, or on a single line. A detailed statement may be found in Proposition \ref{new-struct}.  The latter case, in which almost all points lie on a line, is easily studied. To deal with the other two cases one uses the abelian group structure on irreducible cubics, as well as the analogous pseudo-group structure on the union of a conic and a line.  The information that $P$ contains few ordinary lines can then be converted to an additive combinatorics property on finite subsets of an abelian group. Fairly standard tools from additive combinatorics then show that $P$ is almost a finite subgroup of that abelian group.  This allows us to rule out ``essentially torsion-free'' situations, such as that provided by singular cubic curves (except for the acnodal singular cubic curve), and eventually leads to the full structure theorem in Theorem \ref{main-structure-theorem}.

To solve the Dirac-Motzkin conjecture and the orchard problem for large $n$, we observe that potential counterexamples $P$ to either conjecture will have few ordinary lines and hence can be described by Theorem \ref{main-structure-theorem}.  This quickly implies that $P$ is close, up to projective transformation, to one of the known extremisers coming from roots of unity or from subgroups of elliptic curves, with a small number of additional points added or removed.  The remaining task is to compute the effect that these added/removed points have on the number of ordinary lines or $3$-rich lines in $P$.  Here, to get the strongest results, we will need a slight variant of a result of Poonen and Rubinstein \cite{poonen-rubinstein} in order to control the number of times a point may be concurrent with two roots of unity. See Proposition \ref{poonen-rubinstein-prop} for details.  

Of the two problems, the orchard problem turns out to be somewhat easier, and can be in fact established using only the intermediate structure theorem in Theorem \ref{intermediate} rather than the more difficult structure theorem in Theorem \ref{main-structure-theorem}.\vspace{11pt}

\emph{Acknowledgements.} The authors are greatly indebted to Luke Alexander Betts for indicating to us the proof of Proposition \ref{betts-case}. We are also grateful to Noga Alon, Boris Bukh, Frando Mariacci, Bjorn Poonen and Jozsef Solymosi for helpful comments and conversations, and to Frank de Zeeuw for pointing out the case of the acnodal singular cubic curve, which was not treated properly in the first draft of this paper. The first author thanks UCLA for providing entertaining working conditions for some of the period during which this work was completed, and the second author similarly thanks the Newton Institute for hospitality when another portion of this work was completed.  Finally, we would like to thank the anonymous referee for an extremely thorough reading of the paper. 

The first author was supported by an ERC Starting Investigators Grant. The second author was partially supported by a Simons Investigator award from the Simons Foundation and by NSF grant DMS-0649473.

\section{The key examples}\label{examples-sec}

The aim of this section is to describe the various key examples of sets with few ordinary lines or many $3$-rich lines. In particular we describe  the sets $X_{2m}$ in \eqref{x2m} and the Sylvester examples appearing in the various cases of the main structure theorem, Theorem \ref{main-structure-theorem}.  We will also mention some important ``near-counterexamples'' which do not actually exist as finite counterexamples to the structure theorem, but nevertheless are close enough to genuine counterexamples that some attention must be given in the analysis to explicitly exclude variants of these examples from the list of possible configurations.  All of these examples are connected to the group (or pseudo-group) structure on a cubic curve, or equivalently to Chasles's version of the Cayley-Bacharach theorem as described in Proposition \ref{chas} below. The main variation in the examples comes from the nature of the cubic curve being considered, which may or may not be irreducible and/or nonsingular.\vspace{11pt}

\emph{B\"or\"oczky and near-B\"or\"oczky examples.}  We begin with the sets $X_{2m}$ from \eqref{x2m}, together with some slight perturbations of these sets described by B\"or\"oczky (as cited in \cite{crowe-mckee}). These sets, it turns out, provide the examples of non-collinear sets of $n$ points with the fewest number of ordinary lines, at least for $n$ large. 

\begin{proposition}[B\"or\"oczky examples] \label{boroczky-examples}  Let $m \geq 3$ be an integer.  Then we have the following.
\begin{enumerate}
\item The set $X_{2m}$ contains $2m$ points and spans precisely $m$ ordinary lines. 
\item The set $X_{4m}$ together with the origin $[0,0,1]$ contains $4m + 1$ points and spans precisely $3m$ ordinary lines.
\item The set $X_{4m}$ minus the point $[0,1,0]$ on the line at infinity contains $4m - 1$ points and spans precisely $3m - 3$ ordinary lines.  
\item The set $X_{4m + 2}$ minus any of the $2m + 1$ points on the line at infinity contains $4m + 1$ points and spans $3m$ ordinary lines. 
\end{enumerate}
Thus, if we define a function $f : \N \rightarrow \N$ by setting $f(2m) := m$, $f(4m+1) := 3m$ and $f(4m - 1) := 3m - 3$, then there is an example of a set of $n$ points in $\R \P^2$, not all on a line, spanning $f(n)$ ordinary lines.\end{proposition}

\begin{proof}
This is a rather straightforward check, especially once one has drawn suitable pictures. Whilst the unit circle together with the line at infinity form a pleasant context for calculational work, drawing configurations involving the line at infinity is problematic.  In the four diagrams below, Figures \ref{boro-1-fig}, \ref{boro-2-fig}, \ref{boro-3-fig} and \ref{boro-4-fig}, we have applied a projective transformation to aid visualisation. First of all we applied a rotation about the origin through $\pi/12$, and followed this by the projective map $[x,y,z] \mapsto [-y,x,2z + x]$. The unit circle is then sent to the ellipse whose equation in the affine plane is $4x^2 + 3(y + \frac{1}{3})^2 = \frac{4}{3}$, while the line at infinity is sent to the horizontal line $y = 1$. The origin is mapped to itself, and the point $[0,1,0]$ at infinity now has coordinates $(-\cot (\pi/12),1) \approx (-3.73,1)$. In the pictures, ordinary lines are red and lines with three or more points of $P$ are dotted green.

It is helpful to note that the line joining \[ [ \cos \frac{2\pi j}{m}, \sin \frac{2\pi j}{m}, 1]\] and \[ [ \cos \frac{2\pi j'}{m}, \sin \frac{2\pi j'}{m}, 1]\] passes through the point \[ [-\sin \frac{\pi( j+j')}{m}, \cos \frac{\pi( j+ j')}{m},0]\] on the line at infinity (cf. the proof of Proposition \ref{quasi-group-law}).

For case (i), the ordinary lines are the $m$ tangent lines to the $m^{\operatorname{th}}$ roots of unity. The case $m = 6$ is depicted in Figure \ref{boro-1-fig}.

\begin{figure}[h!]
\includegraphics[scale=0.6]{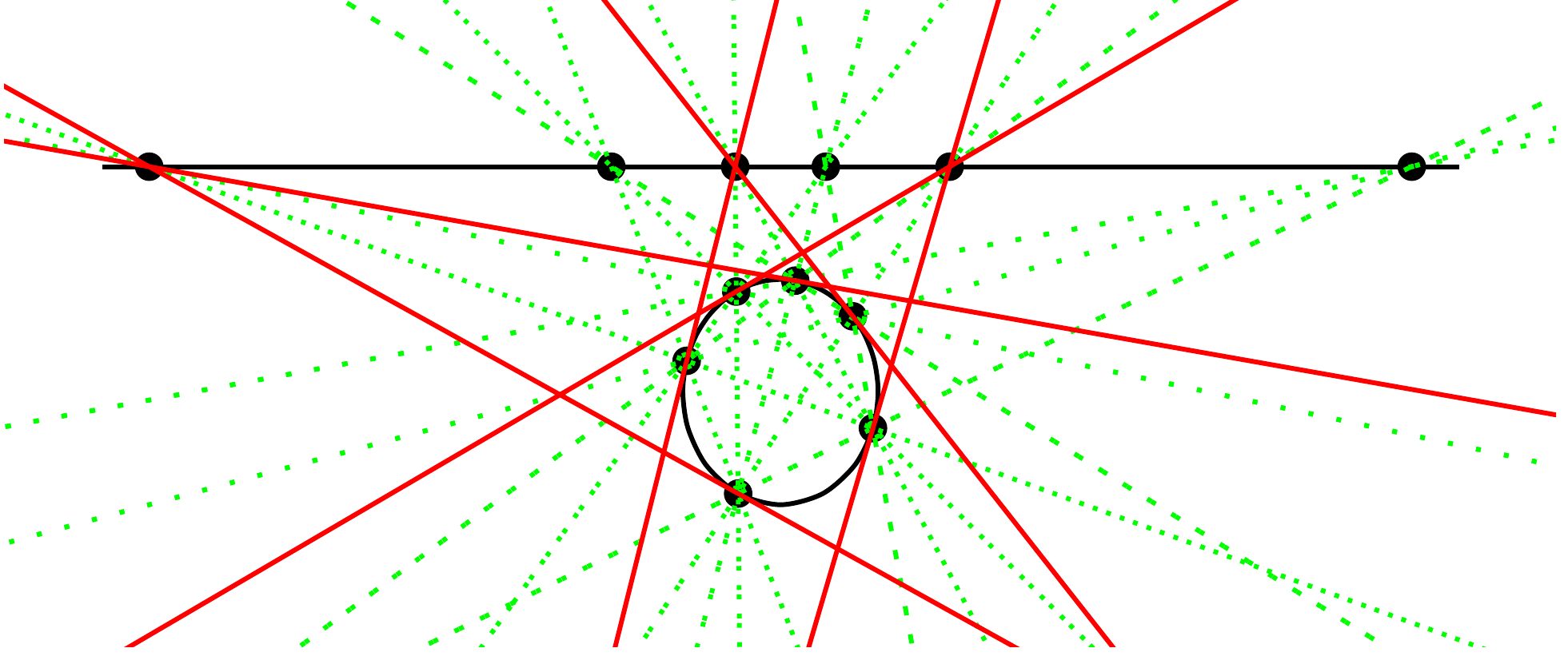}
\caption{The B\"or\"oczky example $X_{12}$, a set with $n = 12$ points and $6$ ordinary lines. The ordinary lines (in red) are just the tangent lines to the $6$th roots of unity on the unit circle.}
\label{boro-1-fig}
\end{figure}

In case (ii), the ordinary lines are the $2m$ tangent lines to the $2m^{\operatorname{th}}$ roots of unity together with the $m$ lines joining the origin $[0,0,1]$ to $[-\sin \frac{\pi j}{2m}, \cos \frac{\pi j}{2m}, 0]$, $j$ even.  The case $m = 6$ is depicted in Figure \ref{boro-2-fig}  .

\begin{figure}[h!]
\includegraphics[scale=0.6]{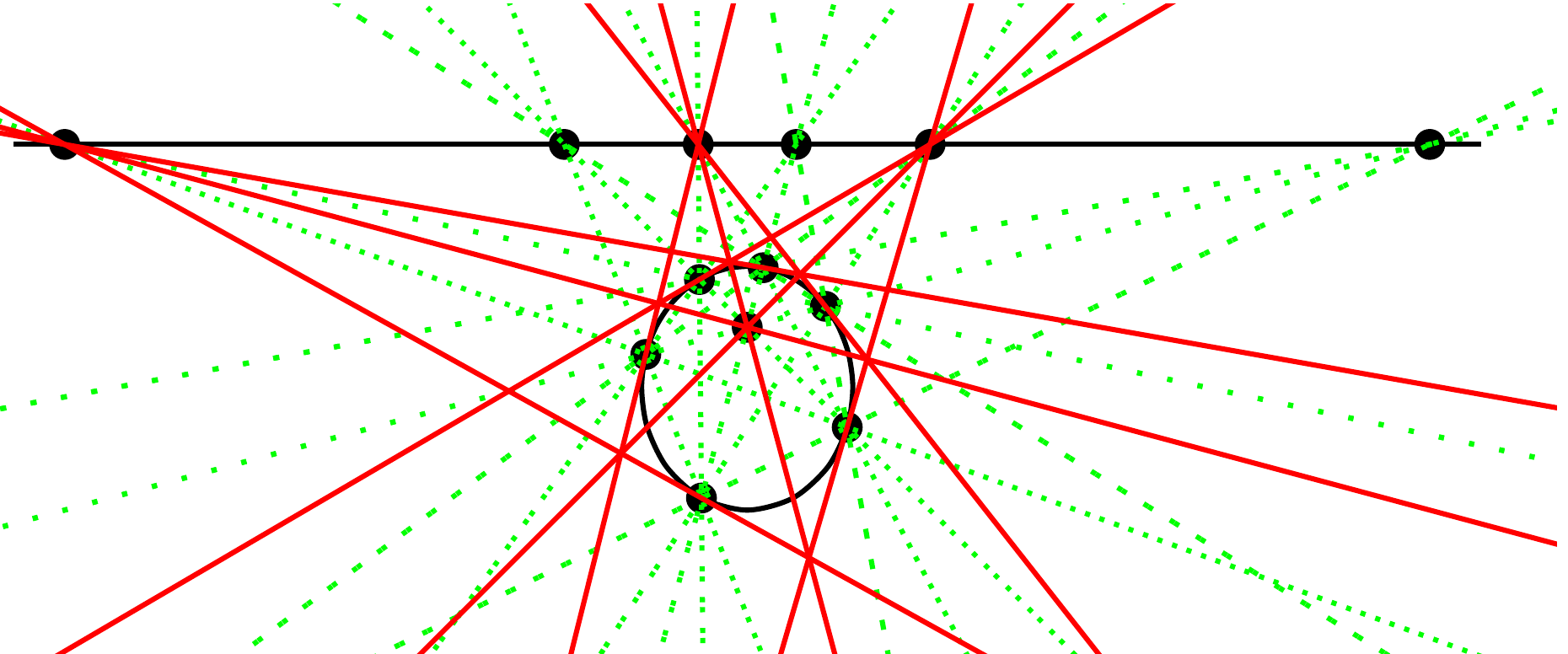}
\caption{The B\"or\"oczky example $X_{12}$ together with the origin $[0,0,1]$, a set with $n = 13$ points and $9$ ordinary lines. The ordinary lines (in red) are just the tangent lines to the $6$th roots of unity on the unit circle, plus 3 extra lines through the origin and 3 of the points on the line at infinity.}
\label{boro-2-fig}
\end{figure}

In case (iii), the ordinary lines are the $2m$ tangent lines to the $2m^{\operatorname{th}}$ root of unity \emph{except} for those at the points $[\pm 1,0,1]$, whose corresponding point at infinity has now been removed. However we do have $m - 1$ new ordinary lines, the vertical lines joining $[\cos \frac{\pi j}{m}, \sin \frac{\pi j}{m},1]$ and $[\cos \frac{\pi (2m - j)}{m}, \sin \frac{\pi(2 j-m)}{m},1]$ for $j = 1,\dots, m-1$. The case $m = 6$ is illustrated in Figure \ref{boro-3-fig}.
\begin{figure}[h!]
\includegraphics[scale=0.6]{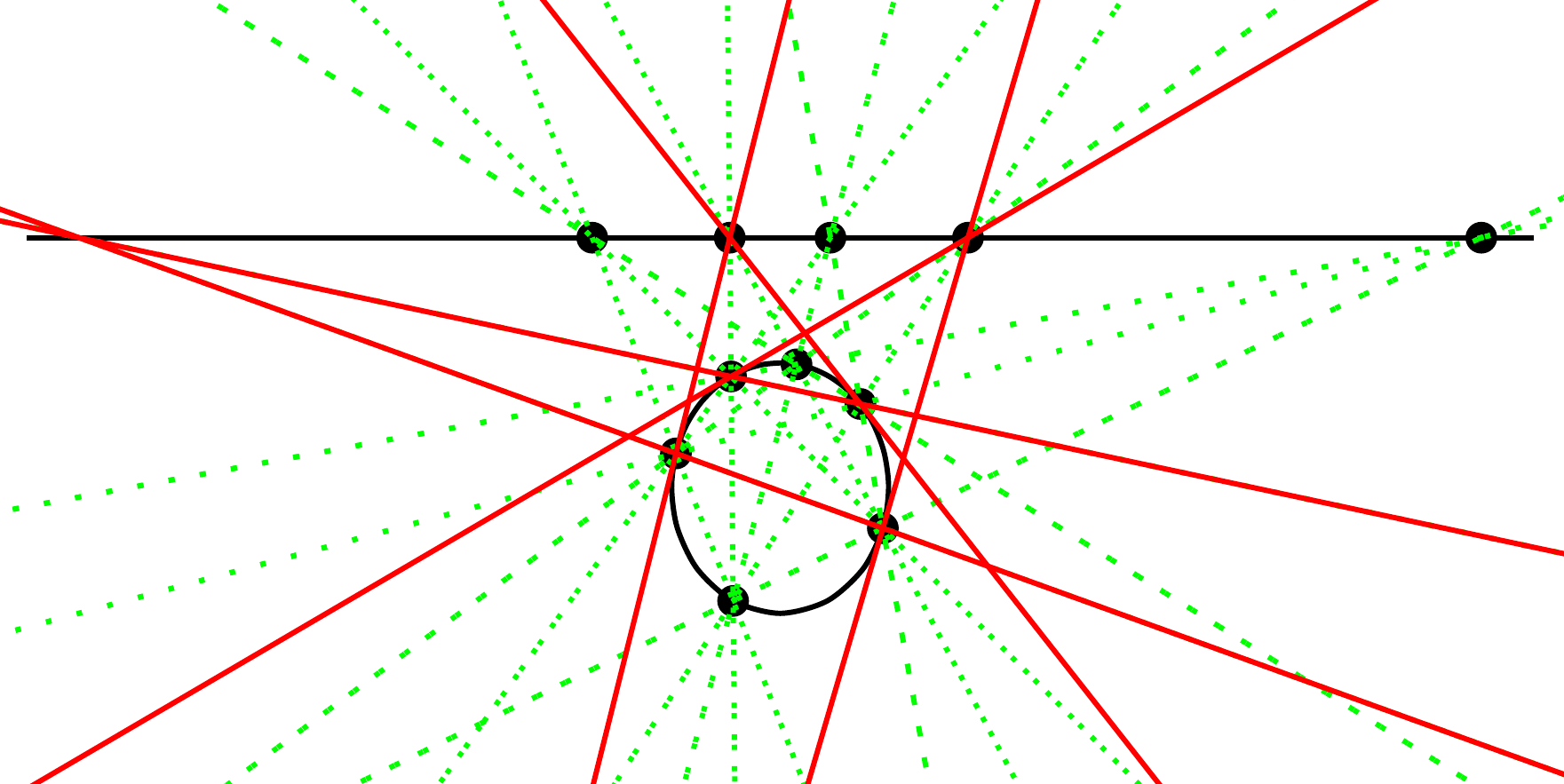}
\caption{The B\"or\"oczky example $X_{12}$ minus the point at infinity $[0,1,0]$, a set with $n = 11$ points and $6$ ordinary lines. The ordinary lines (in red) are the 4 tangent lines to the $6$th roots of unity on the unit circle not through the point at infinity, plus 2 extra lines passing through the point at infinity.}
\label{boro-3-fig}
\end{figure}

Finally, in case (iv) the ordinary lines are the $2m + 1$ tangent lines to the $(2m+1)^{\operatorname{th}}$ roots of unity \emph{except} for one whose corresponding point at infinity has been removed, together with $m$ new ordinary lines joining pairs of roots of unity. This is illustrated in Figure \ref{boro-4-fig} in the case $m = 2$. 
\begin{figure}[h!]
\includegraphics[scale=0.6]{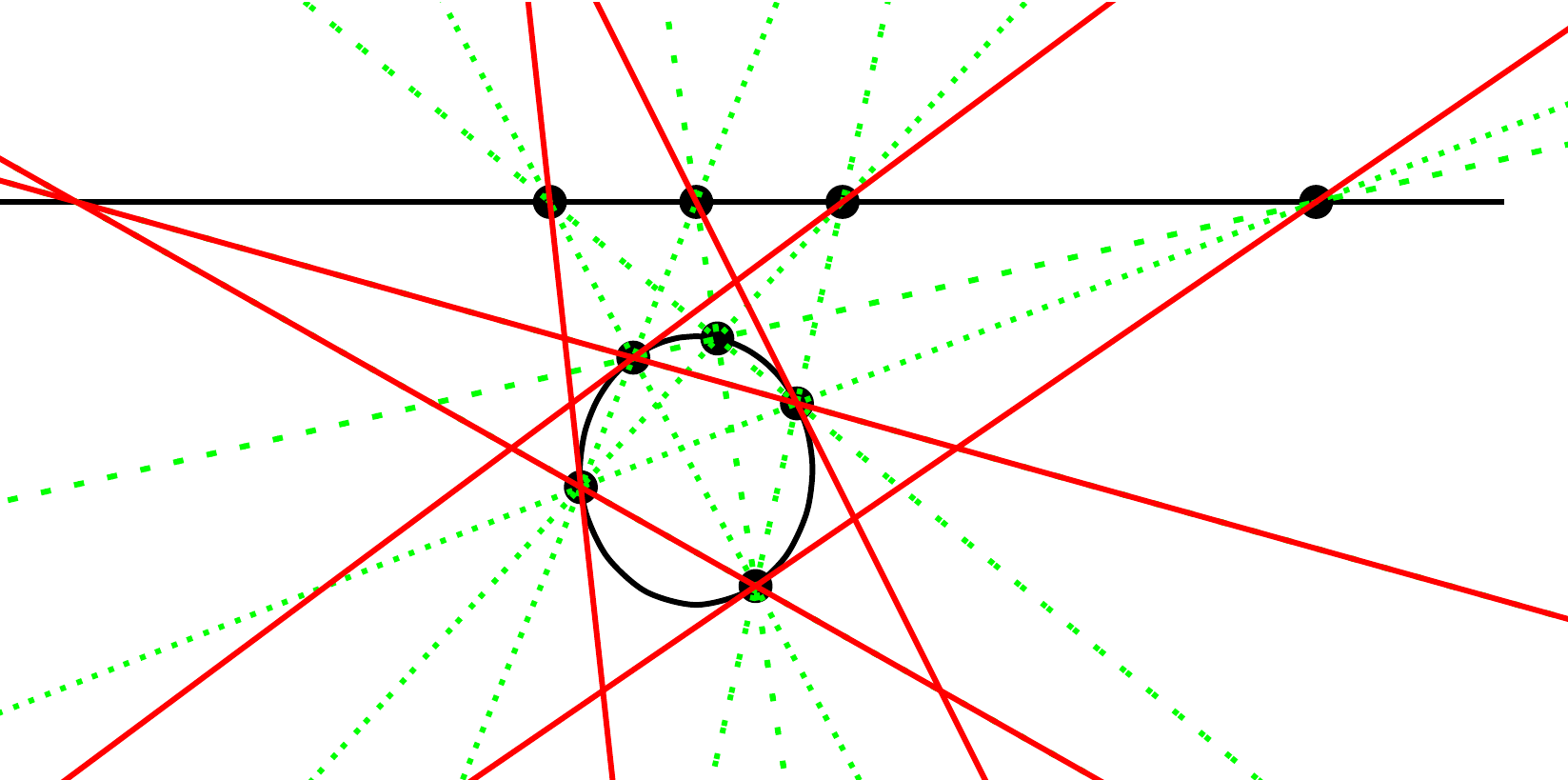}
\caption{The B\"or\"oczky example $X_{10}$ minus the point at infinity $[0,1,0]$, a set with $n = 9$ points and $6$ ordinary lines. The ordinary lines (in red) are the 4 tangent lines to the $6$th roots of unity on the unit circle not through the point at infinity, plus 2 extra lines passing through the point at infinity.}
\label{boro-4-fig}
\end{figure}

\end{proof}

We remark that Proposition \ref{boroczky-examples} illustrates a basic fact, namely that if one adds or removes $K$ points to an $n$-point configuration, then the number of ordinary lines (or $3$-rich lines) is modified by at most $O(Kn + K^2)$; this can be seen by first considering the $K=1$ case and then iterating.  This stability with respect to addition or deletion of a few points is reflected in the conclusions of the various structural theorems in this paper.

We may now state our more precise version of the Dirac-Motzkin conjecture for large $n$.

\begin{theorem}[Sharp threshold for Dirac-Motzkin]\label{dirac-motzkin-1}
Let the function $f : \N \rightarrow \N$ be defined by setting $f(2m) := m$, $f(4m+1) := 3m$ and $f(4m - 1) := 3m - 3$. There is an $n_0$ such that the following is true. If $n \geq n_0$ and if $P$ is a set of $n$ points in $\R \P^2$, not all on a line, then $P$ spans at least $f(n)$ ordinary lines. Furthermore if equality occurs then, up to a projective transformation, $P$ is one of the B\"or\"oczky examples described in Proposition \ref{boroczky-examples} above.  
\end{theorem}

\emph{Remark.} Note in particular that there is an essentially unique extremal example unless $n \equiv 1 \md{4}$, in which case there are two, namely examples (ii) and (iv) above. Note that all of the examples in (iv) are equivalent up to rotation.

Let us record, in addition to the B\"or\"oczky examples mentioned in Proposition \ref{boroczky-examples}, the following near-extremal example. 

\begin{proposition}[Near-B\"or\"oczky example]
The set $X_{4m}$ minus the point $[-\sin \frac{\pi}{2m}, \cos \frac{\pi}{2m},0]$ on the line at infinity contains $4m -1$ points and spans $3m$ ordinary lines.
\end{proposition}

\begin{proof}
This is illustrated in Figure \ref{boro-5-fig} in the case $m = 3$. The ordinary lines are the $2m$ tangent lines to the $2m^{\operatorname{th}}$ roots of unity as well as $m$ lines joining $[\cos \frac{\pi j}{m}, \sin \frac{\pi j}{m},1]$ and $[\cos \frac{\pi j'}{m}, \sin \frac{\pi j'}{m},1]$ with $j + j' \equiv 1 \md{2m}$.
\end{proof}

\begin{figure}[h!]
\includegraphics[scale=2.5]{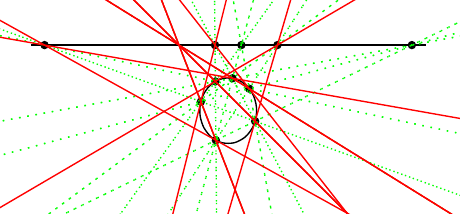}
\caption{The near-B\"or\"oczky example with $n = 11$ points and $9$ ordinary lines. The ordinary lines (in red) are the 6 tangent lines to the $6$th roots of unity on the unit circle plus 3 lines passing through the removed point $[-\sin \frac{\pi}{6}, \cos\frac{\pi}{6},0]$.}
\label{boro-5-fig}
\end{figure}

We may now state a still more precise result, which asserts that all configurations not equivalent to one of the above examples must necessarily have a significantly larger number of ordinary lines than $f(n)$, when $n$ is large. In fact there must be at least $n-O(1)$ ordinary lines in such cases.

\begin{theorem}[Strong Dirac-Motzkin conjecture]\label{dirac-motzkin-2}
There is an absolute constant $C$ such that the following is true. If $P$ is a set of $n$ points in $\R\P^2$, not all on a line,  spanning no more than $n - C$ ordinary lines then $P$ is equivalent under a projective transformation to one of the B\"or\"oczky examples or to a near-B\"or\"oczky example. 
\end{theorem}

The threshold $n-C$ is sharp except for the constant $C$. Indeed we will shortly see that finite subgroups of elliptic curves of cardinality $n$ give examples of sets with $n-O(1)$ lines. This gives infinitely many new examples of sets with few ordinary lines which are inequivalent under projective transformation due to the projective invariance of the discriminant of an elliptic curve.\vspace{11pt}

\emph{Sylvester's cubic curve examples.} We turn now to Sylvester's examples of point sets coming from cubic curves, as further discussed by Burr, Gr\"unbaum and Sloane \cite{bgs}. While these do not provide the best examples of sets with few ordinary lines, it appears that consideration of them is essential in order to solve the Dirac-Motzkin problem. Of course, they also feature in the statement of our main structural result, Theorem \ref{main-structure-theorem}, and are optimal for the orchard problem (see Section \ref{orchard}). Finally, they provide essentially different examples of sets with $n + O(1)$ ordinary lines to any of those considered so far. 

For a leisurely discussion of all the projective algebraic geometry required in this paper, including an extensive discussion of cubic curves, we recommend the book \cite{bix}. 

Let $\gamma$ be any irreducible cubic curve.  It is known (see \cite[Chapter 12]{bix}) that $\gamma$ has a point of inflection, that is to say a point where the tangent meets $\gamma$ to order $3$. By moving this to the point $[0,1,0]$ at infinity, we may bring $\gamma$ into the form $y^2 = f(x)$ in affine coordinates, where $f(x)$ is a cubic polynomial. If $\gamma$ is smooth then it is called an \emph{elliptic curve}. An elliptic curve may have one or two components; these two cases are illustrated in Figure \ref{elliptic}. If $\gamma$ has a singular point then it may be transformed into one of the following three (affine) forms:

\begin{itemize}
\item (nodal case) $y^2 = x^2(x+1)$;
\item (cuspidal case) $y^2 = x^3$;
\item (acnodal case) $y^2 = x^2 (x-1)$.
\end{itemize}

See \cite[Theorem 8.3]{bix} for details. These three singular cases are illustrated in Figure \ref{singular-cubics}. 

\begin{figure}
\includegraphics[scale=0.7]{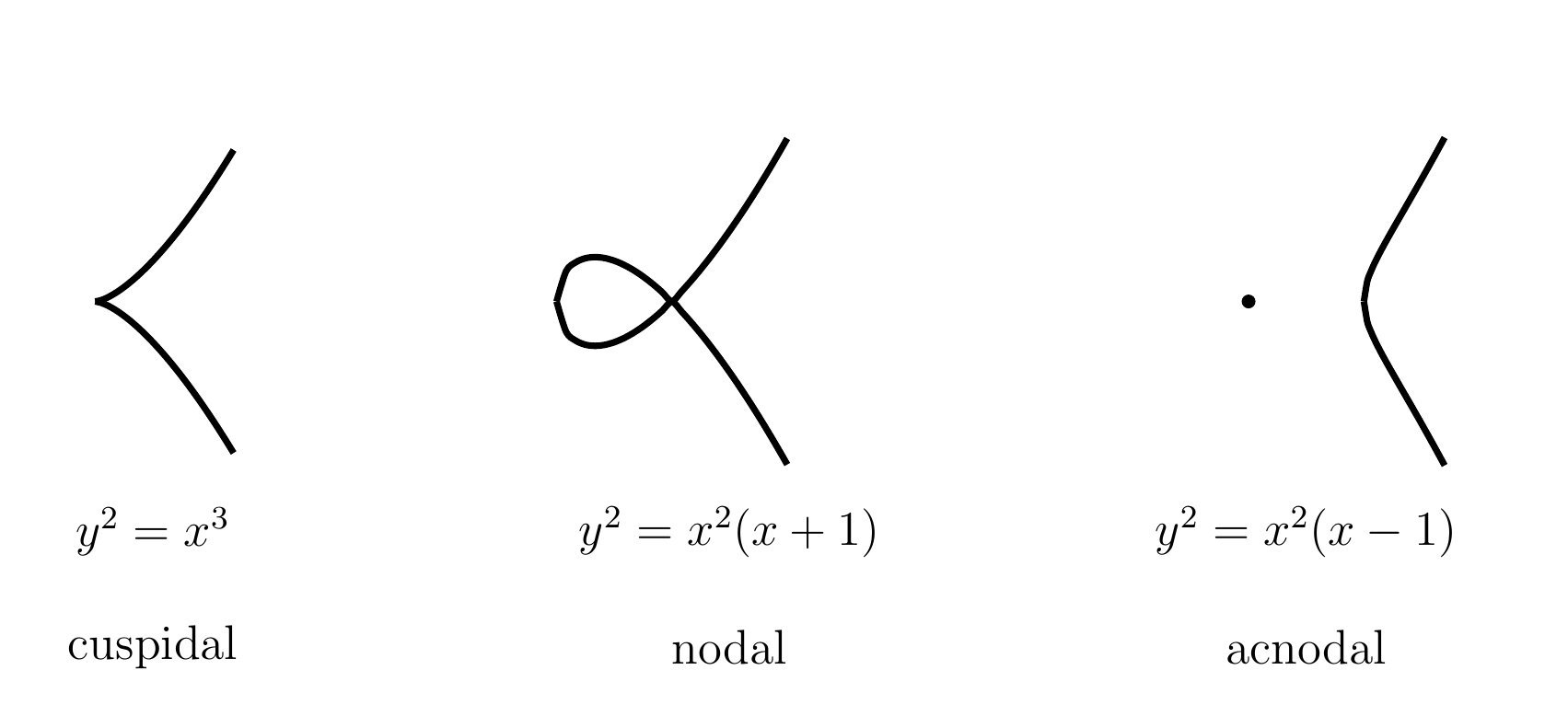}
\caption{The three different types of singular cubic curve. }
\label{singular-cubics}
\end{figure}

We remark that the classification of cubic curves over $\R$ has a long and honourable history dating back to Isaac Newton. \vspace{11pt}

\emph{The group law.} Suppose that $\gamma$ is an irreducible cubic curve, and write $\gamma^*$ for the set of nonsingular points of $\gamma$. If $\gamma$ is smooth then of course $\gamma = \gamma^*$, and in this case $\gamma$ is an elliptic curve. We may define an abelian group structure on $\gamma^*$ by taking the identity $O$ to be a point of inflection on $\gamma^*$ and, roughly speaking, $P \oplus Q \oplus R = O$ if and only if $P,Q,R$ are collinear. The ``roughly speaking'' refers to the fact that we must take appropriate account of multiplicity, thus $P \oplus P \oplus Q = O$ if the tangent to $\gamma$ at $P$ also passes through $Q$. The inverse $\ominus P$ of $P$ is defined using the fact that $\ominus P$, $O$ and $P$ are collinear. See \cite[Chapter 9]{bix} for more details, including a proof that this does indeed give $\gamma^*$ the structure of an abelian group.  

We have the following theorem regarding the nature of $\gamma^*$ as a group.

\begin{theorem}\label{gamma-gp} Let $\gamma$ be an irreducible cubic curve, and let $\gamma^*$ be the set of its nonsingular points. 
Then we have the following possibilities for $\gamma^*$, considered as a group:
\begin{itemize}
\item \textup{(elliptic curve case)} $\R/\Z$ or $\R/\Z \times \Z/2\Z$, depending on whether $\gamma$ has $1$ or $2$ connected components;
\item \textup{(nodal case)} $\R \times \Z/2\Z$;
\item \textup{(cuspidal case)} $\R$;
\item \textup{(acnodal case)} $\R/\Z$.
\end{itemize}
\end{theorem}
Once again, details may be found in \cite{bix}.  Thinking about the curves topologically, the theorem is reasonably evident. In the three singular cases isomorphisms $\phi : G \rightarrow \gamma^*$ can be given quite explicitly, as detailed in the following list.

\begin{itemize}
\item In the nodal case $y^2 = x^2(x+1)$, the map $\phi : \R \times \Z/2\Z \rightarrow \gamma^*$ defined by $\phi(x,\eps) = (t^2 - 1, t (t^2 - 1))$, where $t = \coth x$ if $\eps = 0$ and $t = \tanh x$ if $\eps = 1$ provides an isomorphism;

\item In the cuspidal case $y^2 = x^3$, the map $\phi : \R \rightarrow \gamma^*$ defined by $\phi(x) = (\frac{1}{x^3}, \frac{1}{x^2})$ provides an isomorphism; 

\item In the acnodal case $y^2 = x^2(x -1)$, the map $\phi : \R/\Z \rightarrow \gamma^*$ defined by $\phi(x) = (t^2 + 1, t (t^2 + 1))$, where $t = \cot(\pi x)$, provides an isomorphism. 
\end{itemize}
We leave the reader to provide the details. In the nodal case (for example) we recommend first proving that $(t^2 - 1, t(t^2 - 1)$, $(u^2 - 1, u(u^2 - 1))$ and $(v^2 - 1, v (v^2 - 1))$ are collinear if and only if $-v = (1 + tu)/(t + u)$.

The following maps in the other direction, described in Silverman \cite[III. 7]{silverman-tate} (the acnodal case is described in \cite[Exercise 3.15]{silverman-tate}) are perhaps even tidier. Here $\infty = [0,1,0]$.

\begin{itemize}
\item In the nodal case the map $(x,y) \mapsto (y - x)/(y+x)$ and $\infty \mapsto 1$ gives an isomorphism from $\gamma^*$ to $\R^* \cong \R \times \Z/2\Z$;

\item In the cuspidal case the map $(x,y) \mapsto x/y$ and $\infty \mapsto 1$ gives an isomorphism from $\gamma^*$ to $\R$;

\item In the acnodal case the map $(x,y) \mapsto -(x + iy)^2/x^3$ and $\infty \mapsto 1$ gives an isomorphism from $\gamma^*$ to the unit circle $S^1$ in the complex plane.
\end{itemize}

 \vspace{11pt}

\begin{figure}
\includegraphics{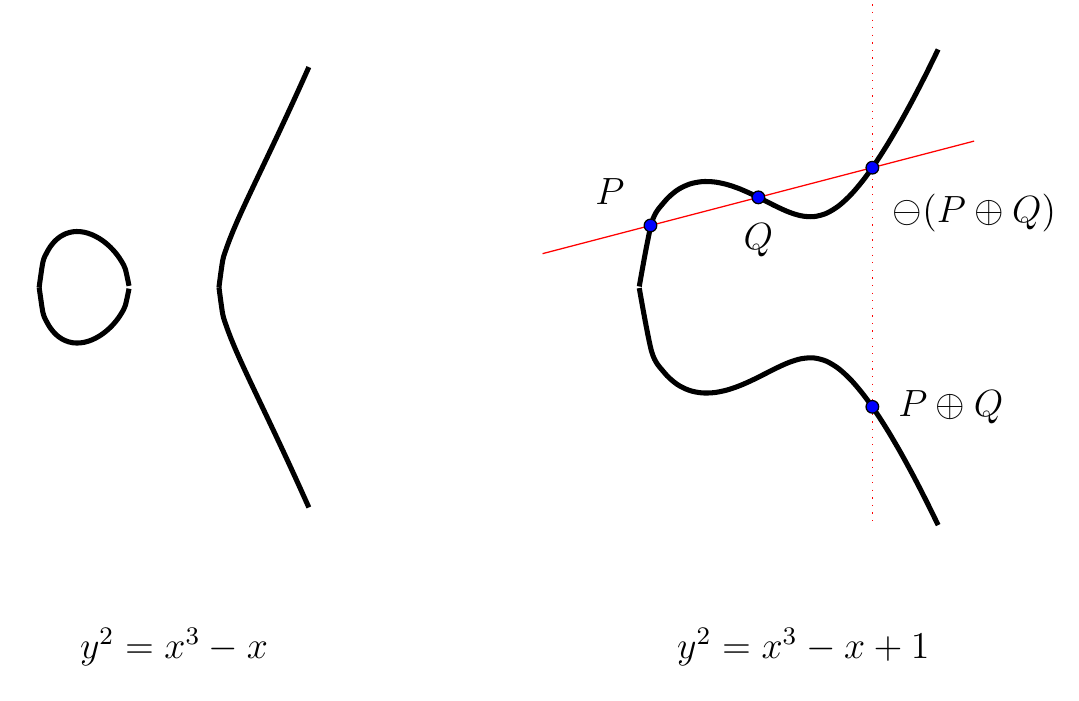}
\caption{Two elliptic curves, illustrating the group law and showing the two possibilites for the group structure. }
\label{elliptic}
\end{figure}

\emph{Sylvester's examples.} By a \emph{Sylvester example} $E_n$ we mean a set of $n$ points $P$ in the plane which corresponds to a subgroup of order $n$ of an irreducible cubic curve $\gamma$. If $n > 2$ the existence of such an example requires $\gamma$ to be either an elliptic curve or an acnodal\footnote{We do not know whether Sylvester himself was interested in the acnodal case. We thank Frank de Zeeuw for correcting an oversight in the first version of this paper by drawing it to our attention.} cubic curve, by the classification of the group structure of $\gamma$ described in Theorem \ref{gamma-gp}. 
A Sylvester example coming from an elliptic curve is depicted in Figure \ref{elliptic-curve-figure}. 

\begin{figure}
\includegraphics[scale=0.4]{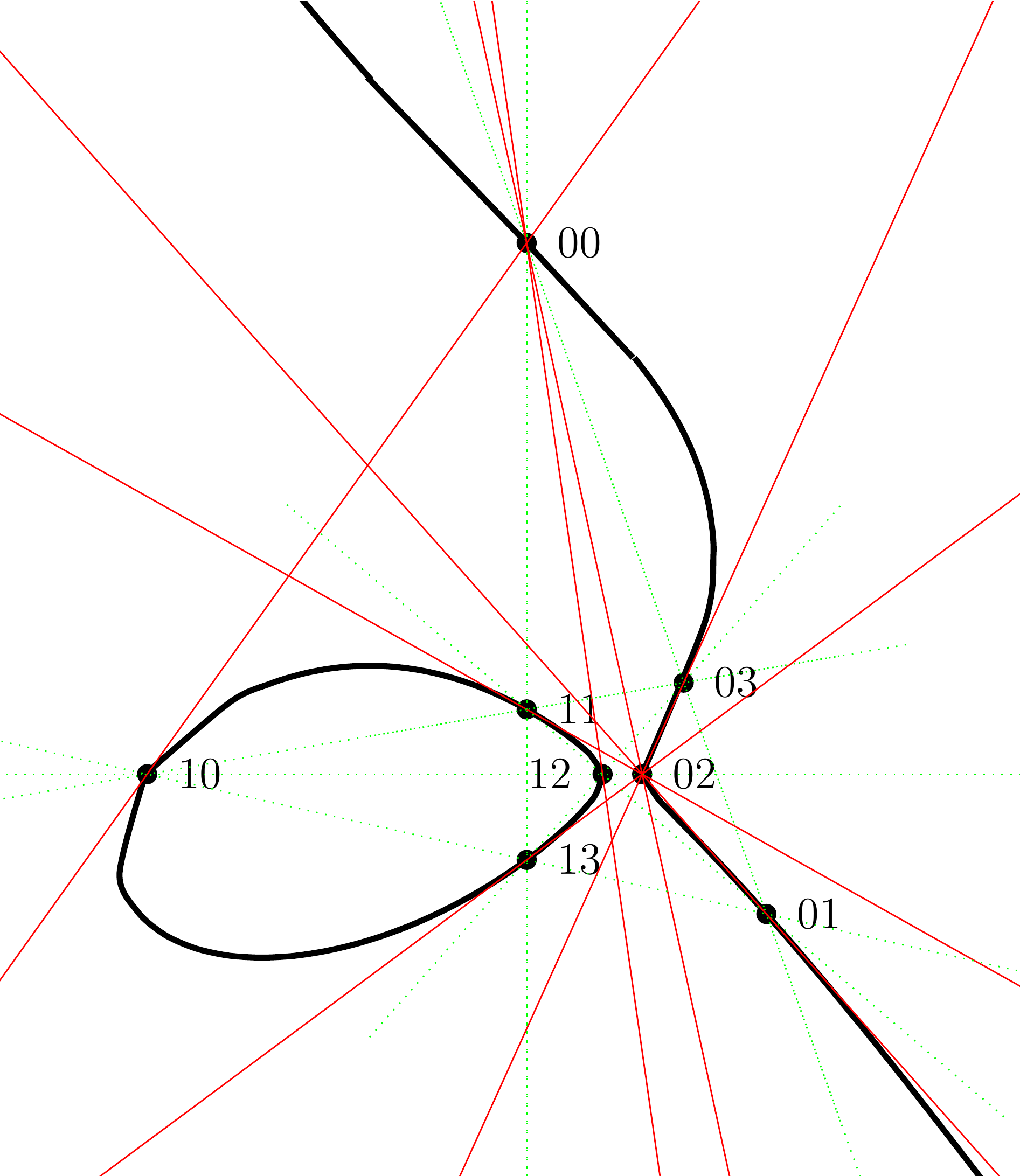}
\caption{A Sylvester example with $n = 8$, the subgroup being isomorphic to $\Z/2\Z \times \Z/4\Z$. The labels reflect the group structure, thus $03$ corresponds to the element $(0,3) \in \Z/2\Z \times \Z/4\Z$. This comes from an elliptic curve with equation $y^2 = x^3 - \frac{1}{36} x^2 - \frac{5}{36} x + \frac{25}{1296} = 0$ to which we have applied the projective transformation $[x,y,z] \mapsto [x,y, x+ y + z]$, so that the point at infinity maps to the point $(0,1)$ in the affine plane (which is then an inflection point for the curve). The are 7 ordinary lines, marked in red, and also $7$ $3$-rich lines, marked in dotted green.}
\label{elliptic-curve-figure}
\end{figure}

As it turns out, Sylvester examples have somewhat more ordinary lines than the B\"or\"oczky examples, namely $n+O(1)$ instead of $n/2+O(1)$ or $3n/4+O(1)$, and are thus not extremisers for the Dirac-Motzkin conjecture.  However, due to the more evenly distributed nature of the Sylvester examples, they have significantly more $3$-rich lines.  Indeed, the following is essentially established in \cite{bgs}.

\begin{proposition}\label{bgs-prop} Let $n \geq 3$, and let $E_n$ be a subgroup of order $n$ in $\gamma^*$, the group of nonsingular points of an irreducible cubic curve $\gamma$ \textup{(}which must be an elliptic curve or an acnodal cubic\textup{)}.  Then $E_n$ spans $n-1-2 \cdot\mathbf{1}_{3|n}$ ordinary lines and $\lfloor \frac{n(n-3)}{6} \rfloor + 1$ $3$-rich lines, where $\mathbf{1}_{3|n}$ is equal to $1$ when $3$ divides $n$ and zero otherwise. Furthermore, if $x \in E$ is such that $x \not \in E_n$ and $x \oplus x \oplus x \in E_n$ then $E_n \oplus x$ has $n-1$ ordinary lines and $\lfloor \frac{n(n-3)}{6} \rfloor$ $3$-rich lines.
\end{proposition}
\begin{proof}  Let $N_2$ be the number of ordinary lines, and $N_3$ be the number of $3$-rich lines.  From B\'ezout's theorem no line can meet $E_n$ in more than three points, and so by double counting we have the identity
$$ N_2 + \binom{3}{2} N_3 = \binom{n}{2}.$$
A brief computation (splitting into three cases depending on the residue of $n$ modulo $3$) then shows that $N_3 = \lfloor \frac{n(n-3)}{6} \rfloor + 1$ if and only if $N_2 = n-1-2 \cdot \mathbf{1}_{3|n}$.  But from the group law the number of ordinary lines is precisely equal to the number of elements $a \in E_n$ such that $-2a$ is distinct from $a$, or in other words the number $n$ of elements in $E_n$ minus the number of third roots in $E_n$. But $\gamma^*$ is isomorphic as a group to either $\R/\Z$ or $(\R/\Z) \times (\Z/2\Z)$, and so $E_n$ is isomorphic to either $\Z/n\Z$ or to $(\Z/(n/2)\Z) \times (\Z/2\Z)$. It has $1 + 2 \cdot \mathbf{1}_{3|n}$ third roots in either case, and the claim follows.

The analysis in the shifted case $E_n \oplus x$ is analogous, the only difference being that $E_n \oplus x$ does not contain any third roots of unity.
\end{proof}

\emph{Remarks.}  For the sake of comparison, the $n$-point examples in Proposition \ref{boroczky-examples} can all be computed to have $n^2/8 + O(n)$ $3$-rich lines instead of $n^2/6 + O(n)$ for the Sylvester examples.  This discrepancy can be explained by the existence of a high-multiplicity line with $n/2+O(1)$ points in those examples. This absorbs many of the pairs of points that could otherwise be generating $3$-rich lines.

We note also that the acnodal case allows for a quite explicit construction of a set of $n$ points defining $\sim n^2/6$ $3$-rich lines, without the use of the Weierstrass $\wp$-function which would be necessary in the elliptic curve case. We leave the reader to supply the details, using the parametrisation detailed after the statement of Theorem \ref{gamma-gp}. We are not sure whether this point has been raised in the literature before.\vspace{11pt}

\emph{Near-counterexamples.} In addition to the actual examples coming from B\"or\"oczky's constructions and from elliptic curve subgroups, there are also some important ``near-counterexamples'' which do not \emph{directly} enter into the analysis (because they involve an infinite number of points, rather than a finite number), but which nevertheless appear to indirectly complicate the analysis by potentially generating spurious counterexamples to the structural theory of points with few ordinary lines. These then need to be eliminated by additional arguments.

As with the previously discussed examples, the near-counterexamples discussed here will lie on cubic curves.  But whilst the actual examples were on an elliptic curve, an acnodal singular cubic curve, or on the union of a conic and a line, the near-counterexamples will lie on three lines (which may or may not be concurrent), or on a non-acnodal singular cubic curve.

We first consider a near-counterexample on three concurrent lines.  Up to projective transformation, one can take the lines to be the parallel lines
\begin{align*}
\ell_1 &:= \{ [x_1,0,1]: x_1\in \R \} \cup \{ [1,0,0] \} \\
\ell_2 &:= \{ [x_2,1,1]: x_2\in \R \} \cup \{ [1,0,0] \} \\
\ell_3 &:= \{ [x_3,2,1]: x_3\in \R \} \cup \{ [1,0,0] \}.
\end{align*} 
Observe that $[x_1,0,1]$, $[x_2,1,1]$ and $[x_3,2,1]$ are colinear if and only if $x_1 + x_3 = 2x_2$.  Thus, if we consider the infinite point set
\begin{align}\nonumber
 P := \{ [n_1,0,1]: n_1 \in \Z \}  \cup \{ [n_2,1,1]: & n_2 \in \textstyle\frac{1}{2}\displaystyle \Z \} \\ & \cup \{ [n_3,2,1]: n_3 \in \Z \}\label{p1}
\end{align}
then there are no ordinary lines whatsoever; every line joining a point in $P \cap \ell_1$ with a point in $P \cap \ell_2$ meets a point in $P \cap \ell_3$, and similarly for permutations.  If $\Z$ could somehow have a non-trivial finite subgroup, then one could truncate this example into a counterexample to the Sylvester-Gallai theorem, i.e. a finite set with no ordinary lines.  Of course, this cannot actually happen, but this example strongly suggests that one needs to somehow use the torsion-free nature of the additive group $\R$ at some point in the arguments, for instance by exploiting arguments based on convexity, or by using additive combinatorics results exploiting the ordered nature of $\R$. One such example, a variant of which we prove in Lemma \ref{add-comb-lem-7}, is the trivial inequality $|A+B| \geq |A| + |B|-1$ for finite subsets $A,B$ of $\R$. This can be viewed as a quantitative version of the assertion that $\R$ has no non-trivial finite subgroups.

There is a similar near-counterexample involving three non-concurrent lines.  Again, after applying a projective transformation, we may work with the lines
\begin{align*}
\ell_1 &:= \{ [x,0,1]: x\in \R \} \cup \{ [1,0,0] \} \\
\ell_2 &:= \{ [0,y,1]: y\in \R \} \cup \{ [0,1,0] \} \\
\ell_3 &:= \{ [-z,1,0]: z\in \R \} \cup \{ [1,0,0] \}.
\end{align*} 
Observe that if $x,y,z \in \R^\times := \R \backslash \{0\}$, then $[x,0,1]$, $[0,y,1]$ and $[-z,1,0]$ are concurrent precisely when $z = x/y$.  Thus, if we consider the infinite point set
\begin{align}\nonumber
 P := \{ [2^{n_1},0,1]: n_1 \in \Z \} \cup \{ [0, & 2^{n_2},1]: n_2 \in \Z \}  \\ & \cup \{ [-2^{n_3},1,0]: n_3 \in \Z \}\label{p2}
 \end{align}
then again there are no ordinary lines: every line joining a point in $P \cap \ell_1$ with a point in $P \cap \ell_2$ meets a point in $P \cap \ell_3$, and similarly for permutations.  As before, this example suggests that the (essentially) torsion-free nature of the multiplicative group $\R^{\times}$ must somehow come into play at some point in the argument.

Finally, we give an example that lies on a cuspidal singular cubic curve, which after projective transformation can be written as
$$ \gamma := \{ [x,y,z]: y z^2 = x^3 \}.$$
Removing the singular point at $[0,1,0]$, we may parameterise the smooth points $\gamma^*$ of this curve by $\{ [t, t^3, 1]: t \in \R \}$.  
One can compute after a brief determinant computation that three distinct smooth points $[t_1,t_1^3,1]$, $[t_2,t_2^3,1]$ and $[t_3,t_3^3, 1]$ on the curve are concurrent precisely when $t_1 + t_2 + t_3 = 0$.  Thus, if one sets $P$ to be the infinite set
\begin{equation}\label{p3}
 P := \{ [n,n^3,1]: n \in \Z \}
 \end{equation}
then there are very few ordinary lines - indeed only those lines that are tangent to $\gamma$ at one point $[n,n^3,1]$ and meet $\gamma$ at a second point $[(-2n), (-2n)^3,1]$ for some $n \in \Z \backslash \{0\}$ will be ordinary.  This example can be viewed as a degenerate limit of the Sylvester examples $E_n$ when the discriminant is sent to zero and $n$ sent to infinity.  Again, finitary versions of this example can be ruled out, but only after one exploits the torsion-free nature of the group associated to $\gamma^*$, which in this case is isomorphic to $\R$.  Similar remarks also apply to nodal singular cubic curves such as $\{ [x,y,z]: y^2 z = x^3 + x^2 z\}$, the smooth points of which form a group isomorphic to $\R \times (\Z/2\Z)$, which is essentially torsion free in the sense that there are no large finite subgroups.

A variant of the example \eqref{p3} lies on the union
$$ \{ [x,y,z]: yz = x^2 \} \cup \{ [x,y,z]: z = 0 \}$$
of a parabola and the line at infinity.  Observe that two points $[t_1,t_1^2,1]$, $[t_2,t_2^2,1]$ on the parabola and a point $[0,t_3,1]$ on the line at infinity with $t_1,t_2,t_3 \in \R$ are concurrent if and only if $t_3 = t_1 + t_2$.  Thus, the infinite set
\begin{equation}\label{p4}
P := \{ [n,n^2,1]: n \in \Z \} \cup \{ [0,n,1]: n \in \Z \},
\end{equation}
which can be viewed as a degenerate limit of a B\"or\"oczky example, has very few ordinary lines (namely, the line tangent to the parabola at one point $[n,n^2,1]$ and also passing through $[0,2n,1]$).

The existence of these near-counterexamples forces us to use a somewhat \emph{ad hoc} case-by-case analysis. Tools such as Chasles's version of the Cayley-Bacharach theorem, which are valid for all cubic curves, get us only so far. They must be followed up by more specialised arguments exploiting the torsion or lack thereof in the group structure. In this way we can rule out near-counterexamples involving triples of lines, or singular irreducible cubics, until only the B\"or\"oczky and Sylvester type of examples remain.

\section{Melchior's proof of the Sylvester--Gallai theorem}\label{melchior-sec}

In this section we review Melchior's proof \cite{melchior} of the Sylvester-Gallai theorem. As mentioned in the introduction, this is the starting point for all of our arguments.  

\begin{sylvester-gallai-rpt}[Sylvester-Gallai, again]
Suppose that $P$ is a finite set of points in the plane, not all on one line. Then $P$ spans at least one ordinary line.
\end{sylvester-gallai-rpt}

\begin{proof} Let $P$ be a set of $n$ points in $\R\P^2$. Consider the dual collection $P^* := \{ p^*: p \in P \}$ of $n$ lines in $\R\P^2$.  These lines determine a graph\footnote{Strictly speaking, $\Gamma_P$ determines a \emph{drawing} of a graph in the projective plane, because we are viewing the edges as curves in $\R\P^2$ rather than abstract pairs of vertices, but we shall abuse notation by identifying a graph with its drawing.} $\Gamma_P$ in $\R\P^2$ whose vertices are the intersections of pairs of lines $p_1^*, p_2^*$ (or equivalently points $\ell^*$, where $\ell$ is a line joining two or more points of $P$), and whose edges are (projective) line segments of lines in $P^*$ connecting two vertices of $\Gamma_P$ with no vertex in the interior. Note that as the points in $P$ were assumed not to lie on one line, every line in $P^*$ must meet at least two vertices of $\Gamma_P$; in particular, the graph $\Gamma_P$ contains no loops.  (It is however possible for a line to meet exactly two vertices in $\Gamma_P$, in which case those two vertices are joined by two edges, rather than one.)  Also, by construction, each vertex of $\Gamma_P$ is incident to at least two lines in $P^*$.  As such, the graph $\Gamma_P$ partitions the projective plane $\R\P^2$ into some number $V$ of vertices, some number $E$ of edges, and some number $F$ of faces, each of which is the projective image of a polygon.  In particular, each face has at least three edges, and any edge is incident to two distinct faces.

By Euler's formula in the projective plane $\R\P^2$ we have\footnote{The Euler characteristic $\chi(\R\P^2)=1$ of the projective plane is of course half of the Euler characteristic $\chi(S^2) = 2$ of the sphere, as the latter is a double cover of the former.}
\begin{equation}\label{euler} V - E + F = 1.\end{equation}
To proceed further, suppose that for each $k = 2,3,4,\dots$ the set $P$ has $N_k$ lines containing precisely $k$ points of $P$. Then $V$, which by duality is the number of lines defined by pairs of points in $P$, satisfies
\begin{equation}\label{v-eq} V = \sum_{k = 2}^n N_k.\end{equation}
Furthermore the degree $d(\ell^*)$ of a vertex $\ell^*$ in our graph is twice the number of lines in $P^*$ passing through $\ell^*$, which is $2|P \cap \ell|$. Thus, summing over all lines $\ell$, 
\begin{equation}\label{e-eq} 2E = \sum_{\ell} d(\ell^*) = 2\sum_{\ell} |P\cap \ell| = \sum_{k = 2}^n 2k N_k.\end{equation}
Finally, for $s = 3,4,5,\dots$ write $M_s$ for the number of faces in $\Gamma_P$ with $s$ edges. Since each edge is incident to exactly two faces, we have
\begin{equation}\label{f-eq} 2E = \sum_{s = 3}^n s M_s.\end{equation}
Combining \eqref{euler}, \eqref{v-eq}, \eqref{e-eq} and \eqref{f-eq} gives the following expression for $N_2$, the number of ordinary lines:
\begin{equation}\label{ordinary-eq} N_2 = 3 + \sum_{k = 4}^n (k-3)N_k + \sum_{s = 4}^n (s-3) M_s.  \end{equation}
It follows immediately that $N_2 \geq 3$, which of course implies the Sylvester-Gallai theorem.
\end{proof}

After discarding the non-negative term $\sum_{s=4}^n (s-3) M_s$, the equation \eqref{ordinary-eq} implies \emph{Melchior's inequality}
$$
N_2 \geq 3 + \sum_{k = 4}^n (k-3)N_k.$$
In this paper, however, we will need to save the term $\sum_{s=4}^n (s-3) M_s$, as it gives crucial control on the geometry of the dual configuration $\Gamma_P$, ensuring that this configuration resembles a triangulation when $N_2$ is small.  More precisely, we have the following proposition.

\begin{proposition}[Few bad edges]\label{bad-edges-prop}
Suppose that $P$ is a set of $n$ points in the projective plane $\R\P^2$, not all on a line, and suppose that $P$ has at most $Kn$ ordinary lines. Consider the planar graph $\Gamma_P$ obtained by dualising $P$ as described above. Then $\Gamma_P$ is an ``almost triangulation''  in the following sense. Say that an edge of $\Gamma_P$ is \emph{good} if both of its vertices have degree $6$, and if both faces it adjoins are triangles. Say that it is \emph{bad} otherwise. Then the number of bad edges in $\Gamma_P$ is at most $16Kn$.
\end{proposition}

\begin{proof} From \eqref{ordinary-eq} we have
\begin{equation}\label{first-eq}
 \sum_{s=4}^n s M_s \leq 4 \sum_{s=4}^n (s-3)M_s \leq 4N_2 \leq 4Kn.
\end{equation}
Secondly, let us observe that
\begin{equation}\label{second-eq} \sum_{\ell : d(\ell^*) \neq 6} d(\ell^*) \leq 12Kn.\end{equation}
To see this, recall that $d(\ell^*) = 2 |P\cap\ell|$. We thus obtain
\begin{align*} \sum_{\ell : d(\ell^*) > 6} d(\ell^*)  = 2\sum_{\ell : |P\cap\ell| > 3} & |P\cap\ell|  \\ & = \sum_{k \geq 4} 2kN_k  \leq 8 \sum_{k \geq 4} (k-3) N_k  \leq 8Kn.\end{align*}
Noting also that 
\[ \sum_{\ell : d(\ell^*) = 4} d(\ell^*) = 2\sum_{\ell : |P\cap\ell| = 2} |P\cap\ell| = 4 N_2 \leq 4Kn,\] 
\eqref{second-eq} follows.

Now we can place an upper bound on the number $B$ of bad edges. Each face with $s > 3$ sides gives $s$ bad edges, and each vertex $\ell^*$ with degree $d(\ell^*) \neq 6$ gives $d(\ell^*)$ bad edges.  As these are the only sources of bad edges, we have
\[ B \leq \sum_{s > 3} s M_s + \sum_{\ell : d(\ell^*) \neq 6} d(\ell^*) \leq 16Kn,\] 
by \eqref{first-eq} and \eqref{second-eq}.
\end{proof}

\section{Triangular structure in the dual and cubic curves}\label{dual-cubic}

In general, the number of edges overall in $\Gamma_P$ is expected to be of the order of $n^2$ (cf. Beck's theorem \cite{beck}).  Thus, when $K$ is small, Proposition \ref{bad-edges-prop} should be viewed as an assertion that almost all the edges of $\Gamma_P$ are good. For instance, it shows that any dual line $p^*$, $p \in P$ should contain at most $O(K)$ bad edges on the average.  Intuitively, this suggests that $\Gamma_P$ is an ``almost triangulation'' in which most vertices have degree $6$ and most faces are triangles. In Section \ref{section-5} we will use this information to put the points of $P$ on a small number of cubic curves, which will be our starting point for more powerful structural theorems on $P$.

By a \emph{cubic curve} we mean a set of points in $\R\P^2$ of the form
\begin{align*} \{[X,Y,Z]: & a_1X^3 + a_2X^2 Y + a_3XY^2 + a_4Y^3 + a_5X^2 Z \\ & + a_6XYZ + a_7Y^2 Z  + a_8XZ^2 + a_9YZ^2 + a_{10}Z^3 = 0 \}\end{align*}
for some coefficients $a_1,\ldots,a_{10} \in \R$, not all zero, or in other words the locus of a nontrivial homogeneous polynomial of degree $3$. Note that we do not assume this polynomial to be irreducible. In particular, we consider the union of three lines, as well as the union of a conic and a line, to be examples of cubic curves.

A key observation in our arguments will be the fact that pockets of true triangular structure in the dual $\Gamma_P$ signify a collection of points of $P$ lying on a single cubic curve. Results of this type may be found in Lemmas \ref{tri-cubic} and \ref{tri-cubic-detailed} below. A key ingredient will be the following incredibly classical fact from projective geometry, usually known as the \emph{Cayley-Bacharach theorem} (although the case we require was proven by Chasles \cite{chasles}, prior to the more general results of Cayley \cite{cayley} and Bacharach \cite{bacharach}).

\begin{proposition}[Chasles]\label{chas}  Suppose that two sets of three lines define nine distinct points of intersection in $\R\P^2$.  Then any cubic curve passing through eight of these points also passes through the ninth.
\end{proposition}

This situation is shown in Figure \ref{chasles-fig} below. See \cite{eisenbud-green-harris} or the blog post \cite{tao-blog-cayley-bacharach} for a discussion of this result, including its link to Pappus's theorem, Pascal's theorem, and the associativity of the group law on an elliptic curve. 

\begin{figure}\includegraphics{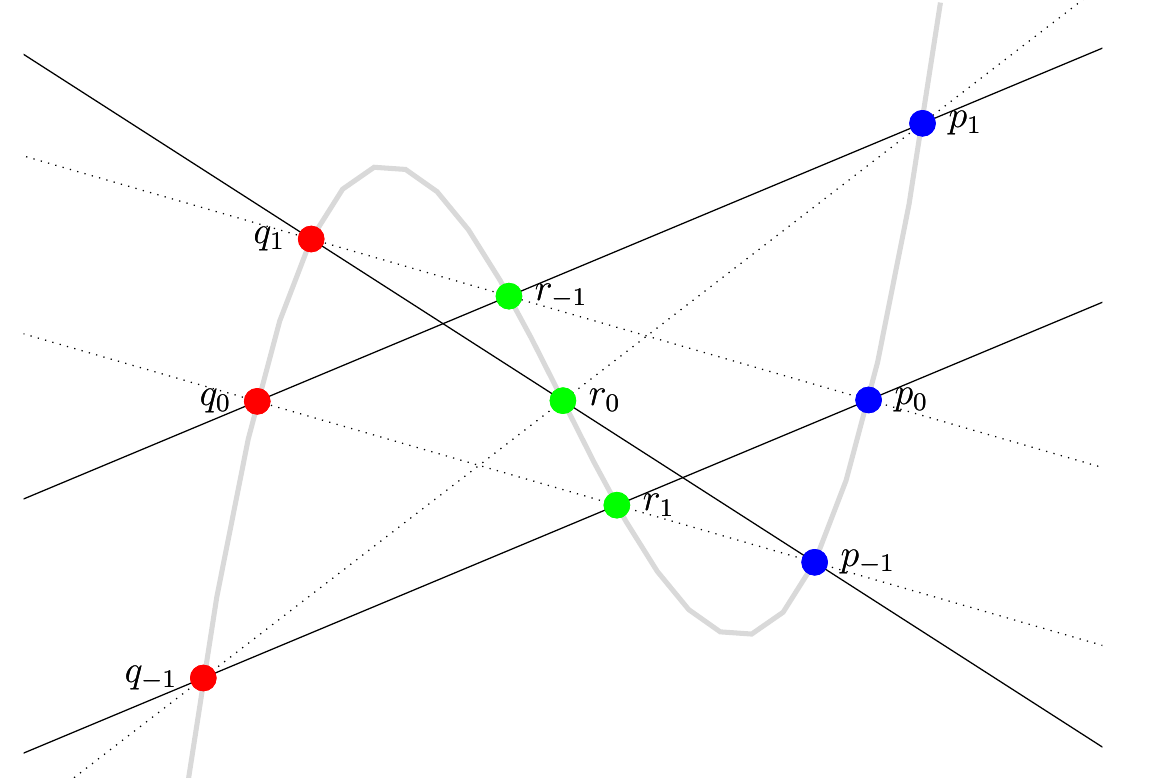}
\caption{Chasles's theorem. There are two sets of three lines: the solid lines $\overline{\{p_0,q_{-1},r_1\}}$, $\overline{\{p_1, q_0, r_{-1}\}}$, $\overline{\{p_{-1}, q_0, r_1\}}$ and the dotted lines $\overline{\{p_0, q_1, r_{-1}\}}$, $\overline{\{p_1, q_{-1}, r_0\}}$ and $\overline{\{p_{-1}, q_0, r_1\}}$. The nine points of intersection $p_{-1}$, $p_0$, $p_1$, $q_{-1}$, $q_0$, $q_1$, $r_{-1}$, $r_0$ and $r_1$ are all distinct. Any cubic curve passing through $8$ of these points also passes through the $9$th; one such curve is shown.}
\label{chasles-fig}
\end{figure}

Proposition \ref{chas} allows one to establish some duality relationships\footnote{Let us note once again that rather similar results were obtained earlier, in a completely different context, by Carnicer and God\'es \cite{carnicer-godes}.} between triangular grids and cubic curves. We first define what we mean by a triangular grid:

\begin{definition}[Triangular grid]\label{trio}  Let $I,J,K$ be three discrete intervals in $\Z$ \textup{(}thus $I$ takes the form $\{ i \in \Z: i_- \leq i \leq i_+\}$ for some integers $i_-, i_+$, and similarly for $J$ and $K$\textup{)}.  A \emph{triangular grid} with dimensions $I,J,K$ is a collection of lines $(p_i^*)_{i \in I}, (q_j^*)_{j \in J}, (r_k^*)_{k \in K}$ in $\R\P^2$, which we will view as duals of not necessarily distinct points $p_i, q_j, r_k$ in $\R\P^2$, obeying the following axioms:
\begin{itemize}
\item[(i)] If $i \in I, j \in J, k \in K$ are integers with $i+j+k=0$, then the lines $p_i^*, q_j^*, r_k^*$ are distinct and meet at a point $P_{ijk}$.  Furthermore, this point $P_{ijk}$ is not incident to any line in the grid which is not already identical to one of the lines $p_i^*, q_j^*, r_k^*$.  Thus, for instance, if $i' \in I$ is such that $p_{i'}^* \neq p_i^*, q_j^*, r_k^*$, then $p_{i'}^*$ cannot contain $P_{ijk}$.
\item[(ii)]  If $i \in I, j,j' \in J, k,k' \in K$ are such that $i+j+k=i+j'+k'=0$ and $0 < |j-j'| \leq 2$ \textup{(}or equivalently $0 < |k-k'| \leq 2$\textup{)}, then the intersection points $P_{ijk}$ and $P_{ij'k'}$ are distinct. In particular, this forces $q_j^* \neq q_{j'}^*$ and $r_k^* \neq r_{k'}^*$.  Similarly for cyclic permutations of $i,j,k$ and of $j',k'$.
\end{itemize}
\end{definition}

An example of a triangular grid is depicted in Figure \ref{triangular-grid}. %Informally, a triangular grid is a collection of lines which has the combinatorial structure of a portion of the standard infinite triangular grid (Figure \ref{tri}).  Morally speaking, all the lines $p_i^*, q_j^*, r_k^*$ in a triangular grid ought to be distinct.  This is not quite the case ``globally'', but the distinctness hypotheses in Definition \ref{trio} will give this distinctness ``locally'', which will suffice for our purposes.  (Intuitively, one can think of a triangular grid as an \emph{immersion} of a portion of the infinite triangular grid, but not necessarily as an \emph{embedding} of that portion.)  It is worth pointing out that the infinite triangular grid also arises as the dual configuration of the infinite near-counterexamples \eqref{p1}, \eqref{p2}, while the dual configuration of the near-counterexamples \eqref{p3}, \eqref{p4} similarly contain a half-infinite triangular grid.

\begin{figure}\includegraphics[scale=0.8]{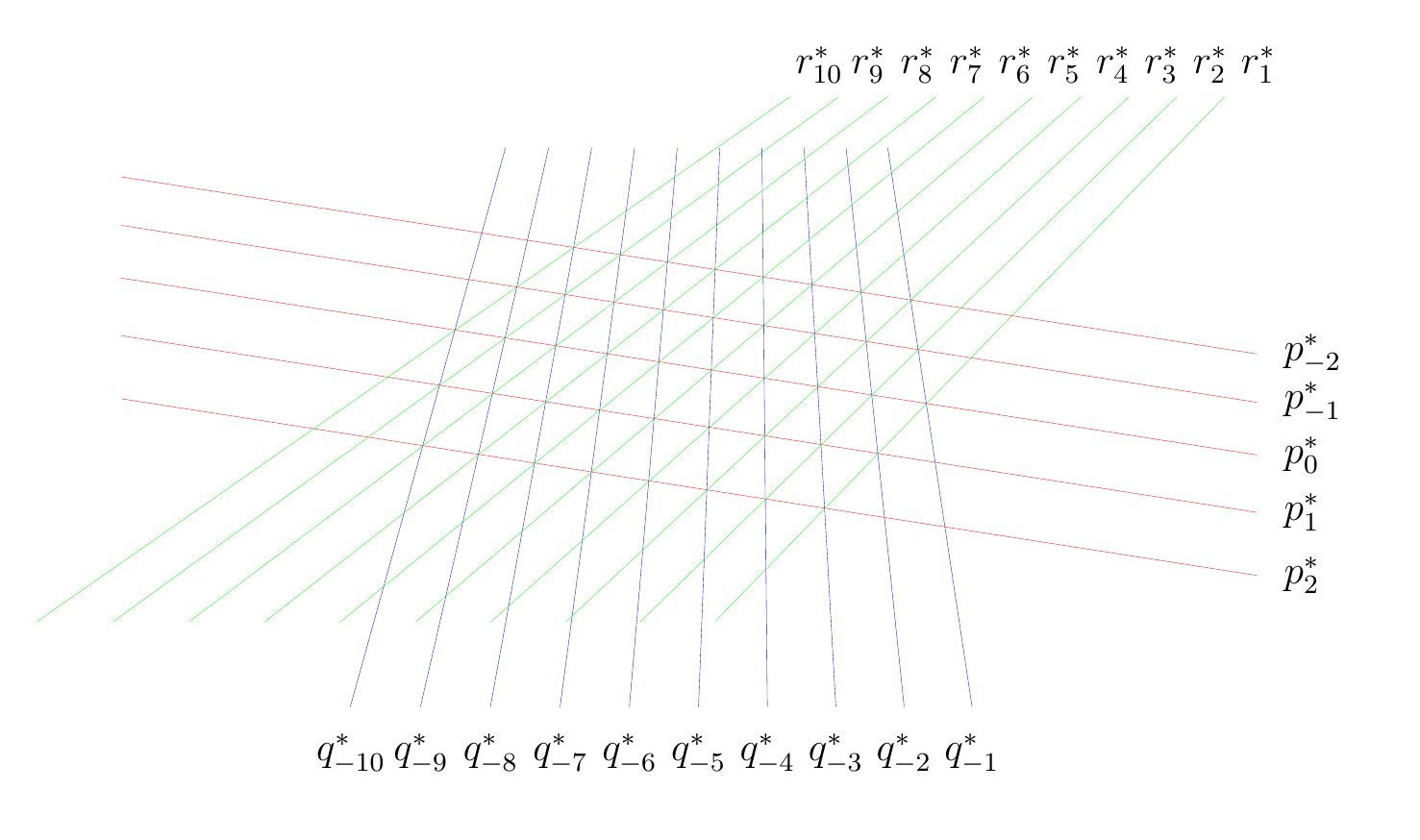}
\caption{A triangular grid with dimensions $\{-2,\dots,2\}$, $\{-10,\dots,-1\}$ and $\{1,\dots,10\}$.}
\label{triangular-grid}
\end{figure}

The following basic consequence of Proposition \ref{chas} drives our whole argument.

\begin{lemma}[Completing a hexagon]\label{hexcomp}  Let $i_0,j_0,k_0$ be integers with $i_0+j_0+k_0=0$, let $I := \{i_0-1,i_0,i_0+1\}$, $J := \{j_0-1,j_0,j_0+1\}$, $K := \{k_0-1,k_0,k_0+1\}$, and let $(p_i)_{i \in I}, (q_j)_{j \in J}, (r_k)_{k \in K}$ be triples of points whose duals form a triangular grid with dimensions $I,J,K$.  Then the nine points $(p_i)_{i \in I}, (q_j)_{j \in J}, (r_k)_{k \in K}$ are distinct, and any cubic curve which passes through eight of them passes through the ninth.
\end{lemma}
\begin{proof}  By relabeling, we may assume that $i_0=j_0=k_0=0$, thus the nine points are $p_{-1},p_0,p_1,q_{-1},q_0,q_1,r_{-1},r_0,r_1$. Once it is shown that these nine points are distinct, their duals form a ``hexagon''  as depicted in Figure \ref{hex} below as part of a larger triangular grid. The configuration in Figure \ref{hex}, however, is precisely the dual of the configuration of 9 points appearing in Chasles's theorem (see Figure \ref{chasles-fig}), and the claim then follows.

It remains to establish the distinctness of the nine points.  By applying Definition \ref{trio} (ii) to the intersections of $p_i^*,q_0^*,r_{-i}^*$ for $i=-1,0,1$ we see that the $p_{-1},p_0,p_1$ are distinct; similarly for $q_{-1},q_0,q_1$ and $r_{-1},r_0,r_1$.  Next, from Definition \ref{trio}(i) we see that $p_i$ and $q_j$ are distinct as long as $-1 \leq i+j \leq 1$, and similarly for cyclic permutations.  The only remaining claim left to check, up to permutations and reflections, is that $p_1$ and $q_1$ are distinct.  But if these two points coincided, then the intersections of $p_1^*, q_{-1}^*, r_0^*$ and $p_{-1}^*, q_1^*, r_0^*$ would then also coincide, contradicting Definition \ref{trio}(ii).  \end{proof}

\begin{figure}
\includegraphics[scale=0.8]{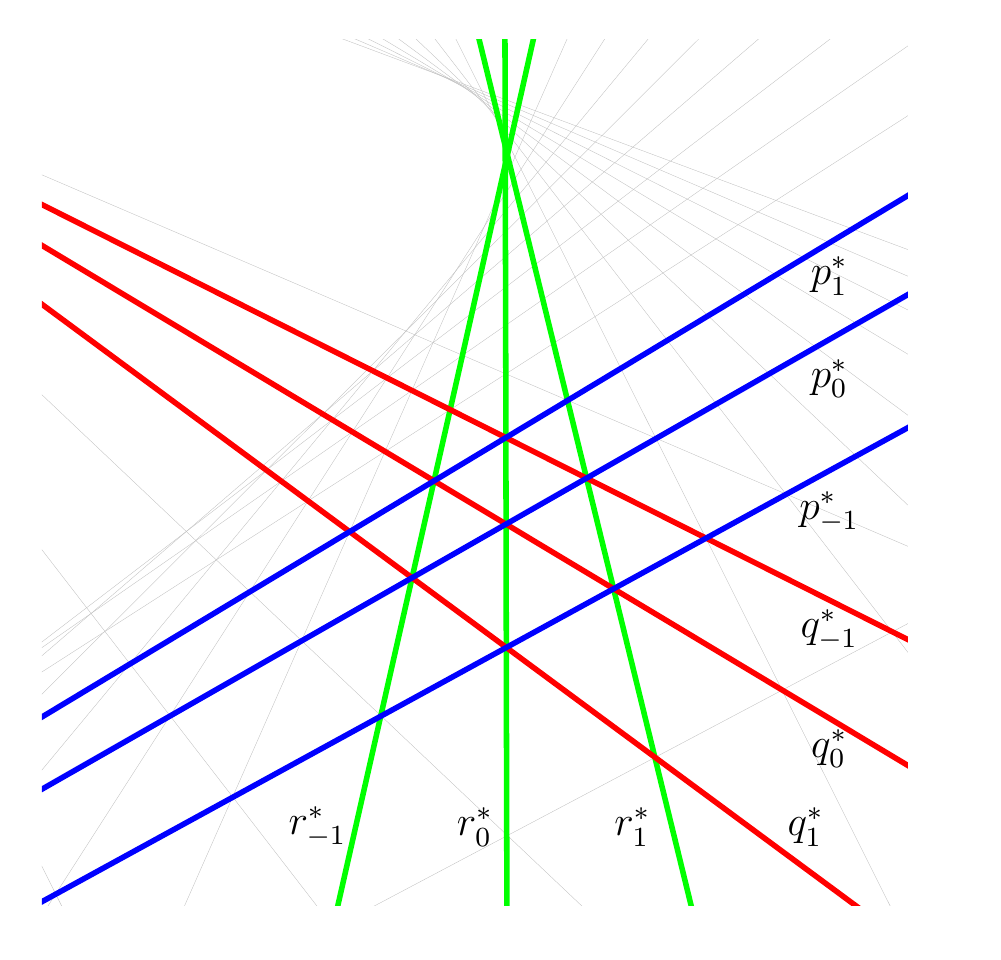}
\caption{The dual of the configuration in Chasles's theorem, showing a ``hexagon'' formed by the duals of two sets of three lines. The light grey lines are duals of other points on the cubic curve shown in Figure \ref{chasles-fig}, specifically those points in longer arithmetic progressions (in the group law on $\gamma$) containing the $p_i$, $q_j$, $r_k$.  We have included them mainly for aesthetic interest, but also as a more complicated example of a triangular grid.}
\label{hex}
\end{figure}

We now iterate the above proposition.

\begin{lemma}\label{tri-cubic}
Suppose that $m \geq 4$ is an integer and that $i_-, i_+$ are integers with $2 \leq i_+ \leq m-2$ and $2-m \leq i_- \leq -1$.
Suppose that we have a collection of points $(p_i)_{i_- \leq i \leq i_+}$, $(q_j)_{-m \leq j \leq -1}$ and $(r_k)_{1 \leq k \leq m}$ in $\R\P^2$ whose duals form a triangular grid with the indicated dimensions \textup{(}The case $i_- = -2$, $i_+ = 2$ and $m = 10$ is illustrated in Figure \ref{triangular-grid}\textup{)}. Then all of the points $p_i, q_j, r_k$ lie on a single cubic curve $\gamma$. 
\end{lemma}
\begin{proof}   Consider the nine points $p_{-1},p_0,p_1,p_2,q_{-3},q_{-2},q_{-1}, r_1, r_2$.  The space of cubic homogeneous polynomials is a vector space of dimension $10$, and so by straightforward linear algebra there is a cubic curve $\gamma$ containing these nine points $p_{-1}, p_0, p_1, p_2, q_{-3}, q_{-2}, q_{-1}, r_1$ and $r_2$.  (Note that it is not necessary for the nine points to be distinct in order to obtain this claim.)  We will now claim that all the remaining points $p_i,q_j,r_k$ in the configuration also lie on $\gamma$.

Firstly, by applying Lemma \ref{hexcomp} to the set
$$
p_{-1}, p_0, p_1, q_{-3}, q_{-2}, q_{-1}, r_1, r_2, r_3
$$
we see that as eight of the points already lie in $\gamma$, the ninth point $r_3$ must also. We now know that the 10 points
$$
p_{-1}, p_0, p_1, p_2, q_{-3}, q_{-2}, q_{-1}, r_1, r_2, r_3
$$
all lie on $\gamma$.
Now apply Lemma \ref{hexcomp} to the set
$$
p_{0}, p_1, p_2, q_{-4}, q_{-3}, q_{-2}, r_1, r_2, r_3.
$$
We conclude that $q_{-4}$ also lies on $\gamma$, so now the 11 points
$$
p_{-1}, p_0, p_1, p_2, q_{-4}, q_{-3}, q_{-2}, q_{-1}, r_1, r_2, r_3
$$
all lie on $\gamma$.

Next apply Lemma \ref{hexcomp} to the set
$$
p_{-1}, p_0, p_1, q_{-4}, q_{-3}, q_{-2},  r_2, r_3, r_4
$$
to conclude that $r_{4}$ lies on $\gamma$. We now know that the 12 points
$$
p_{-1}, p_0, p_1, p_2, q_{-4}, q_{-3}, q_{-2}, q_{-1}, r_1, r_2, r_3, r_4
$$
all lie on $\gamma$.

By shifting the $q$ indices down by one and $r$ indices up by one repeatedly, we may then inductively place $q_{-k}$ and $r_k$ in $\gamma$ for all $4 \leq k \leq m$.  Finally, by applying Lemma \ref{hexcomp} inductively to the sets
$$
p_{i-1}, p_{i}, p_{i+1}, q_{-i-3}, q_{-i-2}, q_{-i-1}, r_1, r_2, r_3
$$
for $i=2,\ldots,i_+-1$ (noting that $-i-3 \geq -i_+-2 \geq -m$) we may place $p_i$ in $\gamma$ for all $2 < i \leq i_+$, and similarly by applying Lemma \ref{hexcomp} inductively to 
$$
p_{-i-1}, p_{-i}, p_{-i+1}, q_{-3}, q_{-2}, q_{-1}, r_{i+1}, r_{i+2}, r_{i+3}
$$
for $i=1,\ldots,-i_--1$ we can also place $p_{-i}$ in $\gamma$ for all $1 < i \leq -i_-$. This concludes the proof of the claim.
\end{proof}

This lemma is already enough to imply our most basic structural result for sets with few ordinary lines, Proposition \ref{basic-cubic-covering}. To get stronger results, such as Proposition \ref{intermediate}, we need to perform a deeper analysis. The new feature in the following lemma is the last statement.

\begin{lemma}\label{tri-cubic-detailed}
Suppose that $L \geq 10$ and that $m \geq 10L$. Suppose that we have a collection of $4L + 1 + 2m$ points $(p_i)_{-2L \leq i \leq 2L}$, $(q_j)_{-m \leq j \leq -1}$ and $(r_k)_{1 \leq k \leq m}$ in $\R\P^2$ whose duals form a $(4L+1) \times m \times m$ triangular grid with the indicated dimensions.  Assume furthermore that the points $p_i,q_j,r_k$ are all distinct.  Then all of the points $p_i, q_j, r_k$ lie on a single cubic curve $\gamma$, each irreducible component of which contains at least $L$ of the points $p_i, q_j, r_k$. 
\end{lemma}

\begin{proof} Note from Definition \ref{trio} and the distinctness of the $p_i,q_j,r_k$ that the intersection points $P_{ijk} = p_i^* \cap q_j^* \cap r_k^*$ in the grid are all distinct.

That all the $p_i, q_j, r_k$ lie on a single cubic curve $\gamma$ follows from Lemma \ref{tri-cubic}.

If $\gamma$ is already an irreducible cubic then we are done. By enlarging $\gamma$ via the addition of extra lines if necessary, we may otherwise suppose that we are in one of the following two cases:

\emph{Case 1:} $\gamma$ is the union of three distinct lines $\ell, \ell', \ell''$;

\emph{Case 2:} $\gamma$ is the union of an irreducible conic $\sigma$ and a line $\ell$. 

In each case we are to show that all irreducible components contain at least $L$ points $p_i, q_j, r_k$.

In Case 1, consider a triple of points $p_i, q_j, r_k$ with $i + j + k = 0$. Since $p_i, q_j, r_k$ are collinear and lie on $\ell \cup \ell' \cup \ell''$, one of the following two possibilities holds: 
\begin{enumerate}
\item One of the three lines $\ell,\ell',\ell''$ is incident to all three of $p_i, q_j, r_k$ (i.e. the line $\overline{\{p_i,q_j,r_k\}}$ is one of $\ell, \ell'$, or $\ell''$);
\item $p_i, q_j, r_k$ lie on one of each of the lines $\ell, \ell', \ell''$ (for instance, one could have $p_i \in \ell', q_j \in \ell'', r_k \in \ell$, or any of the other five possible permutations).  Note that we allow a point to lie on more than one of the lines $\ell,\ell',\ell''$.
\end{enumerate}

First of all note that (i) cannot hold for more than three triples $(i,j,k)$ with $i + j + k = 0$. Indeed, as observed previously, the intersection points $P_{ijk} := p_i^* \cap q_j^* \cap r_k^*$ are all distinct, and so the lines containing $\{p_i, q_j, r_k\}$ are distinct for distinct triples $(i,j,k)$.

Let $\Omega$ be the set of triples $(i,j,k)$ with $i + j + k = 0$, $-L \leq i \leq L$, $-2L \leq j < -L$ and $L < k \leq 2L$. Suppose that (i) holds for some triple $(i,j,k) \in \Omega$ and that $p_i, q_j, r_k$ all lie on $\ell$ (say). We now consider the triples $(i', j', k') \in \Omega$ with $i \neq i'$, $j \neq j'$, $k \neq k'$.  With at most two exceptions, (ii) holds for any such triple. Fix one of these triples for which (ii) holds.  One of the points $p_{i'}, q_{j'}, r_{k'}$ then lies on $\ell$. Suppose that $p_{i'}$ lies in $\ell$.  But, noting that $1 \leq -i' - j \leq m$, we see that $p_{i'}, q_j$ and $r_{-i' - j}$ are collinear and so $r_{-i' - j}$ lies on $\ell$ as well. That is, the lines containing $\{ p_i, q_j, r_k\}$ and $\{p_{i'}, q_j, r_{-i'-j}\}$ are the same. This is a contradiction as we noted above.  Similarly, if $q_{j'}$ lies in $\ell$, then $1 \leq -i -j' \leq m$ and we can conclude that the lines containing $\{p_i,q_j,r_k\}$ and $\{p_i,q_{j'},r_{-i-j'}\}$ are again coincident, a contradiction.  Finally, if $r_{k'}$ lies in $\ell$, then  $-m \leq -i-k' \leq 1$ and the lines containing $\{p_i,q_j,r_k\}$ and $\{p_i,q_{-i-k'},r_{k'}\}$ are coincident, again a contradiction.

It follows that, whenever $(i,j,k) \in \Omega$, we are in case (ii) and not in (i), that is to say the points $p_i, q_j, r_k$ lie on one of each of the lines $\ell, \ell', \ell''$, but do not all lie on one of the lines $\ell$, $\ell'$, or $\ell''$. Suppose without loss of generality that $p_0 \in \ell$, $q_{-2L} \in \ell'$, $r_{2L} \in \ell''$.   
If $q_{-2L + 1} \in \ell''$ then the concurrent lines $p_{-1},q_{-2L+1},r_{2L}$ all lie in $\ell''$; as $(-1,-2L+1, 2L) \in \Omega$, we obtain a contradiction. Similarly, if $q_{-2L+1} \in \ell$, then $p_0, q_{-2L+1},r_{2L+1}$ all lie in $\ell$, again a contradiction.  Thus $q_{-2L + 1} \in \ell'$, which implies $r_{2L-1} \in \ell''$. Repeating this argument we see that in fact all of the points $q_j$, $-2L \leq j < -L$, lie on $\ell'$ and all of the points $r_k$, $L < k \leq 2L$, lie on $\ell''$. Finally, considering the triple $(i, -2L - i, 2L) \in \Omega$, we see that all of the points $p_i$, $0 \leq i < L$, lie on $\ell$. We have established that each of the lines $\ell,\ell',\ell''$ contains at least $L$ of the points $p_i, q_j, r_k$, concluding the proof of the lemma in Case 1.

We turn now to Case 2, where the argument is very similar.  Consider once again a triple of points $p_i, q_j, r_k$ with $i + j + k = 0$. These lie on a line. By B\'ezout's theorem, there are two cases:
\begin{enumerate}
\item $p_i, q_j, r_k$ all lie on $\ell$;
\item two of $p_i, q_j, r_k$ lie on $\sigma$ and the other lies on $\ell$. 
\end{enumerate}
If (i) ever holds for some triple $(i,j,k)$ then there is at most one such triple. Suppose it holds for some triple $(i,j,k) \in \Omega$. There are again many triples $(i',j',k') \in \Omega$ with $i \neq i'$, $j \neq j'$, $k \neq k'$. For any such triple, two of $p_{i'}, q_{j'}, r_{k'}$ lie on $\sigma$ and the other lies on $\ell$. Suppose, without loss of generality, that $p_{i'} \in \ell$. Then, noting that $1 \leq -i' - j \leq m$, we see that $p_{i'}, q_j, r_{-i'-j}$ are collinear and so $r_{-i'-j}$ lies on $\ell$ as well, and thus (i) also holds for the triple $(i', j, -i' - j)$. This leads to a contradiction exactly as before.

It follows that, whenever $(i,j,k) \in \Omega$, two of the points $p_i, q_j, r_k$ lie on $\sigma$ and the other lies on $\ell$. If $p_0 \in \sigma$ then one of $q_{-j}, r_j$ lies on $\sigma$ and the other lies on $\ell$, for each $j$ with $L < j \leq 2L$. Thus both $\sigma$ and $\ell$ contain at least $L$ of the points $p_i, q_j, r_k$.  If, on the other hand, $p_0 \in \ell$, then all the $q_{-j}, r_j$ with $L < j \leq 2L$ lie on $\sigma$. But for each $i$, $|i| \leq L-1$,  there are some $j,k$ with $-2L \leq j < -L$ and $L < k \leq 2L$ such that $(i,j,k) \in \Omega$, and so $p_i \in \ell$ for all these $i$ too. Thus both $\ell$ and $\sigma$ contain at least $L$ of the points $p_i, q_j, r_k$ in this case also.
\end{proof}

We remark that this analysis can be pushed further in order to say something about the distribution of the points $p_i, q_j, r_k$ on (for example) three lines $\ell, \ell',\ell''$. One could most probably give some kind of complete classification of (say) $100 \times 100 \times 100$ triangular grids. However it is also possible to take a self-contained additive-combinatorial approach, leading to better bounds, and this is the technique we pursue in Section \ref{somewhat-collinear}. 

To conclude this section let us remark that a number of beautiful pictures of triangular structures arising from cubic curves (for various different types of cubic) may be found in the paper \cite{carnicer-gasca}, another work in the interpolation theory literature. 

\section{Almost triangular structure and covering by cubics}\label{section-5}

Recall that if $P \subset \R\P^2$ is a set of points then $\Gamma_P$ is the graph defined by the dual lines $p^*$, $p \in P$. We now know (Proposition \ref{bad-edges-prop}) that if $P$ has few ordinary lines then $\Gamma_P$ has a highly triangular structure. We also understand (Lemma \ref{tri-cubic}) that triangular structure in $\Gamma_P$ corresponds to points of $P$ lying on a cubic curve. In this section we put these facts together to prove some of the structural results stated in the introduction. 

The main result is Lemma \ref{cubic-covering-tech}, whose statement and proof are somewhat technical. To convey the main idea (and because we will need it later, and because it may be of independent interest) we first establish the following much easier result. This result also comes with better bounds -- indeed it says something even if one only knows that $P$ spans $o(n^2)$ ordinary lines -- than our more technical later result.

\begin{proposition}[Cheap structure theorem]\label{basic-cubic-covering}
Suppose that $P$ is a finite set of $n$ points in the plane. Suppose that $P$ spans at most $Kn$ ordinary lines for some $K \geq 1$.  Then $P$ lies on the union of $500K$ cubic curves.
\end{proposition}

\begin{proof}  We first dispose of a degenerate case.  Suppose that one of the dual lines $p^*$, $p \in P$, meets fewer than $500K$ points in $\Gamma_P$.  Then every dual line meets one of these points, which means that $P$ is covered by at most $500K$ lines.  As every line is already a cubic curve, we are done in this case.  Thus we may assume that each dual line $p^*$ meets at least $500K$ points in $\Gamma_P$.  In particular, it meets at least three points of $\Gamma_P$.

Recall the definition of a ``good edge'' of $\Gamma_P$: an edge both of whose vertices have degree $6$, and where both faces adjoining it are triangles. 

Let us say that an edge is \emph{really good} if all paths of length two from both of its endpoints consist entirely of good edges. If we have a segment $S$ of $l \geq 1$ consecutive edges on $p^*$, all of which are really good, then the structure of $\Gamma_P$ is locally that of a triangular grid with dimensions $\{-2,\ldots,2\}$, $\{-l-4,\ldots,-1\}$, $\{1,\ldots,l+4\}$; note that the distinct intersection property of Definition \ref{trio}(i) is automatic since every dual line is assumed to meet at least three points in $\Gamma_P$.  Applying Lemma \ref{tri-cubic}, we conclude that if $S$ is such a segment of consecutive really good edges, containing at least one edge, then the set of $q \in P \setminus \{p\}$ for which $q^*$ meets $S$ all lie on a cubic curve $\gamma_S$ (which also contains $p$). 

If an edge is not really good, we say that it is \emph{somewhat bad}.  We know, by Proposition \ref{bad-edges-prop}, that the number of bad edges is at most $16 K n$. Now associated to any somewhat bad edge $e$ is a path of length $1$, $2$ or $3$ whose first edge is $e$ and whose last edge is bad, and which is furthermore the only bad edge on that path (take a minimal path starting in $e$ and ending in a bad edge). The number of paths of length $3$ of the form bad-good-good is at most $16 K n \times 5 \times 5$, since each vertex of a good edge has degree $6$. Taking account of paths of length $2$ and 1 as well, we obtain an upper bound of $500 K n$ for the number of somewhat bad edges.  

By the pigeonhole principle there is a line $p^*$ which contains $t  \leq 500 K$ somewhat bad edges. These somewhat bad edges partition $p^*$ into $t$ segments of consecutive really good edges (a segment may have length zero). Let the segments with at least one edge be $S_1,\dots, S_{t'}$, and let the segments of length zero, which are simply vertices, consist of vertices $v_{t'+1}, \dots, v_{t}$. 

If $q \in P \setminus \{p\}$, then $q^*$ meets $p^*$ either in a vertex of one of the $S_i$, or in one of the additional vertices $v_j$.  In the former case, as discussed previously, Lemma \ref{tri-cubic} places $q$ in a cubic curve $\gamma_{S_i}$ depending on $S_i$.  In the latter case, $q$ lies in the dual line $v_j^*$.  Such a dual line can be thought of as a (degenerate) cubic curve. Taking the union of all these cubic curves, of which there are at most $t' + (t-t') \leq 500K$, gives the result.
\end{proof}

Proposition \ref{basic-cubic-covering} is already a fairly strong structure theorem for sets with few ordinary lines. It is possible that the ordinary lines in a union of $O(1)$ cubics can be analysed directly, though this certainly does not seem to be straightforward. Fortunately, there is much more to be extracted from Proposition \ref{bad-edges-prop} and the results of Section \ref{dual-cubic}, enabling us to prove more precise statements that  refine Proposition \ref{basic-cubic-covering}, albeit with somewhat worse explicit constants. 

The next lemma is the main technical result of this section. In an effort to make the paper more readable, we have formulated it so that, once it is proven, we will have no further need of the dual graph $\Gamma_P$ and consequences of Melchior's inequality.

\begin{lemma}\label{cubic-covering-tech}  
Suppose that $P$ is a set of $n$ points in the plane. Suppose that $P$ spans at most $Kn$ ordinary lines for some $K \geq 1$, and let $L \geq 10$ be a parameter.   Suppose that $P$ cannot be covered by a collection of $4L$ concurrent lines.  Then for every $p \in P$ there is a partition $P = \{p\} \cup \Sigma_{1,p} \cup \dots \cup \Sigma_{c_p, p}$ with the following properties:
\begin{enumerate}
\item For $i = 1,\dots,c_p$ the points of $\Sigma_{i,p}$ lie on a \textup{(}not necessarily irreducible\textup{)} curve $\gamma_{i,p}$ of degree at most three, which also contains $p$;
\item If $\gamma_{i,p}$ is not a line, then each irreducible component of it contains at least $L$ points of $P$;
\item If $\gamma_{i,p}$ is not a line, then the points of $\Sigma_{i,p}$ may be partitioned into pairs $(q,r)$ such that $p,q,r$ are collinear, and no other points of $P$ are on the line joining $p,q$ and $r$;
\item We have the upper bound $\sum_{p \in P} c_p \leq 2^{19} L^3 K n$ on the average size of $c_p$.
\end{enumerate}
\end{lemma}

\begin{proof}
The proof of this lemma is basically the same as the last one, except now we work with a considerably enhanced notion of what it means to be a ``really good edge'' . Call an edge \emph{extremely good} if all paths of length $2L$ from both of its endpoints consist entirely of good edges. In the last proposition, we only needed paths of length $2$. If an edge $e$ is not extremely good, let us say that it is \emph{slightly bad}. 

We now count the number of slightly bad edges by an argument similar to that used to prove the previous proposition.  Let $e$ be a slightly bad edge, and let $r$ be the length of the shortest path from an endpoint of $e$ to a vertex of a bad edge. Thus $0 \leq r \leq 2L-1$, and there is a vertex $v$ of a bad edge $e'$ that is at distance exactly $r$ from an endpoint $w$ of $e$.  Then all paths of length up to $r$ from either of the vertices of $e$ are good, which means that the $r$-neighbourhood of $e$ has the combinatorial structure of a triangular grid, and also that $v$ lies on the boundary of this neighbourhood and has degree six.  Among other things, this implies that among all the paths of length $r$ from $v$ to $w$, there is a path that changes direction only once.  To describe this path, as well as the slightly bad edge $e$, one could specify the bad edge $e'$, followed by an endpoint $v$ of that bad edge of degree six, followed by an edge emanating from $v$, which is followed along for some length $r_1$ to a vertex of degree six, at which point one switches to one of the other four available directions and follows that direction for a further length $r_2$, with $r_1+r_2 \leq r$ (so in particular $0 \leq r_1,r_2 \leq 2L-1$), until one reaches a vertex $w$, at which point the slightly bad edge $e$ is one of the six edges adjacent to $w$.  From Proposition \ref{bad-edges-prop} and simple counting arguments, we may thus bound the total number of slightly bad edges $e$ crudely by
$$ 16 K n \times 2 \times 6 \times 2L \times 4 \times 2L \times 6 \leq 2^{15} K L^2n$$
and so we conclude that the number of slightly bad edges is at most $2^{15} K L^2n$.  One could save a few powers of two here by being more careful, but we will not do so.

%\begin{figure}\label{change}
%\caption{A minimal-length path that changes direction at most once.}
%\end{figure}

Suppose that there are $b_p$ slightly bad edges on $p^*$. Then
\begin{equation}\label{bp-bound} \sum_{p \in P} b_p \leq 2^{15} L^2 K n.\end{equation}

If we have a segment of $m \geq 10L$ consecutive edges on $p^*$, all of which are extremely good, and with $p^*$ containing at least $4L$ additional edges beyond these $m$, then the structure of $\Gamma_P$ is locally that of a triangular grid of dimensions $\{-2L,\ldots,2L\}$, $\{-m,\ldots,-1\}$, $\{1,\ldots,m\}$.  Note that as we are assuming that $P$ cannot be covered by $4L$ concurrent lines, every dual line $p^*$ meets at least $4L+1$ distinct points in $\Gamma_P$, ensuring the disjointness property in Definition \ref{trio}(ii).  Indeed, the fact that each dual line meets at least $4L+1$ distinct points, and that $p^*$ contains at least $4L$ additional edges beyond the $m$ consecutive edges, ensures that all the lines in this triangular grid are distinct.

Thus if $S$ is such a segment then, by Lemma \ref{tri-cubic-detailed}, the set $\Sigma_S$ of all $q \in P \setminus \{p\}$ for which $q^*$ meets $S$ all lie on a cubic curve $\gamma_S$ which contains $p$, and each component of which contains at least $L$ points of $P$. Furthermore, since the lines $q^*$ meet the vertices of $S$ in pairs (since each such vertex certainly has degree $6$) the points of $\Sigma_S$ may be divided into pairs $(q,r)$ such that $p,q,r$ are collinear, and no other point of $P$ lies on the line joining $p,q$ and $r$. Compare with conclusion (iii) of this lemma.

The line $p^*$ is divided into $b_p$ segments, each containing one or more vertices, by the slightly bad edges.  We then create some subsegments $S_1,\ldots,S_t$ by the following rule:

\begin{enumerate}
\item If $p^*$ contains at most $14L$ edges in all, then we set $t=0$, so no subsegments $S_1,\ldots,S_t$ are created;
\item If $p^*$ contains more than $14L$ edges, and one of the segments $S$ cut out by the slightly bad edges contains all but at most $4L$ of the edges, we set $t=1$, and define $S_1$ to be a subsegment of $S$ omitting precisely $4L$ edges;
\item In all other cases, we set $S_1,\dots, S_t$ be those segments cut out by the slightly bad edges with at least $10L$ edges.
\end{enumerate}

We then let $v_{t + 1}, \dots, v_{c_p}$ be the vertices not contained in any of the $S_1,\ldots,S_t$.  By construction, we see that we always have $t \leq b_p$ and that the number $c_p-t$ of remaining vertices $v_i$ is at most $\max( 14L, 4L, (10L+1)b_p) \leq (10L+1)b_p + 14L$.  We thus have
\begin{equation}\label{cp-bound} c_p \leq (10L + 2)b_p + 14L.\end{equation} 
Define $\Sigma_{i,p} := \Sigma_{S_i}$ and $\gamma_{i,p} := \gamma_{S_i}$ for $i \leq t$, and for $i \geq t+1$ let $\gamma_{i,p}$ be the line $v^*_i$ and take $\Sigma_{i,p}$ to consist of the points of $P \setminus \{p\}$ lying on this line.

This collection of cubics and lines has properties (i), (ii) and (iii) claimed in the lemma. The bound (iv) follows immediately from \eqref{bp-bound} and \eqref{cp-bound} and the crude bound $(10L+2) 2^{15} L^2 K + 14L \leq 2^{19} L^3 K$, valid for $L \geq 10$ and $K \geq 1$.
\end{proof}

We are now in a position to prove a result which is still not quite as strong as our main structure theorem, Theorem \ref{main-structure-theorem}, but is still considerably more powerful (albeit with worse explicit constants) than the rather crude statement of Theorem \ref{basic-cubic-covering}. 

\begin{proposition}[Intermediate structure theorem]\label{intermediate}
Suppose that $P$ is a finite set of $n$ points in the plane. Suppose that $P$ spans at most $Kn$ ordinary lines for some $K \geq 1$. Then one of the following three alternatives holds:
\begin{enumerate}
\item $P$ lies on the union of an irreducible cubic $\gamma$ and an additional $2^{75} K^5$ points.
\item  $P$ lies on the union of an irreducible conic $\sigma$ and an additional $2^{64} K^4$ lines. Furthermore, $\sigma$ contains between $\frac{n}{2} - 2^{76} K^5$ and $\frac{n}{2} + 2^{76} K^5$ points of $P$, and $P \setminus \sigma$ spans at most $2^{62} K^4 n$ ordinary lines.
\item $P$ lies on the union of $2^{16} K$ lines and an additional $2^{87} K^6$ points.
\end{enumerate}
\end{proposition}

\emph{Remark.} The explicit expressions such as $2^{75} K^5$ in the above proposition could of course be replaced by the less specific notation $O(K^{O(1)})$ if desired, and the reader may wish to do so in the proof below as well.\vspace{11pt}

\begin{proof}  If $P$ can be covered by $60000K \leq 2^{16} K$ concurrent lines then we are of course done, so we will assume that this is not the case.

By Proposition \ref{basic-cubic-covering} we know that $P$ is covered by at most $500 K$ cubic curves.  By breaking each of these curves up into irreducible components, we may thus cover $P$ by distinct irreducible cubic curves $\gamma_1,\dots,\gamma_m$ for some
\begin{equation}\label{m-def}
m \leq 1500K.
\end{equation}
By B\'ezout's Theorem, no pair of distinct irreducible curves intersects in more than $9$ points, and so there is a set $P' \subset P$, with
$$ |P \backslash P'| \leq 9 \binom{m}{2} \leq 2^{24} K^2,$$
such that each point of $P'$ lies on just one of the curves $\gamma_i$. 

Suppose first of all that one of the $\gamma_i$, say $\gamma_1$, is an irreducible cubic and contains at least $2^{76} K^5$ points of $P$. Then it also contains at least $2^{75} K^5$ points of $P'$.  Write $n_0 := |P' \cap \gamma_1|$: thus $n_0 \geq 2^{75} K^5$. 

By construction and \eqref{m-def}, $P$ is not covered by $40m$ concurrent lines.  Applying Lemma \ref{cubic-covering-tech} with $L := 10m$, we see that for each $p' \in P$ we may partition $P$ as $\{p'\} \cup \Sigma_{1,p'} \cup \dots \cup \Sigma_{c_{p'},p'}$, where $\sum_{p' \in P} c_{p'} \leq 2^{19} (10 m)^3 K n \leq 2^{61} K^4 n$ and each $\Sigma_{i,p'}$ is contained in some (not necessarily irreducible) cubic $\gamma_{i,p'}$ containing $p'$ which is either a line, or has the property that each irreducible component of it contains at least $10m$ points of $P$. 

By the pigeonhole principle, there is some $p' \in P' \cap \gamma_1$ with the property that $c_{p'} \leq 2^{61} K^4 n/n_0$. Fix this $p'$.  By B\'ezout's theorem, an irreducible curve of degree at most three that is not already one of the $\gamma_j$ meets $P$ in no more than $9m$ points, and so we infer that each $\gamma_{i,p'}$ is either a line, or else every irreducible component of it is one of the $\gamma_j$. Since $p'$ lies on $\gamma_1$ but not on any other $\gamma_i$, we infer that all the $\gamma_{i,p'}$ are lines except that one of them, say $\gamma_{1,p'}$, may be $\gamma_1$. Furthermore, none of the lines $\gamma_{j,p'}$, $j = 2,\dots,c_{p'}$, which all contain $p'$, coincides with any of $\gamma_2,\dots,\gamma_m$. By another application of B\'ezout's theorem, each of them contains at most $3m$ points of $P$.  

It follows that 
\begin{align*}
n = |P|   & \leq  |P \cap \gamma_1| + \sum_{j = 2}^{c_{p'}} |P \cap \gamma_{j,p'}| \\ 
& \leq  n_0 + 2^{24} K^2 + 3mc_{p'}  \leq n_0 + \frac{2^{74} K^5 n}{n_0}.
\end{align*}
Since $n_0 \geq 2^{75} K^5$, we conclude that $n_0 \geq n/2$, which when inserted again into the above inequality gives $n_0 \geq n - 2^{75} K^5$, which is option (i) of Proposition \ref{intermediate}. 

The analysis of option (ii) goes along similar lines but is a little more complicated. Suppose now that one of the $\gamma_i$, say $\gamma_1$, is an irreducible conic and contains at least $2^{76} K^5$ points of $P$.  Once again, it also contains at least $2^{75} K^5$ points of $P'$. Write $n_0 := |P' \cap \gamma_1|$; thus $n_0  \geq 2^{75} K^5$.

By Lemma \ref{cubic-covering-tech} as before we may, for each $p' \in P$, partition $P$ as $\{p'\} \cup \Sigma_{1,p'} \cup \dots \cup \Sigma_{c_{p'},p'}$ with $\sum_{p'} c_{p'} \leq 2^{19} (10 m)^3 Kn \leq 2^{61} K^4 n$ and each $\Sigma_{i,p'}$ contained in some (not necessarily irreducible) curve $\gamma_{i,p'}$ of degree at most three containing $p'$ which is either a line, or has the property that each irreducible component contains at least $10 m$ points of $P$.

By the pigeonhole principle as before, we may find $p' \in P' \cap \gamma_1$ such that $c_{p'} \leq 2^{61} K^4 n/n_0$.  Now fix this $p'$.

Suppose that $\gamma_{i,p'}$ is not a line. Then, by B\'ezout's theorem as above, each irreducible component of $\gamma_{i,p'}$ is one of the $\gamma_j$. Since $p' \in \gamma_{i,p'}$, and $p'$ lies on $\gamma_1$ but not on any other $\gamma_j$, one of the irreducible components of $\gamma_{i,p'}$ is $\gamma_1$. Thus for each $i$ one of the following is true:
\begin{enumerate}
\item $\Sigma_{i,p'}$ is contained in a line through $p'$;
\item $\Sigma_{i,p'}$ is contained in the conic $\gamma_1$;
\item $\Sigma_{i,p'}$ is contained in the union of the conic $\gamma_1$ and a line $\gamma_{j_i}$.\end{enumerate}
Now recall Lemma \ref{cubic-covering-tech}. Item (iii) of that lemma asserts that in cases (ii) and (iii) above the points of $\Sigma_{i,p'}$ may be divided into collinear triples $(p',q,r)$. This immediately rules out option (ii). For those $i$ satisfying (iii) we see that $|\Sigma_{i,p'}| = 2|\Sigma_{i,p'} \cap \gamma_1|$. For those $i$ satisfying (i) it follows from B\'ezout's theorem that $|\Sigma_{i,p'}| \leq 3m$.

Let $I$ be the set of indices $i$ for which $\Sigma_{i,p'}$ is not contained in a line through $p'$, that is to say for which option (iii) above holds.
It follows that
\[ 
n = |P|  = 1 + \sum_{i = 1}^{c_{p'}} |\Sigma_{i,p'}| \leq 1 + 2\sum_{i \in I} |\Sigma_{i,p'} \cap \gamma_1| + 3mc_{p'}
\]
However, any line through $p'$ meets $\gamma_1$ (which contains $p'$) in at most one other point, and so
\[ \sum_{i \in I} |\Sigma_{i,p'} \cap \gamma_1| \leq |P \cap \gamma_1| + c_{p'}.\]
Since
$$ |P \cap \gamma_1| \leq |P' \cap \gamma_1| + |P \backslash P'| \leq n_0 + 2^{24} K^2$$
we conclude that
\[ n \leq 2n_0 + 2^{25} K^2 + (3m+2) c_{p'} + 1 \leq 2n_0 + \frac{2^{74} K^5 n}{n_0}.\]
Since $n_0 \geq 2^{75} K^5$, this is easily seen to imply that $n_0 \geq n/4$, and hence $n_0 \geq n/2 - 2^{76} K^5$ and $c_{p'} \leq 2^{63} K^4$. In the converse direction, we have
$$ n \geq \sum_{i \in I} |\Sigma_{i,p'}| = 2\sum_{i \in I} |\Sigma_{i,p'} \cap \gamma_1| $$
and so
$$ |P \cap \gamma_1| \leq \frac{n}{2} + 1 + |\bigcup_{i \not \in I} \Sigma_{i,p'} \cap \gamma_1|.$$
For $i \not \in I$, $\Sigma_{i,p'} \cap \gamma_1$ consists of at most one point, and so
$$ |P \cap \gamma_1| \leq \frac{n}{2} + 1 + c_{p'} \leq \frac{n}{2} + 2^{76} K^5$$
and hence $\gamma_1$ contains between $\frac{n}{2} - 2^{76} K^5$ and $\frac{n}{2} - 2^{76} K^5$ elements of $P$.

Looking back at the three possibilities (i), (ii) and (iii) above,  we see that the other points lie in the union of the lines $\gamma_{i,p'}$ and $\gamma_j$, of which there are at most $c_{p'} + m \leq 2^{64} K^4$. 

To complete the proof that we are in case (ii) claimed in the proposition, we need to give an upper bound for the number of ordinary lines spanned by the set $P \setminus \gamma_1$. Such a line could be ordinary in $P$, but there are at most $Kn$ such lines. Otherwise, such a line passes through a point $p' \in P \cap \gamma_1$, and contains precisely two points in $P \setminus \gamma_1$. Let us say that such a line is \emph{bad}. The number of bad lines arising from $p' \in P \setminus P'$, a set of cardinality at most $2^{24} K^2$, is at most $2^{23} K^2 n$. Suppose then that $p' \in P' \cap \gamma_1$. As above, for each such $p'$ we have a partition $P = \{p'\} \cup \Sigma_{1,p'} \cup \dots \cup \Sigma_{c_{p'},p'}$, and we now know that $\Sigma_{i,p'}$ is either contained in a line through $p'$ or is contained in the union of $\gamma_1$ and a line. Furthermore in the latter case we know from Lemma \ref{cubic-covering-tech} (iii) that every line though $p'$ and a point of $\Sigma_{i,p'}$ passes through precisely two other points of $P$, one on $\gamma_1$ and the other not. Therefore it is not bad. The number of bad lines through $p'$ is thus at most $c_{p'}$, and so the total number of bad lines arising from $p' \in P \cap \gamma_1$ is at most $\sum_{p'} c_{p'} \leq 2^{61} K^4 n$. Statement (ii) of the proposition follows immediately.

We have now considered all cases in which any irreducible cubic or conic from amongst the $m$ curves $\gamma_i$ contains more than $2^{76} K^5$ points of $P$. If this is not the case, the only curves among the $\gamma_i$ containing more than $2^{76} K^5$ points of $P$ are lines. Thus $P$ may be covered by $m \leq 2^{16} K$ lines and at most $2^{76} K^5 m \leq 2^{87} K^6$ points, which gives option (iii).
\end{proof}

\section{Unions of lines}\label{somewhat-collinear}

Suppose that $P$ is a set of $n$ points spanning at most $Kn$ ordinary lines. We know from Proposition \ref{intermediate} that all but $O(K^{O(1)})$ points of $P$ lie on an irreducible cubic, an irreducible conic and some lines, or some lines. The aim of this section is to reduce the number of lines, in all cases, to at most one. The main result of this section is the following theorem, which may again be of independent interest.

\begin{proposition}\label{main-line-prop}
Suppose that a set $P \subset \R\P^2$ of size $n$ lies on a union $\ell_1 \cup \dots \cup \ell_m$ of lines, and that $P$ spans at most $Kn$ ordinary lines. Suppose that $n \geq n_0(m,K)$ is sufficiently large. Then all except at most $3K$ of the points of $P$ lie on a single line.
\end{proposition}

We first handle the (easy) case where one line has almost all the points. For this, one does not need to know that $P$ is contained in the union of a few lines.

\begin{lemma}\label{big-line-lem}
Suppose that $P \subset \R\P^2$ is a set of size $n$, that $P$ spans at most $Kn$ ordinary lines, and that at least $\frac{2}{3}n$ of the points of $P$ lie on a single line $\ell$. Then in fact all except at most $3K$ of the points lie on $\ell$.
\end{lemma}
\begin{proof}
Let $p \in P \setminus \ell$. Then $p$ forms at least $2n/3$ lines with the points of $P \cap \ell$. At most $n/3$ of these contain another point of $P$, and so at least $n/3$ of them are ordinary. Therefore the number of ordinary lines is at least $|P \setminus \ell|n/3$, and the claim follows immediately.
\end{proof}

Suppose now, and for the rest of the section, that $P \subset \ell_1 \cup \dots \cup \ell_m$.
The opposite extreme to that considered by the above lemma is when \emph{all} the lines $\ell_i$ contain many points of $P$. The next result, which is the key technical step in the proof of Proposition \ref{main-line-prop} (and is in fact rather stronger than that proposition), may be of independent interest. 

\begin{proposition}\label{many-rich}
Suppose that $m \geq 2$ and that a set $P \subset \R\P^2$ of size $n$ lies on a union $\ell_1 \cup \dots\cup\ell_m$ of lines, and that at least $\eps n$ points of $P$ lie on each of the lines $\ell_i$.  Suppose that $m, \frac{1}{\eps} \leq n^{\frac{1}{10000}}$. Then $P$ spans at least $\eps^{12} n^2/m^6$ ordinary lines.
\end{proposition}
\emph{Remark.} The exponent $\frac{1}{10000}$ could certainly be improved somewhat, but a really significant improvement -- beyond $\frac{1}{100}$, say -- would require new methods.

The proof of this proposition is quite long. Before embarking upon it we show how to derive Proposition \ref{main-line-prop} as a consequence.\vspace{11pt}

\emph{Deduction of Proposition \ref{main-line-prop} from Proposition \ref{many-rich} and Lemma \ref{big-line-lem}}. Reorder the lines so that $n_1 \geq n_2 \geq \dots \geq n_m$, where $n_i := |P \cap \ell_i|$. Set $\eps_j := n_j/n$. If $\eps_2 \leq 1/3m$ then $\eps_1 \geq 2/3$ and we are done by Lemma \ref{big-line-lem}, so suppose that $\eps_2 \geq 1/3m$. Write $P_j := P \cap (\ell_1 \cup \dots \cup \ell_j)$, $j = 2,3,\dots, m$. By Proposition \ref{many-rich}, the set $P_j$ determines at least $\eps_j^6 n^2/m^5$ ordinary lines (here we have used the trivial lower bound $|P_j| \geq |P_1| \geq n/m$). Since $|P \setminus P_j| \leq m \eps_{j+1}n$, the number of these which fail to be ordinary lines in $P$ is bounded above by $m \eps_{j+1} n^2$. Let $j$ be the least index such that $\eps_{j+1} < \frac{1}{2}(\eps_j/m)^{6}$ or, if there is no such index, set $j := m$. Then it follows that the number of ordinary lines in $P$ is at least $\frac{1}{2}(\eps_j/m)^{6} n^2$. We have the lower bound $\eps_j \geq \exp(-e^{Cm})$, and so $P$ spans $\gg \exp (-e^{Cm})n^2$ ordinary lines. If $n \geq n_0(m,K)$ is sufficiently large this is greater than $Kn$, so we obtain a contradiction. \endproof

\emph{Remark.} We note that $n_0(m,K)$ can be taken to have the shape $n_0(m,K) \sim K \exp\exp (Cm)$.\vspace{11pt}

We may now focus our attention on establishing Proposition \ref{many-rich}. We will divide into two quite different cases, according as the lines $\ell_i$ all intersect at a point or not. 

\begin{proposition}\label{betts-case}
Suppose that $m \geq 2$ and that $P$ lies on  a union $\ell_1 \cup \dots\cup\ell_m$ of lines, all of which pass through a point, and that at least $\eps n$ points of $P$ lie on each of the lines $\ell_i$. Then $P$ spans at least $\eps^2n^2/50$ ordinary lines.
\end{proposition}
\begin{proof} We are greatly indebted to Luke Alexander Betts, a second year undergraduate at Trinity College, Cambridge, who showed us the following argument. For brevity we give a slightly crude version of the argument he showed us. 

Applying a projective transformation, we may assume without loss of generality that all the lines pass through the origin $[0,0,1]$ in the affine part of $\R\P^2$. Dualising, these $m$ lines become points on the line at infinity, and the sets $P \cap \ell_i$ become sets of parallel lines. If $\mathcal{L}$ is a set of lines in $\R^2$, we say that a point is \emph{ordinary} for $\mathcal{L}$ if it lies on precisely two of the lines in $\mathcal{L}$.  We say that $\mathcal{L}$ is $t$-parallel if, for every line $\ell \in \mathcal{L}$, there are at least $t$ other lines parallel to it. Finally, we say that a point lying on three or more of the lines from $\mathcal{L}$ is a \emph{triple point}. The dual statement to Proposition \ref{betts-case} (with $t$ replacing $\eps n$) is then the following.\vspace{11pt}

\noindent\textbf{Proposition.} Let $t>0$ be a real number.  \emph{Suppose that $\mathcal{L}$ is a $t$-parallel set of lines in $\R^2$, and that not all the lines of $\mathcal{L}$ are parallel. Then there are at least $t^2/50$ ordinary points for $\mathcal{L}$. }\vspace{11pt}

The heart of the matter is the following lemma.

\begin{lemma}\label{ordinary}
Suppose that $\mathcal{L}$ is $t$-parallel, but not all the lines of $\mathcal{L}$ are parallel. Then there is a line $\ell \in \mathcal{L}$ containing at least $t/2$ ordinary points for $\mathcal{L}$.
\end{lemma}
\begin{proof}
If there are no triple points determined by $\mathcal{L}$ then the conclusion is immediate, as every line intersects at least $t+1 > t/2$ other lines. If there are triple points determined by $\mathcal{L}$, let $T$ be the set of them. Let $v$ be a vertex of the convex hull of $T$, lying on lines $\ell_1,\ell_2,\ell_3 \in \mathcal{L}$. There are open rays (half-lines) $\ell_1^+,\ell_2^+,\ell_3^+$ emanating from $v$ which do not intersect $T$. Suppose without loss of generality that the rays $\ell_1^+, \ell_3^+$ lie on either side of the line $\ell_2$, as depicted in Figure \ref{betts-diagram}. Each of the $t$ other lines in $\mathcal{L}$ parallel to $\ell_2$ meets either $\ell_1^+$ or $\ell_3^+$, and so at least one of these rays meets at least $t/2$ lines in $\mathcal{L}$. All of these points of intersection must, by construction, be ordinary points for $\mathcal{L}$.
\end{proof}

\begin{figure}\includegraphics{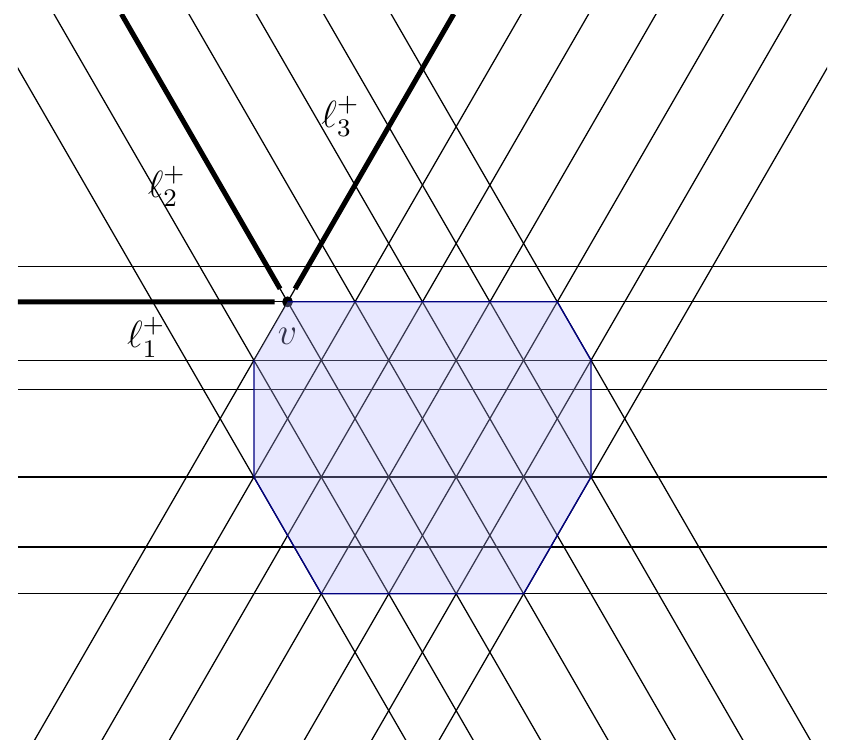}
\caption{Figure relevant to the proof of Lemma \ref{ordinary}. Here, $t = 6$ and the ray $\ell_3^+$ meets $4 > 6/2$ lines in $\mathcal{L}$. All of these points of intersection are ordinary as they lie outside the convex hull of the triple points of $\mathcal{L}$, shaded in blue.}
\label{betts-diagram}
\end{figure}
 
 Now let us return to the main problem, the dual form of Proposition \ref{betts-case} stated above. Note that if $\mathcal{L}'$ is formed by removing at most $t/5$ lines from $\mathcal{L}$ then it is still $4t/5$-parallel. Thus, by $\lceil t/5 \rceil$ applications of Lemma \ref{ordinary} we may inductively find distinct lines $\ell_1,\dots, \ell_{\lceil t/5\rceil}$ such that $\ell_i$ contains at least $2t/5$ ordinary points for $\mathcal{L} \setminus \{\ell_1,\dots, \ell_{i-1}\}$. These are not necessarily ordinary points for $\mathcal{L}$, but any such point that is not lies on one of $\ell_1,\dots,\ell_{i-1}$. Since it also lies on $\ell_i$, there are at most $i-1 < t/5$ such points, and so $\ell_i$ contains at least $t/5$ ordinary points of $\mathcal{L}$. Each ordinary point of $\mathcal{L}$ lies on at most two of the lines $\ell_i$, so we get at least $t^2/50$ ordinary points in total.
\end{proof}

We have now established Proposition \ref{betts-case}, which is the particular case of Proposition \ref{many-rich} in which all the lines $\ell_i$ pass through a single point. We turn now to the case in which this is not so. The next proposition, together with Proposition \ref{betts-case}, immediately implies Proposition \ref{many-rich} and hence the main result of the section, Proposition \ref{main-line-prop}.

\begin{proposition}\label{not-betts-case}
Suppose that $m \geq 2$ and that a set $P \subset \R\P^2$ of size $n$ lies on a union $\ell_1 \cup \dots\cup\ell_m$ of lines, not all of which pass through a single point, and that at least $\eps n$ points of $P$ lie on each of the lines $\ell_i$.  Suppose that $m, \frac{1}{\eps} \leq n^{\frac{1}{10000}}$. Then $P$ spans $\gg\eps^{12} n^2/m^6$ ordinary lines.\end{proposition}

The proof proceeds via several lemmas. It also requires 
some additive-combinatorial ingredients not needed elsewhere in the paper, which we collect in the appendix. It is convenient, for this portion of the argument, to work entirely in the affine plane. Let us begin, then, by supposing that a projective transformation has been applied so that all lines $\ell_i$ and their intersections lie in the affine plane $\R^2$. 

Suppose that $\ell$ is a line. Then by a \emph{ratio map} on $\ell$ we mean a map 
\[ \psi = \psi_{q,q'}: \ell \rightarrow \R \cup \{\infty\}\] of the form
\[ \psi_{q,q'}(p) = \frac{\length(pq)}{\length(pq')},\] where $q,q'$ are distinct points on $\ell$ and the lengths $\length(p q),\length( p q')$ are \emph{signed} lengths on $\ell$. 
We say that the ratio maps $\psi_{q,q'}$ and $\psi_{q',q}$ are \emph{equivalent}, but otherwise all ratio maps are deemed inequivalent.  Note that such ratio maps implicitly appeared in the analysis of the infinite near-counterexample \eqref{p2}.

An ordered triple of lines $\ell_i, \ell_j, \ell_k$ not intersecting in a single point defines two ratio maps $\phi_{i,j,k} : \ell_i \rightarrow \R \cup \{\infty\}$ and $\tilde\phi_{i,j,k} : \ell_j \rightarrow \R \cup \{\infty\}$ via
\[ \phi_{i,j,k} := \psi_{\ell_i \cap \ell_j, \ell_i \cap \ell_k}  \qquad \mbox{and} \qquad \tilde\phi_{i,j,k} := \psi_{\ell_i \cap \ell_j, \ell_j \cap \ell_k}.\]
We will make considerable use of these maps in what follows, as well as of the following definition.

\begin{definition}[Quotient set]
Suppose that $X \subset \R \cup \{\infty\}$ is a set. Then we write $\mathcal{Q}(X)$ for the set of all quotients $x_1/x_2$ with $x_1, x_2 \in X$ and $x_1, x_2 \notin \{0,\infty\}$. 
\end{definition}

The following definition depends on the parameter $n$, which is the number of points in the set $P$. Since this is fixed throughout the section, we do not indicate dependence on it explicitly.

\begin{definition}
Let $A$ be a finite subset of some line $\ell$, and let $\psi$ be a ratio map on $\ell$. Then we say that $A$ is a $\psi$-grid if $A$ is a union of at most $n^{\frac{1}{30}}$ sets $S$ such that $|\mathcal{Q}(\psi(S))| \leq n^{1+ \frac{1}{10}}$. We say that $A$ is a \emph{grid} if it is a $\psi$-grid for some ratio map $\psi$ on $\ell$.
\end{definition}

\begin{lemma}\label{lem7.9}
Let $\ell$ be a line, and suppose that $\psi$ and $\psi'$ are inequivalent ratio maps on $\ell$. Suppose that $A$ is a $\psi$-grid and that $A'$ is a $\psi'$-grid. Then $|A \cap A'| \ll n^{1 - \frac{1}{25}}$.
\end{lemma}
\begin{proof} Without loss of generality we may assume that $\ell$ is the $x$-axis parametrised as $\{(t,0) : t \in \R\}$. By abuse of notation we identify $A$ and $A'$ with subsets of $\R$. The ratio maps $\psi, \psi'$ are given by $\psi(t) = (t+a)/(t+b)$, $\psi'(t) = (t + a')/(t + b')$ with $a \neq b, a' \neq b'$ and $\{a,b\} \neq \{a', b'\}$.  

Suppose that $A$ is a $\psi$-grid and that $A'$ is a $\psi'$-grid. Write $A = \bigcup_{i = 1}^{n^{1/30}} S_i$ and $A' = \bigcup_{j = 1}^{n^{1/30}} S'_j$ where $|\mathcal{Q}(\psi(S_i))|$ and $|\mathcal{Q}(\psi'(S'_j))|$ are both at most $n^{1 + \frac{1}{10}}$. It suffices to show that $|S_i \cap S'_j| \ll n^{\frac{22}{25}}$.

Suppose that $X$ is the set of all $x \in S_i \cap S'_j$ for which $\psi(x) > 0$. Since $X$ is contained in both $S_i$ and $S'_j$, $|\mathcal{Q}(\psi(X))|$ and $|\mathcal{Q}(\psi'(X))|$ are both at most $n^{1 + \frac{1}{10}}$.
Writing $Y := \{\log \psi(x) : x \in X\}$, we see that
\begin{equation}\label{yfy} |Y - Y|, |f(Y) - f(Y)| \leq 2n^{1 + \frac{1}{10}},\end{equation} where $f(y) = \log \circ \psi' \circ \psi^{-1} \circ \exp(y) = $
\begin{align*}  &  = \log ((b - a')e^y + a' - a) - \log((b - b')e^y + b' - a) \\ & = \log (Ae^y + B) - \log (A' e^y + B'),\end{align*} say, and $f(Y) := \{f(y) : y \in Y, f(y) \, \mbox{is defined}\}$. Note that we do not have $AA' = 0$ and $B B' = 0$ (there are four cases to consider).  
We thus compute
\[ f''(y) =  \frac{e^y (b - a)(b' - a') (A A' e^{2y} - B B')}{(A e^y + B)^2 (A' e^y + B')^2}.\]
This is continuous except when $e^y = -B'/A'$ or $-B/A$, and nonzero except when $e^{2y} = BB'/AA'$. It follows that $\R$ may be split into at most $4$ pieces on which $f$ is defined and strictly concave/convex. By Proposition \ref{enr-prop-piecewise}, this implies that at least one of $|Y - Y|, |f(Y) - f(Y)|$ has size $\gg |Y|^{5/4}$. Comparing with \eqref{yfy} we see that $|X| = |Y| \ll n^{\frac{22}{25}}$. 

An almost identical argument applies when $X$ is the set of all $x \in S_i \cap S'_j$ for which $\psi(x) < 0$, taking $Y := \{ \log (-\psi(x)) : x \in X\}$: now we have $f(y) = \log (-Ae^y + B) - \log (-A' e^y + B')$, but the rest of the argument is the same. 

Putting these two cases together gives $|S_i \cap S'_j| \ll n^{\frac{22}{25}}$, which is what we wanted to prove.\end{proof}

Ratio maps may be used in understanding the metric properties of intersections of lines as a consequence of Menelaus's theorem, illustrated in Figure \ref{menelaus-fig}.
\begin{figure}
\includegraphics[scale=0.8]{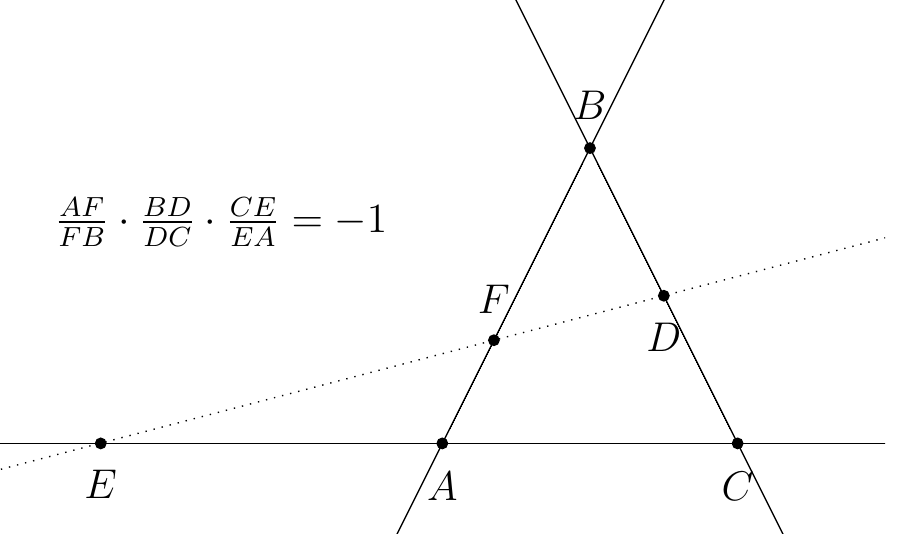} 
\label{menelaus-fig}
\caption{An illustration of Menelaus's theorem. The lengths are signed.}
\end{figure}
\begin{lemma}\label{intersect-lem}
Let $\ell_i, \ell_j, \ell_k$ be three lines not meeting at a point. Let $X_i \subset \ell_i$ and $X_j \subset \ell_j$, and let $\Gamma \subset X_i \times X_j$ be a set of pairs, with neither $X_i$ nor $X_j$ containing $\ell_i \cap \ell_j$.  Let $X_k \subset \ell_k$ be the set of points on $\ell_k$ formed by intersecting the lines $\overline{\{x_i ,x_j\}}$, $(x_i, x_j) \in \Gamma$, with $\ell_k$.  Then
\[ |X_k| =  | \{ \phi_{i,j,k}(x_i)/\tilde\phi_{i,j,k}(x_j) : (x_i, x_j) \in \Gamma\} | .\]
\end{lemma}
\begin{proof} Apply Menelaus's theorem with $AC = \ell_i$, $AB = \ell_j$ and $BC = \ell_k$. Suppose that $x_i \in \ell_i$ and $x_j \in \ell_j$, and write $E = x_i$, $F = x_j$ in the diagram. Then $D$ is the point at which $\overline{\{x_i,x_j\}}$ intersects $\ell_k$.  Note that $\phi_{i,j,k}(x_i) = EA/EC$, $\tilde\phi_{i,j,k}(x_j) = FA/FB$.
By Menelaus' theorem it follows that $\phi_{i,j,k}(x_i)/\tilde\phi_{i,j,k}(x_j) = DB/DC$. This ratio uniquely determines the point $D$, and the lemma follows.
\end{proof}

\begin{lemma}\label{lem7.10a} Suppose that $P$ is a set of $n$ points lying on a union $\ell_1 \cup \dots \cup \ell_m$ of lines, that the lines $\ell_i$ are not all concurrent, and that at least $\eps n$ points of $P$ lie on each line $\ell_i$. Suppose that $m, \frac{1}{\eps} \leq n^{\frac{1}{10000}}$. Then either $P$ spans at least $\eps^2n^2/8$ ordinary lines, or else there are at least two values of $i$ such that $P \cap \ell_i$ contains a grid with size $\gg \eps^4 n/m^2$.
\end{lemma}
\begin{proof} By the dual version of the Sylvester-Gallai theorem, there is some pair of lines $\ell_i, \ell_j$ such that no other line passes through $\ell_i \cap \ell_j$. 
Each of $\ell_i, \ell_j$ contains at least $\eps n$ points of $P$, at least $\eps n - 1 \geq \eps n/2$ of which are not the intersection point $\ell_i \cap \ell_j$. If $P$ spans fewer than $\eps^2 n^2/8$ ordinary lines then there is some $k$ such that for at least $\eps^2 n^2/8m$ pairs $p_i \in \ell_i, p_j \in \ell_j$ ($p_i, p_j \neq \ell_i \cap \ell_j$) the line $\overline{\{p_i ,p_j\}}$ meets $\ell_k$ in a point of $P$.  Write $X_i := (\ell_i \cap P) \setminus (\ell_i \cap\ell_j)$, $X_j := (\ell_j \cap P) \setminus (\ell_i \cap\ell_j)$ and $\Gamma \subset X_i \times X_j$ for the set of pairs $(p_i, p_j) \in X_i \times X_j$ for which $\overline{\{p_i, p_j\}}$ meets $\ell_k$ in a point of $X_k = \ell_k \cap P$. By Lemma \ref{intersect-lem} it follows that 
\[ n \geq |X_k| \geq | \{ \phi_{i,j,k}(x_i)/\tilde\phi_{i,j,k}(x_j) : (x_i, x_j) \in \Gamma\} | .\]
By Corollary \ref{bsg-cor} and the hypothesis on $m$ and $\frac{1}{\eps}$ it follows that there are sets $X'_i \subset X_i$, $X'_j \subset X_j$ with $|X'_i|, |X'_j| \gg \eps^4 n/m^2$ and\[ |\mathcal{Q}(\phi_{i,j,k}(X'_i))|, |\mathcal{Q}(\tilde\phi_{i,j,k}(X'_j) )|  \ll m^{11} n/\eps^{22}  < n^{1 + \frac{1}{10}}.\] Thus certainly
$X'_i$ is a $\phi_{i,j,k}$-grid and $X'_j$ is a $\tilde\phi_{i,j,k}$-grid. 
\end{proof}

As a result of this lemma we may, in proving Proposition \ref{not-betts-case}, restrict attention to sets where $P \cap \ell_i$ contains a large grid for at least two values of $i$. We study this situation further in the next lemma, whose proof is a little involved.

\begin{lemma} \label{lem7.12}  Suppose $P$ is a set of $n$ points lying on a union of lines $\ell_1 \cup \dots \cup \ell_m$, where $m \leq n^{\frac{1}{10000}}$.  Suppose that $i \neq j$ and that $X_i \subset P \cap \ell_i$, $X_j \subset P \cap \ell_j$ are grids, both of size at least $\eps' n$, where $\eps' \gg n^{-\frac{1}{1000}}$, and neither containing $\ell_i \cap \ell_j$. Then either $P$ spans $\gg (\eps')^3 n^2$ ordinary lines, or else there is a line $\ell_k$, not passing through $\ell_i \cap \ell_j$, and a grid $X_k \subset P \cap \ell_k$, such that all but at most $\eps'  |X_i||X_j|/25$ pairs $(x_i , x_j ) \in X_i \times X_j$ are such that the line $\overline{\{x_i, x_j\}}$ meets $\ell_k$ in a point of $X_k$. 
\end{lemma}
\begin{proof} Write $\eta := \eps'/100$. Suppose that $P$ does not contain at least $(\eps')^3 n^2/100 = \eta(\eps' n)^2$ ordinary lines. Then there are at least $(1 - \eta) |X_i||X_j|$ pairs $(x_i, x_j) \in X_i \times X_j$ such that $\overline{\{x_i, x_j\}}$ meets some other line $\ell_k$. Write $\Gamma_k \subset X_i \times X_j$ for the set of pairs $(x_i, x_j) \in X_i \times X_j$ such that $\overline{\{x_i, x_j\}}$ meets $\ell_k$ in a point of $P$. Thus $\sum_k |\Gamma_k| \geq (1 - \eta) |X_i||X_j|$. We claim that there is at most one value of $k$ for which $|\Gamma_k| \geq \eta |X_i||X_j|/m$, and that for this $k$ (if it exists) the line $\ell_k$ does not pass through $\ell_i \cap \ell_j$. Note that if $|\Gamma_k| \geq \eta |X_i||X_j|/m$ then certainly $|\Gamma_k| \geq \delta n^2$, where $\delta = n^{-1/250}$. 

This claim is a consequence of the following two facts. 

\emph{Fact 1.} If $\ell_k$ passes through $\ell_i \cap \ell_j$ then $|\Gamma_k| < \delta n^2$. 

\emph{Fact 2.} If we have two lines $\ell_k, \ell_{k'}$, neither passing through $\ell_i \cap \ell_j$, then at least one of $|\Gamma_k|, |\Gamma_{k'}|$ has size at most $\delta n^2$.

\emph{Proof of Fact 1.}  Suppose that $\ell_k$ passes through $\ell_i \cap \ell_j$, but that $|\Gamma_k| \geq \delta n^2$. 
Here (and in the proof of Fact 2 below) we apply an affine transformation so that $\ell_i$ is the $x$-axis and $\ell_j$ is the $y$-axis. Suppose that $\ell_k$ is the line $\{(t, \lambda_k t) : t \in \R\}$. Suppose that $X_i = \{(a,0) : a \in A\}$ and $X_j = \{(0,b) : b \in B\}$ and, by slight abuse of notation,  identify $\Gamma_k$ with a subset of $A \times B$ in the obvious way. A short computation confirms that the intersection of the line through $(a,0) \in \ell_i$ and $(0,b) \in \ell_j$ with $\ell_k$ is the point $( \frac{1}{1/x + \lambda_k/y}, \frac{\lambda}{1/x + \lambda_k/y} )$. By Corollary \ref{bsg-cor} there are sets $A' \subset A$ and $B' \subset B$ with $|A'| , |B'|\gg \delta n$ and $|\frac{1}{A'} - \frac{1}{A'}| \ll  \delta^{-11}n \ll n^{1 + \frac{1}{10}}$. Since $X_i$ is a grid, there is some ratio function $\psi(t) = (t + \alpha)/(t + \beta)$, $\alpha \neq \beta$, such that  $A' \subset \bigcup_{i = 1}^{n^{1/30}} S_i$, where $|\mathcal{Q}(\psi(S_i))| \leq n^{1 + \frac{1}{10}}$ for each $i$. Let $A''$ be the largest of the intersections $A' \cap S_i$; then 
\[ |A''| \gg \delta n^{1 - \frac{1}{30} }, \, |\mathcal{Q}(\psi(A''))| \leq n^{1 + \frac{1}{10}} , \, 
 |\frac{1}{A''} - \frac{1}{A''}| \ll n^{1 + \frac{1}{10}}.\] Writing $Y := 1/A''$ and $f(y) := \log(1 + \alpha y) - \log(1 + \beta y)$ we thus have \begin{equation}\label{y-bounds} |Y| \gg \delta n^{1 - \frac{1}{30}}, \, |Y - Y|, |f(Y) - f(Y)| \ll n^{1 + \frac{1}{10}}.\end{equation}
However \[ f''(y) = \frac{\beta^2}{(1 + \beta y)^2} - \frac{\alpha^2}{(1 + \alpha y)^2} = (\alpha - \beta)\frac{2\alpha \beta y - (\alpha + \beta)}{(1 + \alpha y)^2 (1 + \beta y)^2}\] is continuous away from $y = -1/\alpha$ and $y = - 1/\beta$ and has just one real zero, and therefore one may divide $\R$ into at most 4 intervals on the interior of which $f$ is defined and strictly convex/concave. By Theorem \ref{enr-prop-piecewise} it follows that one of $|Y - Y|$, $|f(Y) - f(Y)|$ has size $\gg |Y|^{5/4}$. This comfortably contradicts \eqref{y-bounds} for large $n$. 

\emph{Proof of Fact 2.} Suppose that $|\Gamma_k|, |\Gamma_{k'}| \geq \delta n^2$. By Lemma \ref{intersect-lem} and the fact that neither $\ell_k$ nor $\ell_{k'}$ contains more than $n$ points we have
\[  | \{ \phi_{i,j,k}(x_i)/\tilde\phi_{i,j,k}(x_j) : (x_i, x_j) \in \Gamma_{k}\}| \leq n\] and \[ | \{ \phi_{i,j,k'}(x_i)/\tilde\phi_{i,j,k'}(x_j) : (x_i, x_j) \in \Gamma_{k'} \} | \leq n .\]
Applying Corollary \ref{bsg-cor} exactly as before we deduce that there are sets $X^{(k)}_i, X^{(k')}_i \subset X_i$ and sets $X^{(k)}_j, X^{(k')}_j \subset X_j$, all of size $\gg \delta n > n^{1 - \frac{1}{25}}$, such that all of $|\mathcal{Q}(\phi_{i,j,k}(X^{(k)}_i)|$, $|\mathcal{Q}(\phi_{i,j,k'}(X^{(k')}_i))|$, $|\mathcal{Q}(\tilde\phi_{i,j,k}(X^{(k)}_j))|$ and $|\mathcal{Q}(\tilde\phi_{i,j,k'}(X^{(k')}_j))|$ have size $\ll\delta^{-12}n < n^{1 + \frac{1}{10}}$. Thus $X^{(k)}_i$ is a $\phi_{i,j,k}$-grid, $X^{(k')}_i$ is a $\phi_{i,j,k'}$-grid, $X^{(k)}_j$ is a $\tilde \phi_{i,j,k}$-grid and $X^{(k')}_j$ is a $\tilde \phi_{i,j,k'}$-grid. Since $X_i$ and $X_j$ are themselves grids, and since grids corresponding to inequivalent ratio functions intersect in a set of size no more than $n^{1 - \frac{1}{25}}$ by Lemma \ref{lem7.9}, we see that  $\phi_{i,j,k} \sim \phi_{i,j,k'}$ and $\tilde\phi_{i,j,k} \sim \tilde\phi_{i,j,k'}$. 

It follows immediately from the definition of $\phi_{i,j,k}, \phi_{i,j,k'}, \tilde\phi_{i,j,k}, \tilde\phi_{i,j,k'}$ that $\ell_i \cap \ell_k = \ell_i \cap \ell_{k'}$ and $\ell_j \cap \ell_k = \ell_j \cap \ell_{k'}$, and therefore $\ell_k = \ell_{k'}$. This is contrary to assumption, and so we have established Fact 2.

This completes the proof of the claim. It follows that there is some $k$ such that at least $(1 - 2\eta)|X_i||X_j|$ pairs $x_i \in X_i, x_j \in X_j$ are such that $\overline{\{x_i, x_j\}}$ meets $P \cap \ell_k$, and furthermore $\ell_k$ does not pass through $\ell_i \cap \ell_j$. If $x_k \in P \cap \ell_k$, we say that $x_k$ is \emph{well-covered} if there are at least $\eta(\eps')^2 n$ pairs $(x_i , x_j) \in X_i \times X_j$ such that $\overline{\{x_i, x_j\}}$ passes through $x_k$. Since $|P \cap \ell_k| \leq n$, the number of pairs $(x_i, x_j)$ used up by poorly-covered $x_k$ is no more than $\eta(\eps')^2n^2 \leq \eta |X_i||X_j|$. Thus, for at least $(1 - 3\eta)|X_i||X_j|$ pairs $(x_i, x_j)$, the line $\overline{\{x_i ,x_j\}}$ meets $\ell_k$ in a well-covered point $x_k$.  Write $\tilde X_k \subset P \cap \ell_k$ for the set of well-covered points; we are going to show that there is a grid $X_k$ occupying almost all of $\tilde X_k$. 

Note that, by definition of well-covered, for any $Y \subset \tilde X_k$ there is a graph $\Gamma \subset Y \times X_i$, $|\Gamma| \geq \eta (\eps')^2 |Y||X_i|$, such that each line $\overline{\{y , x_i\}}$ meets $X_j$ whenever $(y, x_i) \in \Gamma$. By Lemma \ref{intersect-lem} and the trivial bound $|X_j| \leq n$, this implies that
\[ | \{ \phi_{k,i,j}(y)/\tilde\phi_{k,i,j}(x_i) : (y, x_i) \in \Gamma \} | \leq n.\]

Suppose that $|Y| \geq \eta n$. Then $|\Gamma| \gg \eta^2 (\eps')^3 n^2 \gg n^{2-\frac{1}{200}}$. By Corollary \ref{bsg-cor} there are sets $Y' \subset Y$ and $X'_i \subset X_i$, $|Y'|, |X'_i| \gg n^{1-\frac{1}{200}}$, such that $ | \mathcal{Q}(\phi_{i,j,k}(Y'))|  < n^{1 +\frac{1}{10}}$.
Applying this argument repeatedly, we see that all of $\tilde X_k$ except for a set of size at most $\eta n$ can be covered by disjoint sets $Y'$ with these properties. There are at most $Cn^{\frac{1}{200}} < n^{\frac{1}{30}}$ of these sets, and so the union of them, $X_k$ say, is a $\phi_{i,j,k}$-grid.

Finally, note that the number of pairs $(x_i, x_j)$ for which $\overline{\{ x_i ,x_j\}}$ passes through $\tilde X_k \setminus X_k$ is at most $n |\tilde X_k \setminus X_k| \leq  \eta n$.  For all other pairs, $\overline{\{x_i, x_j\}}$ passes through the grid $X_k$.

At last, this completes the proof of the lemma.
\end{proof}

We are now in a position to complete the proof of Theorem \ref{not-betts-case} and hence of all the other results in this section. \vspace{11pt}

\emph{Proof of Theorem \ref{not-betts-case}.} Suppose that $P$ lies on $\ell_1 \cup \dots \cup \ell_m$, with at least $\eps n$ points on each line and not all of the lines through a single point.  Suppose that $m, \frac{1}{\eps} \leq n^{\frac{1}{10000}}$. By Lemma \ref{lem7.10a} we are done unless there are two values $i,j \in \{1,\dots,m\}$ such that $P \cap \ell_i$, $P \cap \ell_j$ both contain grids of size at least $\eps' n$, with $\eps' \gg \eps^4/m^2 > n^{-\frac{1}{1000}}$. Suppose without loss of generality that $\ell_i$ contains the largest grid amongst all grids in $P$; call this $X_i$. Let $X_j$ be a grid in $P \cap \ell_j$ of size at least $\eps' n$, and apply Lemma \ref{lem7.12} to these two grids $X_i, X_j$. If the first conclusion of that lemma holds then $P$ contains $\gg (\eps' )^3 n^2 \gg \eps^{12} n^2/m^6$ ordinary lines, and we are done. Otherwise there is a line $\ell_k$, not passing through $\ell_i \cap \ell_j$, and a grid $X_k \subset P \cap \ell_k$, such that for all $(x_i, x_j)$ in a set $\Gamma$ of size at least $(1 - 4\eta) |X_i||X_j|$, the line $\overline{\{x_i,x_j\}}$ meets $\ell_k$ in a point of $X_k$.

By Lemma \ref{intersect-lem} we have
\[ |\{ \phi_{i,j,k}(x_i)/\tilde\phi_{i,j,k}(x_j) : (x_i, x_j) \in \Gamma\}| \leq |X_k|.\]
It follows from Lemma \ref{add-comb-lem-7-mult} that 
\[ |X_k| \geq |X_i| + |X_j| - 4 - 4\sqrt{2\eta |X_i||X_j|} .\]
Since $\eta = \eps'/100$ and $n$ is sufficiently large, this is strictly greater than $|X_i|$, contrary to the assumption that $X_i$ was a grid in $P$ of largest size.
\endproof

In this rather long section, we have proved results about sets $P$, contained in a union $\ell_1 \cup \dots \cup \ell_m$ of lines, spanning few ordinary lines. They are plausibly of independent interest. Let us now, however, return to our main task and combine what we have established with the main results of previous sections. By combining Proposition \ref{main-line-prop} (using the bound on $n_0(m,K)$ noted after Proposition \ref{many-rich}) with Proposition \ref{intermediate} we obtain the following structural result. From the qualitative point of view at least, this supersedes all previous results in the paper (and, in particular, implies Theorem \ref{weak-struct} as a corollary).

\begin{proposition}\label{new-struct}
Suppose that $P$ is a set of $n$ points in $\R\P^2$ for some $n \geq 100$ spanning at most $Kn$ ordinary lines, where $1 \leq K \leq c(\log \log n)^c$ for some sufficiently small absolute constant $c$. Then $P$ differs in at most $O(K^{O(1)})$ points from a subset of a set of one of the following three types:
\begin{enumerate}
\item An irreducible cubic curve;
\item The union of an irreducible conic and a line;
\item A line.
\end{enumerate}
\end{proposition}

\begin{proof}  We apply Proposition \ref{intermediate}.  In case (iii) we are already done.  In cases (i) and (ii) we see that, after removing an irreducible conic if necessary, that we have $\geq n/2 - O(K^{O(1)})$ points on $O(K^{O(1)})$ lines determining at most $O(K^{O(1)}) n$ ordinary lines.  By Proposition \ref{main-line-prop} and the hypothesis $K \leq (\log \log n)^c$, all but $O(K^{O(1)})$ of these points lie on a single line, and the claim follows.
\end{proof}

\section{The detailed structure theorem}\label{detailed-structure}

We turn now to the proof of the detailed structure theorem, Theorem \ref{main-structure-theorem}.  We first establish a slightly weaker version in which the linear bounds $O(K)$ have been relaxed to polynomial bounds $O(K^{O(1)})$. This somewhat weaker statement is already sufficient for our main application, the proof of the Dirac--Motzkin conjecture for large $n$.  At the end of this section we indicate how to recover the full strength of Theorem \ref{main-structure-theorem}.

\begin{theorem}\label{main-weak}
Suppose that $P$ is a finite set of $n$ points in the projective plane $\R\P^2$. Suppose that $P$ spans at most $Kn$ ordinary lines for some $K \geq 1$, and suppose that $n \geq \exp\exp(CK^C)$ for some sufficiently large absolute constant $C$. Then, after applying a projective transformation if necessary,  $P$ differs by at most $O(K^{O(1)})$ points from an example of one of the following three types:
\begin{enumerate}
\item $n-O(K^{O(1)})$ points on a line;
\item The set $X_{2m}$ defined in \eqref{x2m} for some $m = \frac{1}{2}n - O(K^{O(1)})$;
\item A coset $H \oplus x$, $3x \in H$, of some finite subgroup $H$ of the real points on an irreducible cubic curve with $H$ having cardinality $n+O(K^{O(1)})$.
\end{enumerate}
\end{theorem}

We now prove this theorem.
We already know from Proposition \ref{new-struct} that a set with at most $Kn$ ordinary lines must mostly lie on a line, the union of a conic and a line, or an irreducible cubic curve, and the proof of Theorem \ref{main-structure-theorem} proceeds by analysing the second two possibilities further. 

The key to doing this is the fact that collinearities on (possibly reducible) cubic curves are related to group structure. This is particularly clear in the case of an irreducible cubic, as we briefly discussed in Section \ref{examples-sec} . However one can see some group structure even in somewhat degenerate cases.

An important ingredient in our analysis is the following result of a fairly standard type from additive combinatorics, which may be thought of as a kind of structure theorem for triples of sets $A, B, C$ with few ``arithmetic ordinary lines''. 

\begin{almost-group-rpt}
Suppose that $A, B, C$ are three subsets of some abe- lian group $G$, all of cardinality within $K$ of $n$, where $K \leq \eps n$ for some absolute constant $\eps > 0$. Suppose that there are at most $Kn$ pairs $(a,b) \in A \times B$ for which $a +b \notin C$. Then there is a subgroup $H \leq G$ and cosets $x + H, y + H$such that $|A \triangle (x + H)|, |B \triangle (y + H)|, |C \triangle (x + y + H)| \leq 7K$.
\end{almost-group-rpt}

This result will not be at all surprising to the those initiated in additive combinatorics, but we do not know of a convenient reference for it. We supply a complete proof in Appendix \ref{add-comb-app}.

Suppose that $P$ is mostly contained in a cubic curve $\gamma$ which is not a line. We subdivide into two cases according to whether $\gamma$ is irreducible or not.

\begin{lemma}[Configurations mostly on an irreducible cubic]\label{irreducible-case}
Suppose that $P$ is a set of $n$ points in $\R\P^2$ spanning at most $Kn$ ordinary lines. Suppose that all but $K$ of the points of $P$ lie on an irreducible cubic curve $\gamma$, and suppose that $n \geq CK$ for a suitably large absolute constant $C$. Then there is a coset $H \oplus x$ of $\gamma$ with $3x = x \oplus x \oplus x \in H$ such that $|P \triangle (H \oplus x)| = O(K)$. In particular, $\gamma$ is either an elliptic curve or an acnodal cubic.
\end{lemma}

\begin{proof} 
Let $\gamma^*$ be the smooth points of $\gamma$, which we give a group law as in Section \ref{examples-sec}.

Set $P' := P \cap \gamma^*$. Then $|P'| = |P| + O(K)$, and $P'$ spans at most $O(Kn)$ ordinary lines. If $p_1, p_2 \in \gamma^*$ are distinct then the line joining $p_1$ and $p_2$ meets $\gamma^*$ again in the unique point $ \ominus p_1 \ominus p_2$. This assumption implies that $\ominus p_1 \ominus p_2 \in P'$ for all but at most $O(Kn)$ pairs $p_1, p_2 \in P'$. Applying Proposition \ref{almost-group}, it follows that there is a coset $H \oplus x$ of $\gamma^*$ such that $|P \triangle (H \oplus x)| = O(K)$ and also $|P \triangle (H \ominus 2x)| = O(K)$. From this it follows that $|(H \oplus x) \triangle (H \ominus 2x)| = O(K)$, which implies that $3x \in H$ if $n \geq CK$ is large enough.

If $\gamma$ is not an elliptic curve or an acnodal cubic then the group $\gamma^*$ is isomorphic to $\R$ or to $\R \times \Z/2\Z$, and neither of these groups has a finite subgroup of size larger than $2$.\end{proof}

We turn now to the consideration of sets $P$ which are almost contained in the union of an irreducible conic $\sigma$ and a line $\ell$. The union $\sigma \cup \ell$ is a reducible cubic and so does not have a \emph{bona fide} group law. There is, however, a very good substitute for one as the following proposition shows. In what follows we write $\sigma^* := \sigma \setminus (\sigma \cap \ell)$ and $\ell^* := \ell \setminus (\sigma \cap \ell)$. Note that the intersection $\sigma \cap \ell$ has size $0, 1$ or $2$. 

\begin{proposition}[Quasi-group law]\label{quasi-group-law}
Suppose that $\sigma$ is an irreducible conic and that $\ell$ is a line. Then there is an abelian group $G = G_{\sigma, \ell}$ with operation $\oplus$ and bijective maps $\psi_{\sigma} : G \rightarrow \sigma^*$, $\psi_{\ell} : G \rightarrow \ell^*$ such that $\psi_{\sigma}(x), \psi_{\sigma}(y)$ and $\psi_{\ell}(z)$ are collinear precisely if $x \oplus y \oplus z = 0$. Furthermore $G_{\sigma, \ell}$ is isomorphic to $\Z/2\Z \times \R$ if $|\sigma \cap \ell| = 2$,  to $\R$ if $|\sigma \cap \ell| = 1$ and to $\R/\Z$ if  $|\sigma \cap \ell| = 0$.
\end{proposition}

\begin{proof}
This is certainly a known result, but it is also an easy and fun exercise to work through by hand, as we now sketch. If $|\sigma \cap \ell| = 2$, we may apply a projective transformation to move the two points of intersection to $[0,0,1]$ and $[0,1,0]$, and $\sigma^*$ to the parabola $\{[a,a^2,1] : a \in \R^\times \}$ and $\ell^*$ to $\{[0,-b,1] : b \in \R^{\times} \}$. We note that $[a_1,a_1^2,1]$, $[a_2, a_2^2, 1]$ and $[0,-b,1]$ are collinear if and only if $b = a_1 a_2$.

Consider the maps $\psi_{\ell} : \Z/2\Z \times \R \rightarrow \ell^*$ defined by $\psi_{\ell}(\eps, x) = [0, -b, 1]$ where $b = (-1)^\eps e^{-x}$ and
$\psi_{\sigma} : \Z/2\Z \times \R \rightarrow \sigma^*$ defined by $\psi_{\sigma}(\eps, x) = [a, a^2, 1]$ where $a = (-1)^\eps e^x$. Then we see that the maps $\psi_{\ell}, \psi_{\sigma}$ are bijections and that the claimed collinearity property holds.

Now suppose that we are in case (ii), that is to say $|\sigma \cap \ell| = 1$. 
Applying a projective transformation, we may suppose that the point of intersection is $[0,1,0]$ and move $\sigma^*$ to the parabola $\{ [a,a^2,1] : a \in \R \}$ and $\ell^*$ to the line at infinity $\{ [1,-b,0]: b \in \R \}$.
Now note that if $[a_1,a_1^2,1]$, $[a_2, a_2^2, 1]$ and $[1,-b,0]$ are distinct and collinear, then $a_1 + a_2 + b = 0$ (cf. the near-example \eqref{p4}).  

Finally suppose that $\sigma$ and $\ell$ do not intersect. Applying a projective transformation we may map $\ell$ to the line at infinity $\{ [\sin \pi \theta, \cos \pi \theta, 0]: \theta \in \R/\Z\}$. As $\sigma$ is disjoint from $\ell$ it must be a compact conic section in $\R^2$, that is to say an ellipse. By a further affine transformation we may assume that it is in fact the unit circle $\{[\cos 2\pi \theta, \sin 2 \pi \theta,1] : \theta \in \R/\Z\}$. By elementary trigonometry it may be verified that the points $[\cos 2\pi \alpha_1, \sin 2 \pi \alpha_1, 1]$, $[\cos 2\pi \alpha_2, \sin 2 \pi \alpha_2, 1]$ and $[\sin \pi \beta, \cos \pi \beta, 0]$ are collinear if  and only if $\alpha_1 + \alpha_2 + \beta = 0$, thus in this case the result is true with $\psi_{\sigma}(\theta) = [\cos 2\pi \theta, \sin 2 \pi \theta,1] $ and $\psi_{\ell}(\theta) = [\sin \pi \theta, \cos \pi \theta, 0]$.
\end{proof}

We may now derive the following consequence, analogously to Lemma \ref{irreducible-case}.

\begin{lemma}[Conic and line]\label{reducible-case}
Suppose that $P \subset \R\P^2$ is a set of $n \geq n_0(K)$ points, all except $K$ of which lie on the union of an irreducible conic $\sigma$ and a line $\ell$. Suppose that $P$ defines at most $Kn$ ordinary lines, and suppose that $P$ has $n/2 + O(K)$ points on each of $\sigma$ and $\ell$. Then, after a projective transformation, $P$ differs from one of the sets $X_{n'}$ by at most $O(K)$ points. 
\end{lemma}

\begin{proof} 
Write $P' := P \cap (\sigma \cup \ell)$. Set $P_{\sigma} := P \cap \sigma^*$ and $P_{\ell} := P \cap \ell^*$. Then $|P_{\sigma}| + |P_{\ell}| = |P| + O(K)$, and $P'$ spans at most $O(Kn)$ ordinary lines. Consider also the pull-backs $A := \psi_{\sigma}^{-1}(P_{\sigma})$ and $B := \psi_{\ell}^{-1}(P_{\ell})$, where $\psi_{\ell}, \psi_{\sigma}$ are the ``quasi-group law'' maps introduced in the preceding proposition. Both $A$ and $B$ are subsets of $G_{\sigma, \ell}$, a group for which there are three possibilities, detailed in Proposition \ref{quasi-group-law}.

The assumption about ordinary lines implies that $\ominus a_1 \ominus a_2 \in B$ for all but at most $O(K n)$ pairs $a_1, a_2 \in A$. Applying Lemma \ref{almost-group}, it follows that there is a subgroup $H \leq G_{\gamma}$ and cosets $x \oplus H,  -2x \oplus H$ such that 
 $|A \triangle (x \oplus H)|, |B \triangle (-2x \oplus  H)| = O(K)$.

If $n \geq CK$ for large enough $C$ then it follows that $|\sigma \cap \ell| = 0$ since (with reference to the three possibilities for $G_{\sigma, \ell}$ described in Proposition \ref{quasi-group-law}) neither $\Z/2\Z \times \R$ nor $\R$ has a finite subgroup of size larger than 2. Applying a projective transformation, we may assume that $\ell$ is the line at infinity and, as in the proof of Proposition \ref{quasi-group-law}, that $\sigma$ is the unit circle. We may apply a further rotation so that $x = 0$, that is to say $|A \triangle H| = O(K)$ and $|B \triangle H| = O(K)$. 

All finite subgroups of $\R/\Z$ are cyclic and so we have $H = \{j/m : j \in \{0,1,\dots, m-1\}\}$ for some $m$, that is to say $H$ consists of the (additive) $m^{\operatorname{th}}$ roots of unity. But then $\psi_{\sigma}(H) \cup \psi_{\ell}(H)$ is precisely the set $X_{n'} = X_{2m}$ described in the introduction and in the statement of Theorem \ref{main-structure-theorem}. 
\end{proof}

Putting Lemmas \ref{irreducible-case} and \ref{reducible-case} together with the main result of the previous section, Proposition \ref{new-struct}, we immediately obtain Theorem \ref{main-weak}. The remainder of this section is devoted to establishing our most precise structure theorem, Theorem \ref{main-structure-theorem}. Let us remind the reader that this is the same as Theorem \ref{main-weak}, only the polynomial error terms $O(K^{O(1)})$ are replaced by linear errors $O(K)$. The reader interested in the proof of the Dirac--Motzkin conjecture for large $n$ may proceed immediately to the next section, where Theorem \ref{main-weak} is already sufficient.\vspace{11pt}

\emph{Proof of Theorem \ref{main-structure-theorem}. } The converse claim to this theorem already follows from the analysis in Section \ref{examples-sec}, so we focus on the forward claim.  We may assume that the constant $C$ is sufficiently large.  We then apply Theorem \ref{main-weak} to obtain (after a projective transformation) that $P$ differs by $O(K^{O(1)})$ points from one of the three examples (i), (ii), (iii) listed. Our task is to bootstrap this $O(K^{O(1)})$ error to a linear error $O(K)$.

Suppose first that case (i) holds, thus all except $O(K^{O(1)})$ points of $P$ lie on a line $\ell$.  Then every point $p$ in $P$ that does not lie on $\ell$ forms at least $n-O(K^{O(1)})$ lines with a point in $P \cap \ell$.  At most $O(K^{O(1)})$ of these can meet a further point in $P$, so each point in $P \backslash \ell$ produces at least $n-O(K^{O(1)})$ ordinary lines connecting that point with a point in $P \cap \ell$.  We conclude that the number of ordinary lines is at least $(n-O(K^{O(1)})) |P \backslash \ell|$; since there are at most $Kn$ ordinary lines, we conclude that $|P \backslash \ell| = O(K)$, and the claim follows.

Now suppose that case (ii) holds, thus $P$ differs by $O(K^{O(1)})$ points from $X_{2m}$ for some $m = \frac{1}{2}n-O(K^{O(1)})$.  To analyse this we need the following result, essentially due to Poonen and Rubinstein \cite{poonen-rubinstein}.

\begin{proposition}\label{poonen-rubinstein-prop}
Let $\Pi_n \subset \C \equiv \R^2$ denote the regular $n$-gon consisting of the $n^{\operatorname{th}}$ roots of unity. Then no point other than the origin or an element of $\Pi_n$ lies on more than $C$ lines joining pairs of vertices in $\Pi_n$, for some absolute constant $C$.
\end{proposition}

Actually, in \cite{poonen-rubinstein} it was shown that $C$ could be taken to be $7$ when one restricts attention to points \emph{inside} the unit circle.  The case of points outside the unit circle was not directly treated in that paper, but can be handled by a variant of the methods of that paper. See Appendix \ref{chord-app} for details.  For the purposes of establishing Theorem \ref{main-structure-theorem}, the full strength of Proposition \ref{poonen-rubinstein-prop} is unnecessary. Indeed, the more elementary weaker version established in Proposition \ref{chords-prop} would also suffice for this purpose.

\begin{corollary}\label{prc}
Suppose that $p \in \R\P^2$ does not lie on the line at infinity, is not an $m^{\operatorname{th}}$ root of unity and is not the origin $[0,0,1]$. Then at least $2m - O(1)$ of the $2m$ lines joining $p$ to a point of $X_{2m}$ pass through no other point of $X_{2m}$. 
\end{corollary}

\begin{proof}
If $x \in X_{2m}$ then we say that $x$ is \emph{bad} if the line $\overline{\{p,x\}}$ passes through another point $y\in X_{2m}$, $y \neq x$. Suppose first that $x$ is an $m^{\operatorname{th}}$ root of unity. We claim that if $x$ is bad then $px$ passes through another $m^{\operatorname{th}}$ root of unity, different from $x$, or else $\overline{\{p,x\}}$ is tangent to the unit circle. If $y$ is already an $m^{\operatorname{th}}$ root of unity then we are done; otherwise $y$ is one of the $m$ points on the line at infinity. But then the line $\overline{\{x,y\}}$ passes through another $m^{\operatorname{th}}$ root of unity $x'$ unless it is tangent to the unit circle, and we have proved the claim. 

This enables us to count the number of bad $x$ which are $m^{\operatorname{th}}$ roots of unity. There are at most two coming from the possibility that $\overline{\{p,x\}}$ is tangent to the unit circle. Otherwise, $\overline{\{ p, x\}}$ contains another $m^{\operatorname{th}}$ root of unity $x'$, whence $p$ lies on the line $\overline{\{x,x'\}}$. This gives at most $O(1)$ further possibilities by Proposition \ref{poonen-rubinstein-prop}. 

Now suppose that $x$ is one of the $m$ points on the line at infinity. If $\overline{\{p,x\}}$ passes through an $m^{\operatorname{th}}$ root of unity $y$ then it passes through another such root of unity $y'$ unless $\overline{\{p,x\}}$ is tangent to the unit circle. There are at most $2$ points on the line at infinity corresponding to the tangent lines, and then at most $O(1)$ corresponding to the chords $\overline{\{y ,y'\}}$ on which $p$ lies, by another application of Proposition \ref{poonen-rubinstein-prop}.
\end{proof}

We return now to the analysis of case (ii).  Let $p$ be a point of $P$ not on either the unit circle or the line at infinity.  Then by Corollary \ref{prc}, only $O(K^{O(1)})$ of the $n-O(K^{O(1)})$ lines connecting $p$ with $X_{2m}$, also meet another element of $X_{2m}$.  As $P$ only differs from $X_{2m}$ by $O(K^{O(1)})$ points, we conclude that there are $n-O(K^{O(1)})$ ordinary lines of $P$ that connect $p$ with an element of $X_{2m}$.  As in case (i), this implies that there are at most $O(K)$ points of $P$ lying outside the union of the unit circle and the line at infinity.  Applying Lemma \ref{reducible-case}, we obtain the claim.

Finally, we consider the case (iii).  By Lemma \ref{irreducible-case}, it suffices to show that there are at most $O(K)$ points of $P$ that do not lie on the curve $E$, which is either an elliptic curve or the smooth points of an acnodal singular cubic curve.  By the same argument used to handle cases (i) and (ii), it then suffices to show that each point $p$ in $P \backslash E$ generates $\gg n$ ordinary lines in $P$.

For this, it suffices to establish the following lemma.

\begin{lemma}\label{elliptic-lemma}
Suppose that $E$ is an elliptic curve or the smooth points of an acnodal singular cubic curve and that $H \oplus x$ is a coset of a finite subgroup of $E$ of size $n > 10^4$. Then, if $p \notin E$ is a point, then there are at least $n/1000$ lines through $p$ that meet exactly one element of $H \oplus x$.
\end{lemma}

\emph{Remark.} The constant $1/1000$ could be improved a little by our methods, but we have not bothered to perform such an optimisation here. It seems reasonable to conjecture, in analogy with the results in \cite{poonen-rubinstein}, that in fact there are only $O(1)$ lines through $p$ that can meet three elements of a coset $x \oplus H$ of a finite subgroup of an elliptic curve, but this would seem to lie far deeper.\vspace{11pt}

\begin{proof} We first exclude one degenerate case, in which $E$ is the smooth points of an acnodal singular cubic curve, and $p$ is the isolated (i.e. acnodal) singular point of that curve.  In this case, any line through $p$ meets exactly one point of $E$, and the claim is trivial.  Thus we may assume that $p$ does not lie on the cubic curve that contains $E$.

Suppose the result is false. Then at least $0.999n$ of the lines joining $p$ to $x \oplus H$ meet $x \oplus H$ in $2$ or $3$ points. In the former case, the line must be tangent to $E$. There are at most $6$ such tangents\footnote{The points of tangency must all lie on the intersection of $E$ with its dual curve with respect to $p$, which has degree $2$; see also the \emph{Pl\"ucker Formul\ae.} }. Thus at least $0.998n$ of the lines joining $p$ to points of $x \oplus H$ meet $x \oplus H$ in 3 points. 

As a topological group, the cubic curve $E$ is isomorphic to either $\R/\Z$ or $\R/\Z \times (\Z/2\Z)$.  Consider all the lines through $p$ that are tangent to $E$; there are at most $6$ such lines, each meeting $E$ in at most $2$ points. These (at most) 12 points partition $E$ into no more than $13$ connected open sets $A_1,\ldots,A_{13}$ (topologically, these are either arcs or closed loops), plus $12$ endpoint vertices.  From a continuity argument, we see that for each $i$ one of the following statements is true:
\begin{enumerate}
\item the lines connecting $p$ to points of $A_i$ do not meet $E$ again;
\item there exist $A_j, A_k$, distinct from each other and from $A_i$, such that any line connecting $p$ and a point in $A_i$ meets $E$ again, once at a point in $A_j$, and once at a point in $A_k$. 
\end{enumerate} 

Suppose that $i$ is of type (i). Then by our supposition that the lemma is false we may assume that $|A_i \cap (H \oplus x)| < 0.001n$, since all the lines from $p$ to $A_i \cap (H \oplus x)$  contain no other point of $E$. By the pigeonhole principle there is some $i$ of type (ii) with $|A_i \cap (H \oplus x)| > \frac{1}{13}(1 - 0.012) n > 0.05n > 3$. By property (ii), lines from $p$ through $A_i$ meet the curve $E$ again in $A_j$ and $A_k$. 

Recall that for all except at most $0.002 n$ elements $q$ of $A_i \cap (H \oplus x)$, the line $\overline{\{p,q\}}$ meets $A_j$ and $A_k$ at elements of $H \oplus x$. It is easy to conclude from this, and similar statements for $j,k$, that the sizes of $A_i \cap (H \oplus x)$, $A_j \cap (H \oplus x)$ and $A_k \cap (H \oplus x)$ differ by at most $0.004 n$.

Let $\phi_{ij}: A_i \to A_j$ be the map that sends a point $q$ in $A_i$ to the point $\overline{\{p,q\}} \cap A_j$, then $\phi_{ij}$ is a homeomorphism from the set $A_i$ to the set $A_j$; in particular, $\phi_{ij}$ is either orientation-preserving or orientation-reversing, once one places an orientation on both $A_i$ and $A_j$.  Furthermore, $\phi_{ij}$ maps all but at most $0.002 n$ of the elements of $A_i \cap (H \oplus x)$ to $A_j \cap (H \oplus x)$ and vice versa.  Now as $H$ is a subgroup of $E$, which as an abelian topological group is either $\R/\Z$ or $\R/\Z \times (\Z/2\Z)$, we see that the sets $A_i \cap (H \oplus x)$, $A_j \cap (H \oplus x)$, being intersections of arcs in $\R/\Z$ with the discrete coset $x \oplus H$, are arithmetic progressions in $x \oplus H$ with a common spacing $h$.  

Now for all but at most $0.002 n$ values of $y \in A_i \cap (H \oplus x)$, $\phi_{ij}$ maps $y$ to a point of $A_j \cap (H \oplus x)$. For all but at most $1 + 0.002 n$ values of $y \in A_i \cap (H \oplus x)$, $\phi_{ij}$ maps $y \oplus h$ to a point of $A_j \cap (H \oplus x)$. (The extra $1$ comes from the fact that there is one endpoint value of $y$ in the progression $A_i \cap (x \oplus H)$ for which $y \oplus h$ does not lie in this progression.) Of the values of $y$ satisfying both of these statements, for all but at most $0.004 n$ values we have 
\begin{equation}\label{phiij}
 \phi_{ij}(y \oplus h) = \phi_{ij}(y) \oplus h'
\end{equation}
for $h'$ equal to either $h$ or $\ominus h$ (depending on whether $\phi_{ij}$ is orientation-preserving or orientation-reversing). 
Thus \eqref{phiij} holds for all except at most $0.01 n$ values of $y \in A_i \cap (H \oplus x)$.

Similarly, defining $\phi_{ik}$ in exactly the same way as $\phi_{ij}$, we see that for all except at most $0.01 n$ elements $y$ in $A_i \cap (H \oplus x)$ we have
\begin{equation}\label{phiik}
\phi_{ik}(y \oplus h) = \phi_{ik}(y) \oplus h''
\end{equation}
for $h''$ equal to either $h$ or $\ominus h$.

Recalling that $|A_i \cap (x \oplus H)| > 0.05n$, we may thus find $y \in (x \oplus H) \cap A_i$ such that both \eqref{phiij} and \eqref{phiik} hold.  On the other hand, as $y, \phi_{ij}(y), \phi_{ik}(y)$ are collinear, we have
$$ y \oplus \phi_{ij}(y) \oplus \phi_{ik}(y) = O$$
and similarly
$$ y \oplus h \oplus \phi_{ij}(y \oplus h) \oplus  \phi_{ik}(y \oplus h) = O.$$
From these equations and \eqref{phiij}, \eqref{phiik} we conclude that
$$ h \oplus h' \oplus h'' = O.$$
Since $h', h''$ are equal to either $h$ or $\ominus h$, we conclude that $h$ has order at most $3$, and so $|A_i \cap (H \oplus x)| \leq 3$. However we have already observed that $|A_i \cap (H \oplus x)| > 3$, a contradiction.\end{proof}

The proof of Theorem \ref{main-structure-theorem} is now complete.

\section{The Dirac--Motzkin conjecture}

The Dirac-Motzkin conjecture is the statement that, for $n$ large, a set of $P$ points in $\R^2$ not all lying on a line spans at least $n/2$ ordinary lines. The main result of this paper is a proof of a more precise version of this for large $n$, Theorem \ref{dirac-motzkin-1}, together with a characterization of the extremal examples.  We prove the even more precise Theorem \ref{dirac-motzkin-2}, of which Theorem \ref{dirac-motzkin-1} is an easy consequence, in this section. We refer the reader to Section \ref{examples-sec} for a precise statement of these two results and a leisurely discussion of the relevant examples. 

Suppose that $P$ spans at most $n$ ordinary lines and that $P$ is not collinear. We may apply our main structure theorem, Theorem \ref{main-structure-theorem}, to conclude that $P$ differs in $O(1)$ points from one of three examples: points on a line, a set $X_{n'}$, and a coset of a subgroup of an irreducible cubic curve (Sylvester-type example). In fact the weaker and rather easier Theorem \ref{main-weak} suffices for this purpose.

It is obvious that the first type of set spans at least $n - O(1)$ ordinary lines. Sets close to a Sylvester  example are also relatively easy to handle.

\begin{lemma}\label{lem8.1}
Suppose that $P \subset \R\P^2$ differs in $K$ points from a coset $H \oplus x$ of a subgroup $H$ of some irreducible cubic curve, where $3x = x \oplus x \oplus x \in H$. Then $P$ spans at least $n - O(K)$ ordinary lines.
\end{lemma}

\begin{proof}
Write $h_0 := 3x$, thus $h_0 \in H$. For every $h \in H$, the line joining $h \oplus x$ and $(-2h \ominus h_0) \oplus x$ is tangent to $\gamma$ at $h \oplus x$, since $(h \oplus x) \oplus  (h \oplus x) \oplus (-2h \ominus h_0 \oplus x) = 0$. Therefore it is an ordinary line unless $3h \oplus h_0 = 0$, in which case the points $h \oplus x$ and $-2h \ominus h_0 \oplus x$ coincide. Thus the only points of $H \oplus x$ not belonging to an ordinary line spanned by $H \oplus x$ correspond to the points of $H$ with $3h \oplus h_0 = 0$. Since $H$ is isomorphic to a subgroup of either $\R/\Z$ or $\Z/2\Z \times \R/\Z$, there are no more than $3$ of these. It follows immediately that any set formed by removing at most $K$ points of $H \oplus x$ has at least $n - O(K)$ ordinary lines, and these are all tangent lines to $\gamma$. No point in the plane lies on more than $6$ tangent lines to $\gamma$, and so the addition of a point destroys no more than $6$ of our $n - O(K)$ ordinary lines. It follows that $P$ itself spans at least $n - O(K)$ ordinary lines, as we wanted to prove.\end{proof}

Combining this lemma with the remarks just preceding it, we have now established the existence of an absolute constant $C$ such that a set of $n$ points, not all on a line, and spanning at most $n - C$ ordinary lines, differs in $O(1)$ points from a set $X_{2m}$ consisting of the $m^{\operatorname{th}}$ roots of unity plus $m$ corresponding points on the line at infinity. Now the $m$ tangents to the unit circle at roots of unity pass through only one other point of $X_{2m}$, and so $X_{2m}$ has $m$ ordinary lines. Furthermore, since each point not on the unit circle can be incident to at most two such tangent lines, the addition/deletion of $O(1)$ points does not affect more than $O(1)$ of these lines. This already establishes a weak version of the Dirac-Motzkin conjecture: every non-collinear set of $n$ points spans at least $n/2 - O(1)$ ordinary lines.

To prove Theorem \ref{dirac-motzkin-2}, a much more precise result, we must analyse configurations close to $X_{2m}$ more carefully. What is needed is precisely the following result which, together with what we have already said in this section, completes the proof of Theorem \ref{dirac-motzkin-2}. Recall from Section \ref{examples-sec} the examples of B\"or\"oczky.

\begin{proposition}\label{boroczky-sharp}
There is an absolute constant $C$ such that the following is true. Suppose that $P$ differs from $X_{2m}$ in at most $K$ points, and that $P$ spans at most $2m - CK$ ordinary lines. Then $P$ is a B\"or\"oczky example or a near-B\"or\"oczky example.
\end{proposition}
We now prove this proposition.  Suppose that $P$ differs from $X_{2m}$ in at most $K$ points. Suppose first of all that $P$ contains a point $p$ outside $X_{2m}$, that $p$ does not lie on the line at infinity and that $p \neq [0,0,1]$. Then by Corollary \ref{prc}, the $m$ tangent lines at the unit circle, as well as at least $m - O(1)$ of the lines $\overline{\{p,x\}}$ connecting $p$ to a point $x \in X_{2m}$, pass through precisely $2$ points of $X_{2m} \cup \{p\}$. It is clear that the addition/deletion of $K$ points \emph{other than $p$} cannot add or delete points on more than $O(K)$ of these lines, and so $P$ spans $2m - O(K)$ ordinary lines. 

Now suppose that $P$ contains an additional point $p$ on the line at infinity. Then the $m$ tangent lines to the $m^{\operatorname{th}}$ roots of unity, as well as the at least $2m - 2$ lines $\overline{\{p,x\}}$, $x$ an $m^{\operatorname{th}}$ root of unity, which are not tangent to the unit circle contain precisely two points of $X_{2m} \cup \{p\}$. Once again the addition/deletion of $K$ points \emph{other than $p$} cannot add or delete points on more than $O(K)$ of these lines, and so again $P$ spans $3m - O(K)$ ordinary lines.

We have now reduced matters to the case $P \subset X_{2m} \cup [0,0,1]$. Starting from $X_{2m}$, the omission of a point or the addition of $[0,0,1]$ creates a certain number of new lines with precisely two points, and of course no point other than $[0,0,1]$ or the omitted point is on more than one of these new lines. By inspection in any case there are always at least $m/2 - O(1)$ of these lines, and so there are at least $2m - O(K)$ ordinary lines unless we do \emph{at most one} of the operations of adding $[0,0,1]$ or removing a point of $X_{2m}$. At this point a short inspection of the possibilities leads to the conclusion that the B\"or\"oczky examples and the near-B\"oroczky examples are the only ones which do not have at least $3m - O(1)$ ordinary lines. This, at last, concludes the proof of Theorem \ref{dirac-motzkin-2}.\endproof\vspace{11pt}

\emph{Remark.} We relied on Corollary \ref{prc}, which depended on the result of Poonen and Rubinstein \cite{poonen-rubinstein}. For the purposes of proving the Dirac--Motzkin conjecture for large $n$, the somewhat easier Proposition \ref{chords-prop} is sufficient.

\section{The orchard problem} \label{orchard}

In this section we establish Theorem \ref{mainthm-orchard}, the statement that a set of $n$ points in the plane contains no more than $\lfloor \frac{1}{6} n (n-3) \rfloor + 1$ $3$-rich lines when $n$ is sufficiently large.  The sharpness of this bound was established in Proposition \ref{bgs-prop}.

If $N_k$ is the number of lines containing precisely $k$ points of $P$ then, by double-counting pairs of points in $P$, we have
\begin{equation}\label{double-count-pre}  \sum_{k \geq 2} \binom{k}{2} N_k = \binom{n}{2}.\end{equation}
From this it follows that if $N_3 > \lfloor \frac{1}{6} n(n-3) \rfloor + 1$ then 
\begin{equation}\label{double-count} N_2 + \sum_{k \geq 4} \binom{k}{2} N_k \leq n,\end{equation} from which we conclude that $N_2$, the number of ordinary lines spanned by $P$, is at most $n$. Furthermore no line contains more than $O(\sqrt{n})$ points.

We may now apply Theorem \ref{main-structure-theorem}, our structure theorem for sets with few ordinary lines.  Since no line contains more than $O(\sqrt{n})$ points of $P$ we see that in fact only option (iii) of that theorem can occur, that is to say $P$ differs in $O(1)$ points from a coset $H \oplus x$, $3x \in H$, of a subgroup $H$ of some irreducible cubic curve $\gamma$, which is either an elliptic curve or (the smooth points of) an acnodal singular curve. The rest of the analysis is straightforward but a little tedious.

Suppose that $3x = h_0$. As in the proof of Lemma \ref{lem8.1}, the tangent line to $\gamma$ at $h \oplus x$ meets $H \oplus x$ in the point $(-2h \ominus h_0) \oplus x$, which is distinct from the first point unless $3h \oplus h_0 = 0$. There are at most $O(1)$ of these. In creating $P$ from $H \oplus x$ by the addition/deletion of $O(1)$ points, at most $O(1)$ of these lines are affected. 

Since $P$ spans at most $n$ ordinary lines, it follows that $P$ contains only $O(1)$ ordinary lines other than these tangent lines. Furthermore, since we now know that $P$ contains \emph{at least} $n + O(1)$ ordinary lines, that is to say $N_2 = n + O(1)$, we conclude from \eqref{double-count} that $N_4 = O(1)$. We are going to conclude that $P = H \oplus x$, a statement whose proof we divide into three parts.

\emph{Claim 1.} There is no point of $P$ off the curve $\gamma$. If $p$ is such a point, all except $O(1)$ of the lines joining $p$ to points of $P \cap \gamma$ must contain precisely two points of $P \cap \gamma$, or else there would be too many lines containing $p$ with $2$ or $4$ points of $P$.
Note that this cannot happen if $\gamma$ is the smooth points of an acnodal singular cubic curve and $p$ is the isolated singular point, since every line through $p$ meets at most one point of $\gamma$; thus $p$ lies outside of the cubic curve containing $\gamma$.  Consider the lines $\ell$ from $p$ to $P \cap \gamma$ which are not tangent to $\gamma$ and which contain precisely two points of $P \cap \gamma$ and precisely two points of $H \oplus x$, these points being the same. Since $P \cap \gamma$ differs from $H \oplus x$ in $O(1)$ points, all except $O(1)$ of the lines from $p$ have this property. But any line not tangent to $\gamma$ and containing the two points $h \oplus x$ and $h' \oplus x$ also contains $-(h \oplus h' \oplus h_0) \oplus x$, a third point of $H \oplus x$. This is a contradiction.
 
\emph{Claim 2.}  There is no point of $P$ outside the set $H \oplus x$. Suppose that $k \oplus x$ is such a point. Then if $h \in H$, the line joining $k \oplus x$ and $h \oplus x$ meets $\gamma$ again at $-(k \oplus h \oplus h_0) \oplus x$, which is not a point of $H \oplus x$. This point can thus only lie in $P$ for $O(1)$ values of $h$, and hence there are $n - O(1)$ ordinary lines of $P$ emanating from $k \oplus x$. In addition to the $n - O(1)$ tangent lines, this gives at least $2n - O(1)$ ordinary lines in $P$, a contradiction.

\emph{Claim 3.} $P$ contains all of $H \oplus x$. Suppose that $h_* \oplus x$ is a point of $H \oplus x$ not contained in $P$. For all except $O(1)$ values of $h$, the points $h \oplus x$ and $-(h \oplus h_* \oplus h_0) \oplus x$ lie in $P$, and the line joining $h_* \oplus x$ to them is not tangent to $\gamma$. All such lines then contain precisely two points of $P$, and once again we obtain $n - O(1)$ ordinary lines to add to the $n - O(1)$ tangent lines we already have. Once again a contradiction ensues. 

We have now shown that if $P$ is a set of $n$ points in the plane with $N_3$, the number of lines in $P$ spanning precisely 3 points, satisfying $N_3 > \lfloor \frac{1}{6} n(n-3) \rfloor + 1$, then $P$ is a coset $H \oplus x$ on $\gamma$, an elliptic curve or the smooth points of an acnodal cubic, with $3x \in H$.  But by Proposition \ref{bgs-prop} we have $N_3 \leq \lfloor \frac{1}{6} n(n-3) \rfloor + 1$ in any such case, and we are done.
\vspace{11pt}

\emph{Remarks.} Note that we have in fact classified (for large $n$) the optimal configurations in the orchard problem as coming from cosets in elliptic curves or acnodal cubics. We note that nothing like the full force of Theorem \ref{main-structure-theorem} is required for the orchard problem (as opposed to the Dirac--Motzkin conjecture). Once the much weaker Proposition \ref{intermediate} is established, we can immediately rule out possibilities (ii) and (iii) of that proposition and hence do away with all of the material in Section \ref{somewhat-collinear} and some of the material in Section \ref{detailed-structure} too.

\appendix

\section{Some tools from additive combinatorics}\label{add-comb-app}

In this section we collect some more-or-less standard tools from additive combinatorics used in Sections \ref{somewhat-collinear} and \ref{detailed-structure}.

If $A, B$ are two sets in some abelian group, and if $\Gamma \subset A \times B$ is a set of pairs, we write $A +_{\Gamma} B := \{a + b : (a,b) \in \Gamma\}$. The next result is known as the Balog-Szemer\'edi-Gowers theorem. The precise form we use is a variant of Gowers's version \cite[Proposition 12]{gowers-4aps} due to Sudakov, Szemer\'edi and Vu \cite[Theorem 4.1]{ssv}. 

\begin{theorem}[Balog-Szemer\'edi-Gowers]\label{bsg}
Suppose that $A, B$ are two sets in an abelian group, both of size at most $n$.  Suppose that $\Gamma \subset A \times B$ is a set \textup{(}which may be thought of as a bipartitie graph\textup{)} with $|\Gamma| \geq n^2/K$. Suppose that $|A +_{\Gamma} B| \leq K'n$. Then there are sets $A' \subset A, B' \subset B$ with $|A'| \geq \frac{n}{4K}$, $|B'| \geq \frac{n}{16K^2}$ such that $|A' + B'| \leq 2^{12} (K')^3 K^5 n$. \end{theorem}
\begin{proof} See \cite[Theorem 4.1]{ssv}. In the statement of that result $A$ and $B$ are both supposed to have size $n$, but it easy to see that the proof works under the assumption that they both have size \emph{at most} $n$, for instance by adding dummy elements to $A$ or $B$ (enlarging the group $G$ if necessary) while keeping $\Gamma$ unchanged. \end{proof}

On several occasions in Section \ref{somewhat-collinear} we will apply the preceding theorem together with Ruzsa's triangle inequality (see e.g. \cite[Lemma 2.6]{tao-vu}), which states that $|U| |V - W| \leq |U - V| |U - W|$ for any sets $U,V,W$ in an abelian group (in fact, the group does not even need to be abelian). Let us record, as a corollary, the result of doing this in the particular context we need.

\begin{corollary}\label{bsg-cor}
Suppose that $A, B$ are two sets in an abelian group, both of size at most $n$. Suppose that $\Gamma \subset A \times B$ is a set with $|\Gamma| \geq \delta n^2$ for which $|A +_{\Gamma} B| \leq n$. Then there are sets $A' \subset A$, $B' \subset B$ with $|A'|, |B'| \gg \delta^2 n$ such that $|A' - A'|, |B' - B'| \ll \delta^{-11} n$.
\end{corollary}
\begin{proof}
This follows immediately from the preceding lemma and the Ruzsa triangle inequality.
\end{proof}

The following is a ``robust'' version of the elementary sumset estimate $|U + V| \geq |U| + |V| - 1$.

\begin{lemma}\label{add-comb-lem-7}
Let $U, V \subset \R$ be sets of size $r$ and $s$ respectively. Suppose that $\Gamma \subset U \times V$ has cardinality at least $(1 - \delta)rs$. Then $|U +_{\Gamma} V| \geq r + s - 2 -  2\sqrt{2\delta rs}$.
\end{lemma}
\begin{proof}
Suppose that $U = \{u_1,\dots,u_r\}$ with $u_1 < \dots < u_r$, and $V = \{v_1,\dots, v_s\}$ with $v_1 < \dots < v_s$. For any $1 \leq k \leq \min(r,s)$ we have
\[ u_1 + v_k < u_2 + v_k < \dots < u_{r-k} + v_k < u_{r-k} + v_{k+1} < \dots < u_{r - k} + v_s,\] giving $r + s - 2k$ distinct elements of $U + V$. As $k$ varies, no pair $(u_i, v_j)$ appears in this listing more than twice. Thus by the pigeonhole principle there is, for any choice of positive integer $k_0$,  some $k \leq k_0$ such that 
at most $2\delta rs/k_0$ elements of this listing come from pairs $(u_i, v_j)$ not lying in $\Gamma$. It follows that
\[ |U +_\Gamma V| \geq r + s - 2k_0 - \frac{2\delta rs}{k_0}.\] Choosing $k_0 := \lceil \sqrt{\delta rs}\rceil$ confirms the result.\end{proof}

We actually need a variant of this result for subsets of the \emph{multiplicative} group $\R^{*}$. If $U, V \subset \R^{*}$ and if $\Gamma \subseteq U \times V$ then we write $U \cdot_{\Gamma} V = \{uv : (u,v) \in \Gamma\}$.  

\begin{lemma}\label{add-comb-lem-7-mult}
Let $U, V \subset \R^{*}$ be sets of size $r$ and $s$ respectively. Suppose that $\Gamma \subset U \times V$ has cardinality at least $(1 - \delta)rs$. Then $|U +_{\Gamma} V| \geq r + s - 4 -  2\sqrt{2\delta rs}$.
\end{lemma}
\begin{proof}
As an additive group, $\R^*$ is isomorphic to $\Z/2\Z \times \R$. By abuse of notation, we identify $U$ and $V$ with subsets of this additive group and use additive notation. Define $U_0 = (\{0\} \times \R) \cap U$, $U_1 = (\{1\} \times \R) \cap U$, $V_0 = (\{0\} \times \R) \cap V$ and $V_1 = (\{1\} \times \R) \cap V$. Write $r_0 = |U_0|$, $r_1 = |U_1|$, $s_0 = |V_0|$ and $s_1 = |V_1|$. Suppose that $\Gamma \cap (U_i \times V_j)$ has $\delta_{i,j} r_i s_j$ edges; then 
\begin{equation}\label{edges} \delta_{0,0} r_0 s_0 + \delta_{0,1} r_0 s_1 + \delta_{1,0} r_1 s_0 + \delta_{1,1} r_1 s_1 = \delta rs.
\end{equation}
Clearly 
\[ (U +_{\Gamma} V) \cap (\{0\} \times \R) \supset U_0 +_{\Gamma} V_0, \;\; U_1 +_{\Gamma} V_1  \] and
\[ (U +_{\Gamma} V) \cap (\{1 \} \times \R) \supset U_1 +_{\Gamma} V_0, \; \; U_0 +_{\Gamma} V_1.\]
Therefore by the preceding lemma
\begin{align*} |U +_{\Gamma} V| & \geq \max(r_0 + s_0 - 2\sqrt{2\delta_{0,0} r_0 s_0}, r_1 + s_1 - 2\sqrt{2\delta_{1,1} r_1 s_1}) \\ & + \max(r_0 + s_1 - 2\sqrt{2\delta_{0,1} r_0 s_1}, r_1 + s_0 - 2\sqrt{2 \delta_{1,0} r_1 s_0})  - 4\\ & \geq r_0 + s_0 + r_1 + s_1  - \sum_{i,j} \sqrt{2\delta_{i,j} r_i s_j}  - 4.\end{align*}
Using the inequality 
\[ \sqrt{x} + \sqrt{y} + \sqrt{z} + \sqrt{w} \leq 2\sqrt{x + y + z + w}\] (easily established using Cauchy-Schwarz) together with \eqref{edges}, we obtain
\[ |U+_{\Gamma} V| \geq r + s - 2\sqrt{2\delta rs} - 4,\] as claimed.

\end{proof}

The following result was used heavily in Section \ref{detailed-structure}. It is of a fairly standard type and will be of no surprise to experts in additive combinatorics, but we do not know of a convenient reference.

\begin{proposition}\label{almost-group}
Suppose that $A, B, C$ are three subsets of some abe- lian group $G$, all of cardinality within $K$ of $n$, where $K \leq \eps n$ for some absolute constant $\eps > 0$. Suppose that there are at most $Kn$ pairs $(a,b) \in A \times B$ for which $a +b \notin C$. Then there is a subgroup $H \leq G$ and cosets $x + H, y + H$such that $|A \triangle (x + H)|, |B \triangle (y + H)|, |C \triangle (x + y + H)| \leq 7K$.
\end{proposition}

As remarked in Section \ref{detailed-structure}, results of this general type are quite familiar to additive combinatorialists and are of the general form ``an almost-group is close to a group''. We supply a complete proof here for the convenience of the reader.  Variants of it are possible. For the most part the ideas are due to Kneser \cite{kn53, kn55}, Freiman \cite{frei} and Fournier \cite{fournier}.

We first note that it is enough to prove the following weaker proposition, which may then be ``cleaned up'' to give the stated result
.
\begin{proposition}\label{almost-group-weak} Let $\eps$ be a positive quantity, less than some absolute constant. 
Suppose that $A, B, C$ are three subsets of some abelian group $G$, all of size within $\eps n$ of $n$. Suppose that there at most $\eps n^2$ pairs $(a,b) \in A \times B$ for which $a +b \notin C$. Then there is a subgroup $H \leq G$ and cosets $x + H, y + H$ such that $|A \triangle (x + H)|, |B \triangle (y + H)|, |C \triangle (x + y + H)| \leq \eps' n$, where $\eps'$ can be taken to be $O(\eps^{c})$ for some absolute constant $c > 0$.
\end{proposition}

Let us deduce Proposition \ref{almost-group} from this. Let $A, B$ and $C$ be as in the hypotheses of that proposition. Provided that $\eps$ is small enough, Proposition \ref{almost-group-weak} applies and we conclude that there is a subgroup $H \leq G$ and cosets $x + H, y + H$ such that $|A \triangle (x + H)|, |B \triangle (y + H)|, |C \triangle (x+y + H)| \leq \eps' n$ with $\eps' = O(\eps^c)$. By translating $A$ and $B$ if necessary we may assume without loss of generality that $x = y = 0$. 

Suppose that $A = (H \setminus X) \cup X'$, $B = (H \setminus Y) \cup Y'$ and $C = (H \setminus Z) \cup Z'$, with $X, Y, Z \subset H$, $X', Y', Z'$ disjoint from $H$ and all of $X, X', Y, Y', Z, Z'$ having cardinality at most $\eps' n$. Now if $a \in X'$ then the elements $a + b$, $b \in H \setminus X$, are all distinct and none of them lie in $H$. If such an element $a + b$ lies in $C$, it must therefore lie in $Z'$. Thus if $a \in X'$ then there are at least $|H| - |X| - |Z'| \geq (1 - 4\eps') n > \frac{1}{2}n$ elements $b \in H \setminus X$ for which $a + b \notin C$. By assumption it follows that $\frac{1}{2} n|X'|  \leq Kn$, which implies that $|X'| \leq 2K$. Similarly $|Y'| \leq 2K$. 

Now note that, since $|X|, |Y| < \frac{1}{8}|H|$, every element of $H$ has at least $\frac{3}{4}|H| \geq \frac{1}{2}n$ representations as a sum $a + b$. Indeed if $h \in H$ then by we have $|(h - (H \setminus X)) \cap (H \setminus Y)| > \frac{3}{4}|H|$ by the pigeonhole principle. It follows that if we pass to a subset of these sums by removing all sums $a + b$ with $(a,b)$ lying in a set of size at most $Kn$, at least $|H| - 2K$ elements of $H$ are still represented.  By assumption, the set $C$ contains a set of this form, and it follows that $|Z| \leq 2K$. 

We have now demonstrated the inequalities
\[ |A|, |B| \leq |H| + 2K, \qquad |C| \geq |H| - 2K.\] Since the sizes of $A$, $B$ and $C$ differ by at most $K$, we must in fact have 
\[ |H| - 3K \leq |A|, |B|, |C| \leq |H| + 3K.\]
This allows us to conclude that $|X|, |Y|, |Z'| \leq 5K$. Proposition \ref{almost-group} follows immediately.\endproof

We turn now to the task of proving Proposition \ref{almost-group-weak}. We require the following result, which could be deduced from results of Kneser \cite{kn53, kn55} and Freiman \cite{frei}.

\begin{lemma} \label{lemmaa5} Let $\eps < \frac{1}{60}$. Suppose that $A$ is a subset of an abelian group $G$ with $|A| = n$, and suppose that $|A - A| \leq (1 + \eps)n$. Then there is a subgroup $H \leq G$ and a coset $x + H$ such that $|A \triangle (x + H)| \leq 6\eps n$.
\end{lemma}

\begin{proof}
This is basically the argument of Fournier \cite{fournier}. Write (cf. \cite{tao-vu}) $\Sym_{\alpha}(A)$ for the set of all $d$ which have at least $\alpha n$ representations as $a_1 - a_2$, $a_1, a_2 \in A$. Note that $\Sym_{1 - \delta_1} (A) + \Sym_{1 - \delta_2}(A) \subset \Sym_{1 - \delta_1 - \delta_2}(A)$, and note also that $|\Sym_{5/6}(A)| \geq (1 - 5\eps) n \geq \frac{11}{12}n$. This follows from double-counting pairs $(a_1, a_2) \in A^2$: we have
\begin{align*} n^2 = |A|^2 & = \sum_{d \in A - A} | \{ (a_1, a_2) : a_1, a_2 \in A, a_1 - a_2 = d\}| \\ &  \leq |\Sym_{5/6}(A)||A| + \textstyle\frac{5}{6}|(A - A) \setminus \Sym_{5/6}(A)| |A| \\ & \leq \textstyle\frac{1}{6}|\Sym_{5/6}(A)|n + \frac{5}{6}(1 + \eps) n^2.\end{align*}

We claim that $H = \Sym_{2/3}(A)$ is a group. Certainly $0 \in H$, and $H + H \subset \Sym_{1/3}(A)$, so all we need do is check that $\Sym_{2/3}(A) = \Sym_{1/3}(A)$. Suppose that $d \in \Sym_{1/3}(A)$. Then $d = a_1 - a_2$ in at least $\frac{1}{3}n$ ways. If $t \in \Sym_{5/6}(A)$ then $t = a'_1 - a'_2$ in at least $\frac{5}{6}n$ ways. For at least $\frac{1}{6}n$ of these we will have $a'_1 = a_2$ for some $a_2$ such that $d = a_1 - a_2$, and thus $d + t = (a_1 - a_2) + (a'_1 - a'_2) = a_1 - a'_2 \in A - A$. That is, $|(d + \Sym_{5/6}(A) ) \cap (A - A)| \geq \frac{1}{6}n$. In particular, $d + \Sym_{5/6}(A)$ intersects $\Sym_{5/6}(A)$ (which has size at least $\frac{11}{12}n $) and therefore $d \in \Sym_{5/6}(A) - \Sym_{5/6}(A) \subset \Sym_{2/3}(A)$, as required.

To see that $A$ is close to a coset of $H$, note that $a_1 - a_2 \in H$ for all but at most $|(A - A) \setminus \Sym_{2/3}(A)| |A| \leq 6\eps n^2$ of the pairs $(a_1, a_2) \in A$. In particular there is some $x = a_2$ such that $a_1 \in x + H$ for all but at most $6\eps n$ values of $a_1 \in A$.
\end{proof}

We also need the following ``99\%'' version of the Balog-Szemer\'edi-Gowers theorem.

\begin{lemma}\label{99bsg}
Suppose that $A$ and $B$ are sets in some abelian group, and that $\Gamma \subset A \times B$ is a set with $|\Gamma| \geq (1 - \eps)|A||B|$. Suppose that $|A +_{\Gamma} B| \leq (1 + \eps)|A|^{1/2}|B|^{1/2}$. Then there are sets $A' \subset A$ and $B' \subset B$ with $|A'|/|A|, |B'|/|B| \geq 1 - \eps'$ such that $|A' - B'| \leq (1 + \eps') |A|^{1/2}|B|^{1/2}$.  We can take $\eps' = O(\eps^{c})$ for some $c > 0$.
\end{lemma}
\begin{proof}
This follows from \cite[Theorem 2.31]{tao-vu}. If one wanted instead the conclusion $|A' + B'| \leq (1 + \eps') |A|^{1/2}|B|^{1/2}$ (for comparison with Theorem \ref{bsg}) then one could additionally apply \cite[Proposition 2.27]{tao-vu}. 
\end{proof}

With these lemmas in hand, we can conclude the proof of Proposition \ref{almost-group-weak}. In what follows $\eps_1,\eps_2,\dots$ are all quantities bounded by $O(\eps^{O(1)})$. Explicit dependencies could be given if desired, but this would require Lemma \ref{99bsg} to be made explicit. With the hypotheses as in Proposition \ref{almost-group-weak}, first apply Lemma \ref{99bsg} to conclude that there are sets $A' , B' $ with $|A \triangle A'|, |B \triangle B'| \leq \eps_1n$ and $|A ' - B'| \leq (1 + \eps_2)n$. Applying the Ruzsa triangle inequality we obtain $|A' - A'|, |B' - B'| \leq (1 + \eps_3)n$. By Lemma \ref{lemmaa5} (assuming $\eps$ is sufficiently small) there are subgroups $H, H'$ and cosets $x + H, y + H'$ such that $|A \triangle (x + H)|, |B \triangle (y + H')| \leq \eps_4 n$. In particular $(1 - \eps_5)n \leq |H|, |H'| \leq (1 + \eps_5)n$.

Now, by assumption, for all except $\eps n^2$ pairs $(a,b) \in A \times B$, $a + b$ lies in a set $C$ of size at most $(1 + \eps)n$. It follows easily that for all except $\eps_6 n^2$ pairs $(h,h') \in H \times H'$, $h + h'$ lies in a set of size $(1 + \eps_7)n$. We claim that this forces $H = H'$. To see this, note that the assumption easily implies that there are at least $(1 - \eps_{8})n^3$ additive quadruples $h_1 - h_2 = h'_1 - h'_2$, and so for all but $\eps_{9} n^2$ pairs $(h_1,h_2) \in H$ we have $h_1 - h_2 = h'_1 - h'_2 \in H'$. This implies that all but $\eps_{10} n$ elements of $H$ lie in $H'$, and hence $|H \triangle H'| \leq \eps_{11} n$ and so $|H \cap H'| \geq (1 - \eps_{12})n  > \frac{1}{2}n$, provided $\eps$ is sufficiently small. Invoking Lagrange's theorem (the order of a subgroup divides the order of the group), it follows that in fact $H = H \cap H' = H'$, as claimed.

Finally, note that since $A$ occupies at least $7|H|/8$ of $x + H$, and $B$ at least $7|H|/8$ of $y + H$, every element of $x + y + H$ is a sum $a + b$ in at least  $3|H|/4 > n/2$ ways. It follows that $C$ must contain all but at most $\eps_{13} n$ of the elements of $x + y + H$, and this concludes the proof.\endproof

The next result is due to Elekes, Nathanson and Ruzsa \cite{elekes-nathanson-ruzsa}. 

\begin{proposition}\label{enr-prop-piecewise} Let $A \subset \R$ be a set of cardinality $n$, and suppose that there are $x_1 < \dots < x_{10}$ such that $f : \R \rightarrow \R$ is defined except possibly at $x_1,\dots, x_{10}$ and is strictly concave or convex on each open interval $(x_i, x_{i+1})$. Then either $|A - A|$ or $|f(A) - f(A)|$ has cardinality at least $cn^{5/4}$ for some absolute constant $c>0$.\end{proposition}
\begin{proof}
Suppose first of all that $f$ is strictly convex or concave on an interval containing $A$. Then by \cite[Corollary 3.1]{elekes-nathanson-ruzsa} we have $|A - A| |f(A) - f(A)| \gg n^{5/2}$, and so either $|A - A|$ or $|f(A) - f(A)|$ has cardinality at least $cn^{5/4}$. The proposition follows by applying this to the largest of the sets $A_i := A \cap (x_i, x_{i+1})$. 
\end{proof}

\section{Intersections of lines through roots of unity}\label{chord-app}

In this section we establish Proposition \ref{poonen-rubinstein-prop}.  Our arguments will be a crude variant of those used in \cite{poonen-rubinstein}.  

It will be convenient to identify the plane $\R^2$ with the complex numbers $\C$.  Let $\Pi_n$ be the $n^{\operatorname{th}}$ roots of unity, and suppose that $p$ is a point other than the origin or an element of $\Pi_n$ which is incident to $m$ lines $\ell_1,\ldots,\ell_m$, each of which pass through two points $e^{2\pi i\beta_j}, e^{2\pi i\gamma_j}$ of $\Pi_n$, where $0 \leq \beta_1,\ldots,\beta_m,\gamma_1,\ldots,\gamma_m < 1$ are distinct rationals with denominator $n$.  Our objective is to show that $m=O(1)$.

We claim the identity
$$ \frac{|p - e^{2\pi i\beta_j}|}{|p - e^{2\pi i \beta_k}|} = \frac{|e^{2\pi i \beta_j} - e^{2\pi i \gamma_k}|}{|e^{2\pi i \beta_k} - e^{2\pi i \gamma_j}|}$$
for any distinct $1 \leq j,k \leq m$.  Indeed, from elementary trigonometry we see that $( p, e^{2\pi i \beta_j}, e^{2\pi i \gamma_k} )$ and
$( p, e^{2\pi i \beta_k}, e^{2\pi i \gamma_j} )$ form a pair of similar triangles, regardless of the relative ordering between the points involved.   The right-hand side can be simplified as
$$ \frac{|\sin( \pi (\beta_j - \gamma_k) )|}{|\sin(\pi(\beta_k - \gamma_j))|}.$$
We conclude that
$$ \frac{|\sin( \pi (\beta_j - \gamma_k) )|}{|\sin(\pi(\beta_k - \gamma_j))|}
\frac{|\sin( \pi (\beta_k - \gamma_l) )|}{|\sin(\pi(\beta_l - \gamma_k))|} \frac{|\sin( \pi (\beta_l - \gamma_j) )|}{|\sin(\pi(\beta_j - \gamma_l))|} = 1$$
for any distinct $1 \leq i, j,k \leq m$, and thus
\begin{align*} \sin( \pi (\beta_j - \gamma_k) ) & \sin( \pi (\beta_k - \gamma_l) ) \sin( \pi (\beta_l - \gamma_j) ) \\ & = \pm
\sin( \pi (\beta_k - \gamma_j) ) \sin( \pi (\beta_l - \gamma_k) ) \sin( \pi (\beta_j - \gamma_l) )\end{align*}
for some choice of sign $\pm$.  Actually, we claim that the sign here is always given by the $+$ sign.  To see this, let us temporarily forget that the $e^{2\pi i\beta_j}, e^{2\pi i\gamma_k}$ were constrained to be roots of unity, and that $p$ was assumed not to take values at the origin or at infinity (since we have not yet actually used these hypotheses).  We first observe that the sign does not change if we shift any of the $\beta_j$ or $\gamma_k$ by an integer, so we may assume that these phases take values in $\R/\Z$ rather than $[0,1)$.  Then we observe that the sign is stable with respect to continuous perturbations of the $\beta_j,\gamma_k$ and $p$, so long as no two phases cross each other, and that $e^{2\pi i\beta_j}, e^{2\pi i \gamma_j}, p$ remain collinear for all $j$.  From this we may reduce to the case when $p$ is at the origin (so that $\beta_j =\gamma_j+1/2$ for all $j$) or at, say, $[1,0,0]$ (so that $\beta_j = 1/2 - \gamma_j$ for all $j$), and the sign is easily verified in these cases.

Expanding out $\sin x$ as $(e^{ix} - e^{-ix})/2i$, we conclude that
\begin{align*} (& e^{\pi i (\beta_j-\gamma_k)} -  e^{\pi i (\gamma_k - \beta_j)})
(e^{\pi i (\beta_k-\gamma_l)} - e^{\pi i (\gamma_l - \beta_k)}) (e^{\pi i (\beta_l-\gamma_j)} - e^{\pi i (\gamma_j - \beta_k)}) \\
& =
(e^{\pi i (\beta_k-\gamma_j)} - e^{\pi i (\gamma_j - \beta_k)})
(e^{\pi i (\beta_l-\gamma_k)} - e^{\pi i (\gamma_k - \beta_l)}) (e^{\pi i (\beta_j-\gamma_k)} - e^{\pi i (\gamma_k - \beta_j)}).\end{align*}
Multiplying out both sides, and cancelling the common terms 
\[ \pm e^{\pm \pi i(\beta_j+\beta_k+\beta_l-\gamma_j-\gamma_k-\gamma_l)}\] appearing on both sides, one arrives at an identity of the form
\begin{equation}\label{instance}
 \sum_{r=1}^{12} \epsilon_r e^{\pi i \alpha_{r; j,k,l}} = 0,
\end{equation}
where the $\epsilon_r = \pm 1$ are signs depending only on $r$, and $\alpha_{r;j,k,l}$ are twelve linear combinations of $\beta_j,\beta_k,\beta_l,\gamma_j,\gamma_k,\gamma_l$, each of the form
$$ \alpha_{r;j,k,l} = \pm \beta_j \pm \beta_k \pm \beta_l \pm \gamma_j \pm \gamma_k \pm \gamma_l$$
where the six signs $\pm$ do not need to be equal, but depend only on the index $r$.  These signs can of course be worked out explicitly, but we will not need to do so here, save to note that the linear forms $\alpha_{r;j,k,l}$ are all distinct in $r$, thus $\alpha_{r;j,k,l} - \alpha_{r';j,k,l}$ is a non-trivial combination of $\beta_j,\beta_k,\beta_l,\gamma_j,\gamma_k,\gamma_l$ whenever $r,r'$ are distinct.

Now we reinstate the hypothesis that the $e^{2\pi i \beta_j}, e^{2\pi i \gamma_k}$ are $n$th roots of unity, which ensures that the $\epsilon_r e^{\pi i \alpha_{r; j,k,l}}$ are also $n$th roots of unity.  The sets of twelve $n$th roots of unity that sum to zero were completely classified in \cite[Theorem 3]{poonen-rubinstein}.  The exact classification is somewhat messy, but we only require the following qualitative consequence of it. 

\begin{proposition} There exists a finite set $S$ of roots of unity with the property that whenever $e^{2\pi i \alpha_1},\ldots,e^{2\pi i \alpha_{12}}$ are roots of unity with $\sum_{r=1}^{12} e^{2\pi i \alpha_r} = 0$, then one has $e^{2\pi i (\alpha_r - \alpha_{r'})} \in S$ for some $1 \leq r < r' \leq 12$.
\end{proposition}

Indeed, one can take $S$ to be the ratios of the roots of unity arising in the minimal relations of weight up to $12$ that were classified in \cite[Theorem 3]{poonen-rubinstein}, the key point being that there were only finitely many ($107$, to be precise) such relations up to rotation.  In fact one can take $S$ to consist of $30030$th roots of unity if desired.

Applying this proposition, we conclude that for any distinct $i,j,k,l$ with $1 \leq i,j,k,l \leq m$, one has
$$ \pm e^{\pi i (\alpha_{r;j,k,l} - \alpha_{r';j,k,l})} \in S$$
for some $r,r'$ with $1 \leq r < r' \leq 12$ and some choice of sign $\pm$.  Applying the pigeonhole principle, we conclude (for $m$ large enough) that there exist $r,r'$ with $1 \leq r < r' \leq 12$ and a phase $\theta$ such that
$$ \alpha_{r;j,k,l} - \alpha_{r';j,k,l} = \theta$$ 
for $\gg m^3$ triples of distinct $1 \leq j,k,l \leq m$.

Fix $r,r',\theta$ as above.  As mentioned earlier, $\alpha_{r;j,k,l} - \alpha_{r';j,k,l}$ is a non-trivial linear form in the $\beta_j,\beta_k,\beta_l,\gamma_j,\gamma_k,\gamma_l$.  By symmetry, we may then assume that at least one of the $\beta_j, \gamma_j$ coefficients in this form are non-zero.  As these coefficients lie in $\{-2,0,+2\}$, we may thus write
$$ \alpha_{r;j,k,l} - \alpha_{r';j,k,l} = 2a \beta_j + 2b \gamma_j + c_{k,l}$$
for some coefficients $a,b \in \{-1,0,+1\}$ not both zero, and some phases $c_{k,l}$ independent of $j$.   Note that these coefficients are equal to $-2$, $0$, or $+2$.  Pigeonholing in the $k,l$, we may then find distinct $1 \leq k,l\leq m$ and a phase $\theta'$ such that
$$ a \beta_j + b \gamma_j = \theta'$$
for $\gg m$ values of $j$.  But from elementary trigonometry, and the hypothesis that $p$ is not at the origin (or at infinity) we see that as $\beta,\gamma \in [0,1)$ range over the distinct phases for which
$e^{2\pi i \beta}, e^{2\pi i \gamma}$ are concurrent with $p$, the phase $a\beta+b\gamma$ can take on any specific value $\theta'$ at most $O(1)$ times, and so $m = O(1)$ as desired.  This concludes the proof of Proposition \ref{poonen-rubinstein-prop}.\endproof\vspace{11pt}

 \emph{Remark.} In principle, an explicit computational analysis of the minimal relations that were classified in \cite[Theorem 3]{poonen-rubinstein} should yield the optimal value of $C$ in Proposition \ref{poonen-rubinstein-prop}.  In \cite{poonen-rubinstein} it was shown that one can take $C=7$ if one restricts $p$ to be in the interior of the circle, and it is likely that the same bound holds in the exterior region also.  However, we will not perform this computation here.\vspace{11pt}

We now give a weaker version of Proposition \ref{poonen-rubinstein-prop} which is completely elementary. In particular, it avoids the Poonen-Rubinstein classification of tuples of twelve roots of unity summing to zero.  It can be used as a substitute for that proposition in the proof of Theorem \ref{mainthm} (and hence Theorem \ref{dirac-motzkin-1}), but not in the stronger Theorem \ref{dirac-motzkin-2}.

\begin{proposition}\label{chords-prop}
Let $\Pi_n \subset \C \equiv \R^2$ denote the regular $n$-gon consisting of the $n^{\operatorname{th}}$ roots of unity. Then no point other than the origin lies on more than $O(n^{5/6})$ lines joining pairs of vertices of $\Pi_n$.
\end{proposition}

\begin{proof} The argument here will be similar to that used at the end of Section \ref{detailed-structure}, with the roots of unity $\Pi_n$ playing the role of the coset $H \oplus x$ in that analysis.

Let $p$ be a point other than the origin. Let the vertices of $\Pi_n$ be $v_1,\dots, v_n$ in order. Suppose that the line connecting $p$ and $v_j$ meets the line at infinity in the point $[-\sin \pi \theta_j, \cos \pi \theta_j,0]$.  As $j$ ranges between $1$ and $n$, the $\theta_j$ can be taken to be an increasing sequence in $[0, 2]$. 

A Euclidean geometry exercise, left as an exercise to the reader, confirms the following claim: for any fixed integer $a$ and for any $\phi$ there are at most two values of $j$ such that $\theta_{j + a} - \theta_j = \phi$.  

Suppose that $p$ lies on $\delta n$ lines joining pairs of vertices of $\Pi_n$.  If $v_j$ is one such vertex, then from elementary trigonometry we see that $n \theta_j$ is an integer.  Thus  there is a set $J$, $|J| = \delta n$, such that all $n \theta_j$ are integers in $\{1,\dots, 2n\}$ for $j \in J$.  Split $\{1,\dots, 2n\}$ into $m \sim \frac{1}{5}\delta n$ intervals of length $\sim 10/\delta$, and suppose that the number of points of $J$ in these intervals is $N_1, \dots, N_m$. Since $N_1 + \dots + N_m = \delta n$, the Cauchy-Schwarz inequality implies that $N_1^2 + \dots + N_m^2 \geq \delta^2 n^2/m \geq 5\delta n$. On the other hand this sum is at most the number of pairs in $J \times J$ differing by at most $10/\delta$. The contribution from the diagonal (pairs $(j,j)$) is just $\delta n$, and so there are at least $4\delta n$ pairs in $J \times J$ differing by at most $10/\delta$. By the pigeonhole principle there is some integer $a$, $0 < a \leq 10/\delta$,  such that $j, j + a \in J$ for $\gg \delta^2 n$ values of $j$. From this sequences of $j$s, we may then extract a subsequence $j_1 < \dots < j_d$, $d \gg \delta^3 n$, with $j_{i+1} > j_i + a$, such that once again $j_i, j_i + a \in J$ for each $i$. 

Write $x_i := n( \theta_{j_{i} + a} - \theta_{j_i})$. Then, since the $\theta_j$ are increasing as a function of $j$, all the $x_i$ are positive. Furthermore we have $x_1 + \dots + x_d \leq n$. However the $x_i$ are all integers, and no integer can occur more than twice as a value of $x_i$ by the claim we established at the start of the proof. From this it follows that $d \ll \sqrt{n}$.

Comparing these inequalities yields $\delta \ll n^{-1/6}$, and this completes the proof.
\end{proof}

\setcounter{tocdepth}{1}

\end{document}